\documentclass[journal,onecolumn,12pt]{IEEEtran}
\date{}

\usepackage{amsfonts}
\usepackage{amsmath}
\usepackage{amssymb}
\usepackage{amsthm}
\usepackage{array}
\usepackage{caption}
\usepackage{cite}
\usepackage{citesort}
\usepackage{color}
\usepackage{epsfig}
\usepackage{enumitem} 
\usepackage{flexisym}
\usepackage{graphicx}
\usepackage{hyphenat}
\usepackage{ifpdf}
\usepackage{latexsym}
\usepackage{mathtools}
\usepackage{multirow}
\usepackage{psfrag}
\usepackage{setspace}
\usepackage{subfigure}
\usepackage{tikz}
\usepackage{times}
\usepackage{tkz-berge}
\usepackage{xparse}

\newtheorem{theorem}{Theorem}
\newtheorem{definition}{Definition}

\newtheorem{lemma}{Lemma}
\newtheorem{corollary}{Corollary}

\newtheorem{remark}{Remark}
\newtheorem{example}{Example}

\title{Deterministic and Probabilistic Conditions for Finite Completability of Low-Tucker-Rank Tensor}

\author{Morteza Ashraphijuo, Vaneet Aggarwal,~\IEEEmembership{Senior Member,~IEEE}, and Xiaodong Wang,~\IEEEmembership{Fellow,~IEEE}\thanks{\noindent  Morteza Ashraphijuo and Xiaodong Wang are with the Department of Electrical Engineering, Columbia University, NY, email: \{ashraphijuo,wangx\}@ee.columbia.edu. Vaneet Aggarwal is with the School of IE, Purdue University, West Lafayette, IN, email: vaneet@purdue.edu. }}
\begin{document}
\maketitle

%{\bf  Morteza}

\begin{abstract}

We investigate the fundamental conditions on the sampling pattern, i.e., locations of the sampled entries, for finite completability of a low-rank tensor given some components of its Tucker rank. In order to find the deterministic necessary and sufficient conditions, we propose an algebraic geometric analysis on the Tucker manifold, which allows us to incorporate multiple rank components in the proposed analysis in contrast with the conventional geometric approaches on the Grassmannian manifold. This analysis characterizes the algebraic independence of a set of polynomials defined based on the sampling pattern, which is closely related to finite completability of the sampled tensor,  where finite completability simply means that the number of possible completions of the sampled tensor is finite. Probabilistic conditions are then studied and a lower bound on the sampling probability is given, which guarantees that the proposed deterministic conditions on the sampling patterns for finite completability hold with high probability. Furthermore, using the proposed geometric approach for finite completability, we propose a sufficient condition on the sampling pattern that ensures there exists exactly one completion of the sampled tensor. %Finally, a converse problem is studied where the rank of the sampled tensor is unknown but a tensor is given such that agrees with the sampled entries. Using finite completability result, we show under some  assumptions that the rank of the unknown sampled tensor is equal to the rank of the given tensor with probability one. We show that the provided  necessary and sufficient conditions are milder than the sufficient conditions obtained by matricization analysis on the Grassmannian manifold.

\

\begin{IEEEkeywords}
Low-rank tensor completion, finite completion, unique completion, Grassmannian manifold, Tucker manifold, extended Hall's theorem.
\end{IEEEkeywords}

\end{abstract}

\newpage
\section{Introduction}

%Nowadays, high-dimensional data analysis has received a remarkable amount of attention due to the world-wide applications of images and videos, product ranking datasets, gene expression database and other commonly used big datasets. Hyperspectral images and videos generate datasets with millions of dimensions. Curse of dimensionality suggests that in such big datasets, a minor increment in dimensionality in the datasets entails a significant increment in the amount of the data, and this fact causes computational challenges to analyze big high-dimensional datasets. Hence, proposing a statistically rigorous result on the recovery of such big datasets requires a massive amount of data that grows exponentially with the dimension. However, as the dimensionality increases, the massive increment in the amount of data causes sparsity, and in fact many real-world high-dimensional datasets can be expressed as a much lower dimensional structure, e.g., a low-rank tensor \cite{donoho-high,donoho-high2}. Efficiently using this low-rankness structure for analyzing large high-dimensional datasets is one of the most vital challenges of research in machine learning and data mining. In this paper, we propose a new geometric approach for tackling the problem of recovery of big datasets that have the above low-rank structure.

\IEEEPARstart{T}{ensors} are generalizations of vectors and matrices:  a vector is a first-order tensor and a matrix is a second-order tensor. Most data around us are better represented with multiple dimensions to capture the correlations across different attributes.  For example, a colored image can be considered as a third-order tensor, two of the dimensions (rows and columns) being spatial, and the third being spectral (color); while a colored video sequence can be considered as a fourth-order tensor, with time being the fourth dimension besides spatial and spectral dimensions. Similarly, a colored 3-D MRI image across time can be considered as a fifth-order tensor. In many applications, part of the data may be missing. This paper investigates the fundamental  conditions on the locations of the non-missing entries such that the multi-dimensional data can be recovered in finite and/or unique choices. In particular, we investigate deterministic and probabilistic conditions on the sampling pattern for finite or unique solution to a low-rank tensor completion problem given the sampled tensor and some of its Tucker rank components, i.e., ranks of some of its matricizations.

There are numerous applications of low-rank data completion in various areas including image or signal processing \cite{phase,Image}, data mining \cite{data}, network coding \cite{network}, compressed sensing \cite{lim,sid,gandy}, reconstructing the visual data \cite{visual,liu2016low}, seismic data processing \cite{kreimer,ely20135d,wang2016tensor}, RF fingerprinting \cite{liu2016tensor,7347424}, and reconstruction of cellular data \cite{vaneetcnsm}.

The majority of the literature on matrix and tensor completion are concerned with developing various optimization-based algorithms under some assumptions such as incoherence \cite{jain2013low}, etc., to construct a completion. In particular, low-rank matrix completion has been widely studied and many algorithms based on convex relaxation of rank \cite{candes,candes2,cai}, non-convex optimization \cite{ashraph19low} and alternating minimization \cite{jain2013low}, etc., have been proposed. Also, a generalization of the low-rank matrix completion, which is completion from several low-rank sources has attracted attention recently \cite{hav,parsons,pimentelgroup,pimentel2014sample}. For the tensor completion problem various solutions have been proposed that are based on convex relaxation of rank constraints \cite{kreimer,gandy,tomioka,nuctensor,romera,ashraphijuoc}, alternating minimization \cite{wang2016tensor,liulow2,liu2016low} and other heuristics \cite{7347424,low,low2,goldfarb}.
%Moreover, when the general linear combinations of the data can be measured, the conditions on the number of measurements have been considered in \cite{li2016optimal,aggarwal2015information}, and for matrix completion have been studied in \cite{suh2014information}.

In the existing literature on optimization-based matrix or tensor completions, in addition to meeting the lower bound on the sampling probability, conditions such as incoherence  \cite{jain2013low,goldfarb,relern}, which constrains the values of the matrix or tensor entries, are required to obtain a completion with high probability. On the other hand, fundamental completability conditions that are independent of the specific completion algorithms have also been investigated. In \cite{charact,ashraphijuo3,kiraly2,ashrphrization} deterministic conditions on the locations of the sampled entries (sampling pattern) have been studied through algebraic geometry approaches on Grassmannian manifold that lead to finite and unique solutions to the matrix completion problem, where finite completability simply means that the number of possible completions of the sampled tensor is finite. Specifically, in \cite{charact} a deterministic sampling pattern is proposed that is necessary and sufficient for finite completability of the sampled matrix of the given rank. Such an algorithm-independent condition can lead to a much lower sampling rate than that is required by the optimization-based completion algorithms. For example, the required number of samples per column in \cite{jain2013low} is on the order of $\mathcal{O}(\log(n)r^{2.5}\log(\|\mathbf{X}\|_{F}))$, where $\mathbf{X}$ is the unknown matrix with $n$ rows and of rank $r$, while the required number of samples per column in \cite{charact} is on the order of $\mathcal{O}(\max\{\log(n),r\})$. The analysis on Grassmannian manifold in \cite{charact} is not capable of incorporating more than one rank constraint, and therefore this method is not efficient for solving the same problem for a tensor given multiple rank components. In this paper, we propose a geometric analysis on Tucker manifold to obtain deterministic and probabilistic conditions that lead to finite or unique completability for low-rank tensors when multiple rank components are given. Moreover, other related problems have been studied using algebraic geometry analysis, including high-rank matrix completion \cite{bahigh}, rank estimation \cite{ashraphijuo5}, and subspace clustering with missing data  \cite{pianmmple,ashraphijuoustering,pimentelc4,yarse,ashrximation,pimentelc3}. %Then, we also give a lower bound on the sampling probability that ensures the finite completability with high probability.

This work is inspired by \cite{charact}, where the analysis on Grassmannian manifold is proposed for a single-view matrix. Specifically, in \cite{charact}  a novel approach is proposed to consider the rank factorization of a matrix and to treat each observed entry as a polynomial in terms of the entries of the components of the rank factorization. Then, under the genericity assumption, the algebraic independence among the mentioned polynomials is studied. In this paper, we consider the low-Tucker-rank tensor and follow the general approach that is similar to that in \cite{charact}. We mention some of the main differences: (i) geometry of the manifold, (ii) the equivalence class for the core tensor and consequently (iii) the canonical core tensor, (iv) structure of the polynomials, etc. are fundamentally different from those in \cite{charact}. Moreover, (v) the idea of using more than one rank constraint simultaneously in the algebraic geometry approach is also new. Hence, the manifold structure for the low-Tucker-rank tensor is fundamentally different from the Grassmannian manifold and we need to develop almost every step anew.
%{\it \color{blue} where finite completability simply means the number of possible completions of the sampled tensor is a finite number.} 
%(also known as higher-order singular value decomposition when the factor matrices represent all the singular vectors of the corresponding matricization of the tensor)

Tucker decomposition is a well-known method to represent a tensor \cite{Tuck,SVD,Tuckermanifold}.  In this paper, we use this decomposition to represent the sparsity of a tensor and use Tucker rank to model the low-rank structure of the tensor. There are several other well-known decompositions of a tensor as well, including polyadic decomposition \cite{ten,kruskal}, tensor-train decomposition \cite{oseledets,holtz}, hierarchical Tucker representation \cite{backll,hack9new}, tubal rank decomposition \cite{kilmer2013third} and others. 

%To emphasize the fact that proposing an algebraic geometry analysis on tensors is not a simple generalization of the existing analysis on Grassmannian manifold, we point out an example. In general, dealing with the tensor completion problem is more complex as compared to the matrix completion problem since some of the properties of matrix do not hold for higher dimensions (tensor). For example, to find the lowest rank matrix that agrees with the observed entries, Cand\`es {\it et al.} in \cite{candes} propose to minimize nuclear norm instead of the matrix rank as the objective function. The reason behind this idea is that nuclear norm is the convex envelope of rank that is the tightest convex relaxation. In \cite{nuctensor}, the authors propose the trace norm of tensor as a convex relaxation of the rank of tensor. However, in contrast with the matrix case, the proposed trace norm of tensor is not the convex envelope of rank and a tighter convex relaxation of tensor rank is given in \cite{romera} for some cases.

This paper focuses on the low-rank tensor completion problem, given a portion of the rank vector of the tensor. Specifically, we investigate the following three problems:

\begin{itemize}
	\item {\bf Problem (i):} Characterizing the necessary and sufficient conditions on the sampling pattern to have finitely many tensor completions for the given rank.
	
%Let $\text{Gr}(n,r)$ denote the Grassmannian of $r$-dimensional subspaces of an $n$-dimensional vector space and let $\text{Tu}((n_1,r_1),\dots,(n_d,r_d))$ denote the $d$-dimensional Tucker manifold where for each dimension $i \in \{1,\dots,d\}$ it consists of $r_i$-dimensional subspaces of an $n_i$-dimensional vector space (generalization of Grassmannian manifold). A consequence of the contributions in this paper is that we analyze the algebraic variety on Tucker manifold to be able to incorporate multiple rank components in our analysis. As a result of this approach, we obtain the necessary and sufficient condition on the sampling patterns for  finite completability of the sampled tensor given multiple rank components vector (any proper subset of the entire rank vector).  Even though both of the analysis (over Tucker or Grassmannian) result in a sufficient condition on the sampling patterns for finite completability of the sampled tensor given all rank components, we show that the condition on the sampling patterns obtained through our proposed analysis on Tucker manifold requires much less amount of samples to be hold, in general.

To solve this fundamental problem, we propose a geometric analysis framework on Tucker manifold. Specifically, we obtain a set of polynomials based on the location of the sampled entries in tensor and use Bernstein's theorem \cite{Bernstein,schaeffer1941inequalities} to identify the condition on the sampling pattern for ensuring sufficient number of algebraically independent polynomials in the mentioned set. Given any nonempty proper subset of the Tucker rank vector, this analysis leads to the necessary and sufficient condition on the sampling patterns for finite completability of the tensor. Given the entire Tucker rank vector this condition is sufficient for finite completability.

	\item {\bf Problem (ii):} Characterizing conditions on the sampling pattern to ensure that there is exactly one completion for the given rank.
	
We use our proposed geometric analysis for finite completability of low-rank tensors to obtain a sufficient conditions on the sampling patterns to ensure unique completability, which is milder than the sufficient condition for unique completability obtained through matricization analysis and applying the matrix method in \cite{charact}.

	\item {\bf Problem (iii):} If the elements in the tensor are sampled independently with probability $p$, what are the conditions on $p$ such that the conditions in Problems (i) and (ii) are satisfied with high probability? 
	
 We bound the number of needed samples to ensure the proposed sampling patterns for finite and unique tensor completability hold with high probability. Even though we follow a similar approach to \cite{charact} for the matrix case, we develop a generalization of Hall's theorem for bipartite graphs which is needed to prove the correctness of the bounds for both the tensor and the matrix cases. Moreover, it is seen that our proposed analysis on Tucker manifold leads to a much lower sampling rate than the corresponding analysis on Grassmannian manifold for both finite and unique tensor completions. 

%In Theorem \ref{thm11} which applies the analysis in \cite{charact} (analysis on Grassmannian manifold), the number of needed samples per column of the $i$-th matricization to ensure finite completability with high probability is on the order of $\mathcal{O}(\max\{\log(n_i),r_i\})$, and therefore the number of needed samples in total is on the order of $\mathcal{O}(\frac{n_1\dots n_d}{n_i} \max\{\log(n_i),r_i\})$. According to Theorem \ref{thm6}, as the tensor is at least of order three, this number reduces to order of $\mathcal{O}((n_{i+1}\dots n_d) \max\{\log(n_1\dots n_i),\log(r_{i+1}\dots r_d)\})$ samples in total if $r_{i+1}\dots r_d < \frac{n_{i+1}\dots n_d}{n_1\dots n_i}$.

%	\item {\bf Problem (iv):} Given a component-wise lower bound on the rank vector of the unknown sampled tensor and assuming that there exists a completion where its rank is equal to the lower bound vector, characterizing conditions on $\Omega$ to ensure $\mathcal{U}$ has the same rank as the lower bound, with probability  one.
\end{itemize}

The remainder of this paper is organized as follows. In Section \ref{notationsx}, some preliminaries and notations are presented, and also an example is given that illustrates the advantage of tensor analysis over analyzing matricizations of a tensor. In Section \ref{sec2}, Problem (i) is studied and the sampling patterns that ensure finite completions are found using tensor algebra. In Section \ref{corrcetion}, we study Problem (iii) for the case of finite completion  and the key to solving this problem is the proof of the generalized Hall's theorem, which is an independent result in graph theory.  Section \ref{sec:uniq} considers Problem (ii) to give a sufficient condition on the sampling pattern for unique completability. Further, Problem (iii) for unique completion is also studied. Some numerical results are provided in Section \ref{simulations} to compare the sampling rates for finite and unique completions based on our proposed tensor analysis versus the matricization method. Finally, Section \ref{conc} concludes the paper.

%In Section \ref{sec5}, our results on Problem (iv) are provided using the proposed approach on Tucker manifold in Section \ref{sec2}.

\section{Background}\label{notationsx}
\subsection{Preliminaries and Notations}\label{notations}

%{\it \color{blue} Define $\mathbb{P}_0$ as the Lebesgue measure on $\mathbb{R}^{r_1 \times r_2 \times \dots \times r_d }$ and $\mathbb{P}_i$ as the Lebesgue measure on $\mathbb{R}^{r_1 \times n_1 }$, $i = 1,\dots , d$. We assume that $\mathcal{U}$ is chosen generically from the manifold corresponding to rank vector $(r_1,\dots,r_d)$, or in other words, the entries of $\mathcal{U}$ are drawn independently with respect to Lebesgue measure on the corresponding manifold. Hence, any statement that holds for $\mathcal{U}$, it basically holds for almost every (with probability one) tensor of the same size and Tucker-rank with respect to the product measure $\mathbb{P}_0 \times \mathbb{P}_1 \times \dots \times \mathbb{P}_{d}$.}

In this paper, it is assumed that a $d^{\text{th}}$-order tensor $\mathcal{U} \in \mathbb{R}^{n_1  \times \cdots \times n_d}$ is chosen generically from the manifold of $n_1 \times \dots \times n_d$ tensors of the given Tucker rank (will be explained rigorously later). For the sake of simplicity in notation, define $N \triangleq \left( \Pi_{j=1}^{d} \  n_j \right)$ and $N_{-i} \triangleq \frac{N}{n_i}$. Also, for any real number $x$, define $x^+\triangleq\max\{0,x\}$. Let $\mathbf{U}_{(i)} \in \mathbb{R}^{n_i \times {N_{-i}}}$ be the $i$-th matricization of the tensor $\mathcal{U}$ such that $\mathcal{U}(\vec{x}) = {\mathbf{U}}_{(i)}(x_i,\mathcal{M}_{(i)} (x_1,\ldots,x_{i-1},x_{i+1},\ldots,x_d))$, where $\mathcal{M}_{(i)}$ is an arbitrary bijective mapping $\mathcal{M}_{(i)}: (x_1,\ldots,x_{i-1},x_{i+1},\ldots,x_d) \rightarrow \{1,2,\dots, N_{-i}\}$ and $\mathcal{U}(\vec{x})$ represents an entry of the tensor $\mathcal{U}$  with coordinate $\vec{x}=(x_1,\dots,x_d)$.

Given $\mathcal{U} \in \mathbb{R}^{n_1  \times \cdots \times n_d}$ and $\mathbf{X} \in \mathbb{R}^{n_i \times n_i^{\prime}}$, $\mathcal{U^{\prime}}  \triangleq  \mathcal{U} \times_i \mathbf{X} \in \mathbb{R}^{n_1 \times \cdots \times n_{i-1} \times n_i^{\prime} \times n_{i+1} \times \cdots \times n_d}$ is defined as
\begin{eqnarray}
\mathcal{U^{\prime}} (x_1,\cdots,x_{i-1},k_i,x_{i+1},\cdots,x_d) \triangleq \sum_{x_i=1}^{n_i} \mathcal{U} (x_1,\cdots,x_{i-1},x_i,x_{i+1},\cdots,x_d) \mathbf{X}(x_i,k_i).
\end{eqnarray}
Throughout this paper, we use Tucker rank as the rank of a tensor, which is defined as $\text{rank} (\mathcal{U})=(r_1,\ldots,r_d)$ where $r_i = \text{rank}(\mathbf{U}_{(i)})$. The Tucker decomposition of a tensor $\mathcal{U}$ is given by
\begin{eqnarray}\label{TuckTuck}
\mathcal{U} = \mathcal{C} \times_{i=1}^d \mathbf{T}_i,
\end{eqnarray}
where $\mathcal{C} \in \mathbb{R}^{r_1 \times \cdots \times r_d}$ is the core tensor and $\mathbf{T}_i \in \mathbb{R}^{r_i \times n_i}$ are $d$ orthogonal matrices. Then, \eqref{TuckTuck} can be written as
\begin{eqnarray}\label{tuckertu}
\mathcal{U}(\vec{x}) =  \sum_{k_1=1}^{r_1} \cdots \sum_{k_{d}=1}^{r_{d}} \mathcal{C}(k_1,\ldots,k_d) \mathbf{T}_1(k_1,x_1)  \dots \mathbf{T}_d(k_d,x_d).
\end{eqnarray}

The space of fixed Tucker-rank tensors is a manifold and the dimension of this manifold is shown in \cite{Tuckermanifold} to be 
%\begin{eqnarray}
$\sum_{i=1}^{d} \left(n_i \times r_i  - r_i^2\right)  + \Pi_{i=1}^{d} \ r_i$. Denote $\Omega$ as the binary sampling pattern tensor that is of the same size as $\mathcal{U}$ and $\Omega(\vec{x})=1$ if $\mathcal{U}(\vec{x})$ is observed and  $\Omega(\vec{x})=0$ otherwise. For each subtensor $\mathcal{U}^{\prime} $ of the tensor $\mathcal{U}$, define $N_{\Omega}({\mathcal{U}^{\prime}})$ as the number of observed entries in $\mathcal{U}^{\prime}$ according to the sampling pattern $\Omega$. %Moreover, for any tensor $\mathcal{U}$, define the sampled tensor $\mathcal{U}_{\Omega}$ as
%\begin{eqnarray}
%\mathcal{U}_{\Omega} (\vec{x}) = \left\{
%	\begin{array}{ll}
%		\mathcal{U} (\vec{x})  & \mbox{if } \Omega(\vec{x})=1, \\
%		0  & \mbox{if } \Omega(\vec{x})=0.
%	\end{array}
%\right.
%\end{eqnarray}

%\end{eqnarray}
%{\bf Morteza - please check. $\mathbf{T}_i$ orthogonal do not decrease $r_i^2$ entries but $r_i(r_i+1)/2$ entries.}

%Given the randomly sampled tensor $\mathcal{U}_{\Omega}$ and a subset of the rank vector components, based on the location of the sampled entries, i.e., the matrix $\Omega$, we are interested to verify if there are infinite, finite, or unique completions of the sampled tensor that satisfy the rank constraints. The authors of \cite{charact} considered this question for matrices. In this paper, we investigate this question for general order tensors.
%to Illustrate the Overperformance of Tensor Analysis Over Matricization Analysis

\subsection{Problem Statement and A Motivating Example}\label{example}

%if conditions on the sampling pattern tensor $\Omega$ under which

We are interested in finding deterministic  conditions on the sampling pattern tensor $\Omega$ under which there are infinite, finite, or unique completions of the sampled tensor $\mathcal{U}$ that satisfy $\text{rank}(\mathcal{U})=(r_1,r_2,\dots,r_d)$. Moreover, we are interested in finding probabilistic sampling strategies that ensure the obtained conditions for finite and unique completability hold, respectively, with high probability. The matrix version of this problem has been treated in \cite{charact}. In this paper, we investigate this problem for general order tensors.

%This depends on the sampling pattern tensor $\Omega$. 

In this subsection, we intend to compare the following two approaches in an example to emphasize the necessity of our analysis for general order tensors: (i) analyzing each matricization individually with the rank constraint of the corresponding matricization, (ii) analyzing via Tucker decomposition. In particular, we will show via an example that analyzing each of the matricizations separately is not enough to  guarantee finite completability when multiple rank components are given. On the other hand, we show that for the same example Tucker decomposition ensures finite completability. Hence, this example illustrates that matricization analysis does not take advantage of the full information of  given Tucker rank and thus fails to provide a necessary and sufficient condition for finite completability when more than one component of the rank vector is given. 

%Intuitively, the above-mentioned statement makes sense because while investigating one particular matricization, the rest of the information that is included in vector of rank is not used, i.e., we only consider one element of rank vector. 

%Given the sampled tensor $\mathcal{U}_{\Omega}$ and a subset of the rank vector components, we are interested in finding out if there are infinite, finite, or unique completions of the sampled tensor that satisfy the rank constraints.

Consider a $3^{\text{rd}}$-order tensor $\mathcal{U} \in \mathbb{R}^{2 \times 2 \times 2}$ with Tucker rank $(1,1,1)$. First, we show that having any $4$ entries of $\mathcal{U}$, there are infinitely many completions of any matricization with the corresponding rank constraint. Hence, along each dimension there exist a set of infinite completions given the corresponding rank constraint. Note that the analysis on Grassmannian manifold in \cite{charact} is not capable of incorporating more than one rank constraint. However, as we show it is possible that the intersection of the mentioned three infinite sets is a finite set and that is why we need an analysis that is able to incorporate more than one rank constraint. Without loss of generality, it suffices to show the claim only for its first matricization. Therefore, the claim reduces to the following statement: 

{\it {\bf Statement:} Having any $4$ entries of a rank-$1$ matrix $\mathbf{U} \in \mathbb{R}^{2 \times 4}$, there are infinitely many completions for it.}

In order to prove the above statement, we need to consider the following four possible scenarios:
\begin{enumerate}[label=(\roman*)]

\item The $4$ observed entries are in a row. In this case, clearly, there are infinitely many completions for the other row as it can be any scalar multiplied by the first row.

\item The $4$ observed entries are such that there is a column in which there is no observed entries. In this case, there are infinitely many completions for this column as it can be any scalar multiplied by the other columns.

\item The $4$ observed entries are such that there is one observed entry in each column, and also each row has exactly two observed entries. Assume that the two observed entries in the second row are the pair $(a,b)$. In this case, for every pair $(ka,kb)$ as the value of the two non-observed entries of the first row (where $k$ is an arbitrary scalar) there is a unique completion for the rest of the entries. As a result, there are infinitely many completions for this matrix.

\item The $4$ observed entries are such that there is one observed entry in each column, and also the first and second rows have $3$ and $1$ observed entries, respectively. In this case, for each value of the only non-observed entry of the first row there is a unique completion. Therefore, there are infinitely many completions for this matrix.

\end{enumerate}

Assume that the entries $\mathcal{U}(1,1,1)$, $\mathcal{U}(2,1,1)$, $\mathcal{U}(1,2,1)$, and $\mathcal{U}(1,1,2)$ are observed. Now, we take advantage of all elements of Tucker rank simultaneously, in order to show there are only finitely many tensor completions. Using Tucker decomposition \eqref{TuckTuck}, and given the rank is $(1,1,1)$, without loss of generality, assume that the scalar $\mathcal{C}=1$ and $\mathbf{T}_1=(x,x^{\prime})$, $\mathbf{T}_2=(y,y^{\prime})$ and $\mathbf{T}_3=(z,z^{\prime})$, and then the following equalities hold
\begin{align}
\mathcal{U}(1,1,1) &= xyz, & \mathcal{U}(2,2,1) &= x^{\prime}y^{\prime}z, \\ \nonumber
\mathcal{U}(2,1,1) &= x^{\prime}yz,  & \mathcal{U}(2,1,2) &= x^{\prime}yz^{\prime} , \\ \nonumber
\mathcal{U}(1,2,1) &= xy^{\prime}z, & \mathcal{U}(1,2,2) &= xy^{\prime}z^{\prime}, \\ \nonumber
\mathcal{U}(1,1,2) &= xyz^{\prime},  & \mathcal{U}(2,2,2) &= x^{\prime}y^{\prime}z^{\prime}. \nonumber
\end{align}
The unknown entries can be determined uniquely in terms of the $4$ observed entires as
\begin{eqnarray}
\mathcal{U}(2,2,1) &=& x^{\prime}y^{\prime}z = \frac{\mathcal{U}(2,1,1)\mathcal{U}(1,2,1)}{\mathcal{U}(1,1,1)}, \\ \nonumber
\mathcal{U}(2,1,2) &=& x^{\prime}yz^{\prime} = \frac{\mathcal{U}(2,1,1)\mathcal{U}(1,1,2)}{\mathcal{U}(1,1,1)}, \\ \nonumber
\mathcal{U}(1,2,2) &=& xy^{\prime}z^{\prime} = \frac{\mathcal{U}(1,2,1)\mathcal{U}(1,1,2)}{\mathcal{U}(1,1,1)}, \\ \nonumber
\mathcal{U}(2,2,2) &=& x^{\prime}y^{\prime}z^{\prime} = \frac{\mathcal{U}(2,1,1)\mathcal{U}(1,2,1)\mathcal{U}(1,1,2)}{\mathcal{U}(1,1,1)\mathcal{U}(1,1,1)}. \nonumber
\end{eqnarray}
Therefore, considering the Tucker decomposition, there is only one (finite) completion(s) having this particular $4$ observed entries as above. Note that only given $r_2=r_3=1$, it can be verified using Tucker decomposition similarly that the completion is still unique.

\section{Deterministic Conditions for Finite Completability}\label{sec2}

This section characterizes the connection between the sampling pattern and the number of solutions of a low-rank tensor completion. In Section \ref{indepbernbef}, we define a polynomial based on each observed entry. Then, for a given subset of the rank components we transform the problem of finite completability of $\mathcal{U}$ to the problem of finite completability of the core tensor in the Tucker decomposition of $\mathcal{U}$. In Section \ref{indepbern}, we propose a geometric analysis on Tucker manifold, by defining a structure for the core tensor of the Tucker decomposition such that we can determine if two core tensors span the same space. In Section \ref{consttensdef}, we construct a constraint tensor based on the sampling pattern $\Omega$. This tensor is useful for analyzing the algebraic independency of a subset of polynomials among all defined polynomials. In Section \ref{algebraind}, we show the relationship between the number of algebraically independent polynomials in the mentioned set of polynomials and finite completability of the sampled tensor. Finally, Section \ref{sec4} characterizes finite completability in terms of the sampling pattern instead of the algebraic variety for the defined set of polynomials. 

%\subsection*{Analysis on Grassmannian Manifold}\label{matrixangr}

%In the case that only one component of the rank is given, e.g., $r_{i_0}$, we consider the $i_0$-th matricization of the tensor. Observe that since we have only one rank constraint, this tensor completion problem is {\bf equivalent} to a matrix completion problem, where the matrix is the $i_0$-th matricization which is of rank $r_{i_0}$. Note that the geometric approach on Grassmannian manifold in \cite[Theorem 1]{charact} gives the necessary and sufficient condition on the sampling patterns for finite completability of this matrix completion problem, i.e., finite completability of this tensor completion problem when only one component of the rank is given (which is only a sufficient condition if more than one component of the rank is given). 

%Unfortunately, as we observed in Section \ref{example}, such analysis on Grassmannian manifold is not an effective approach to analyzing finite completability if we have more than one given rank constraint. Instead, we propose our geometric approach on Tucker manifold in the remainder of this section to characterize the necessary and sufficient condition on the sampling patterns for finite completability given multiple components of the rank vector.

\subsection{Condition for Finite Completability Given the Core Tensor}\label{indepbernbef}

Assume that the sampled tensor is $\mathcal{U}\in \mathbb{R}^{n_1 \times n_2 \times \cdots  \times n_d }$ and rank components $\{r_{j+1},\dots,r_d\}$ are given, where $j \in \{1,2,\dots,d-1\}$ is an arbitrary fixed number. Without loss of generality assume that $r_{j+1} \geq \dots \geq r_d$ throughout the paper.  Define $\mathbb{P}_0$ as the Lebesgue measure on $\mathbb{R}^{r_1 \times r_2 \times \dots \times r_d }$ and $\mathbb{P}_i$ as the Lebesgue measure on $\mathbb{R}^{r_1 \times n_1 }$, $i = 1,\dots , d$. We assume that $\mathcal{U}$ is chosen generically from the manifold corresponding to rank vector $(r_1,\dots,r_d)$, or in other words, the entries of $\mathcal{U}$ are drawn independently with respect to Lebesgue measure on the corresponding manifold. Hence, any statement that holds for $\mathcal{U}$, it basically holds for almost every (with probability one) tensor of the same size and Tucker-rank with respect to the product measure $\mathbb{P}_0 \times \mathbb{P}_1 \times \dots \times \mathbb{P}_{d}$.

Let the $d^{\text{th}}$-order tensor $\mathcal{C} \in \mathbb{R}^{n_1 \times n_2 \times \cdots \times n_j \times r_{j+1} \times r_{j+2} \times \cdots \times r_d }$  be a core tensor of the sampled tensor $\mathcal{U} \in \mathbb{R}^{n_1 \times n_2 \times \cdots \times n_d }$. Then, there exist full-rank matrices $\mathbf{T}_i$'s with $\mathbf{T}_i \in \mathbb{R}^{r_i \times n_i}$ such that
\begin{eqnarray}\label{nonlinear2}
\mathcal{U} = \mathcal{C} \times_{i=j+1}^d \mathbf{T}_i,
\end{eqnarray}
or equivalently
\begin{eqnarray}\label{tuckerj}
\mathcal{U}(\vec{x}) =  \sum_{k_{j+1}=1}^{r_{j+1}} \cdots \sum_{k_{d}=1}^{r_{d}} \mathcal{C}(x_1,\dots, x_j,k_{j+1},\ldots,k_d) \mathbf{T}_{j+1}(k_{j+1},x_{j+1})  \dots \mathbf{T}_d(k_d,x_d).
\end{eqnarray}

 %Equations \eqref{nonlinear2} and \eqref{nonlinear} can be written as their restricted format to the location of the corresponding sampled entries as (defined in notations)
%\begin{eqnarray}\label{sampeq2}
%\mathcal{U}_{\Omega} = (\mathcal{V} \times_{i=j+1}^d \mathbf{T}_i)_{\Omega},
%\end{eqnarray}
%and
%\begin{eqnarray}\label{sampeq}
%\mathcal{Y}_{\Omega} = (\mathcal{V} \times_{i=j+1}^d \theta_i)_{\Omega}.
%\end{eqnarray}
For notational simplicity, define $\mathbb{T}=(\mathbf{T}_{j+1},\dots,\mathbf{T}_{d})$. Figure \ref{fig0} represents a Tucker decomposition for a $3^{\text{rd}}$-order tensor given the second and third components of its rank vector.

 \begin{figure}[h]
	\centering
		{\includegraphics[width=8.5cm]{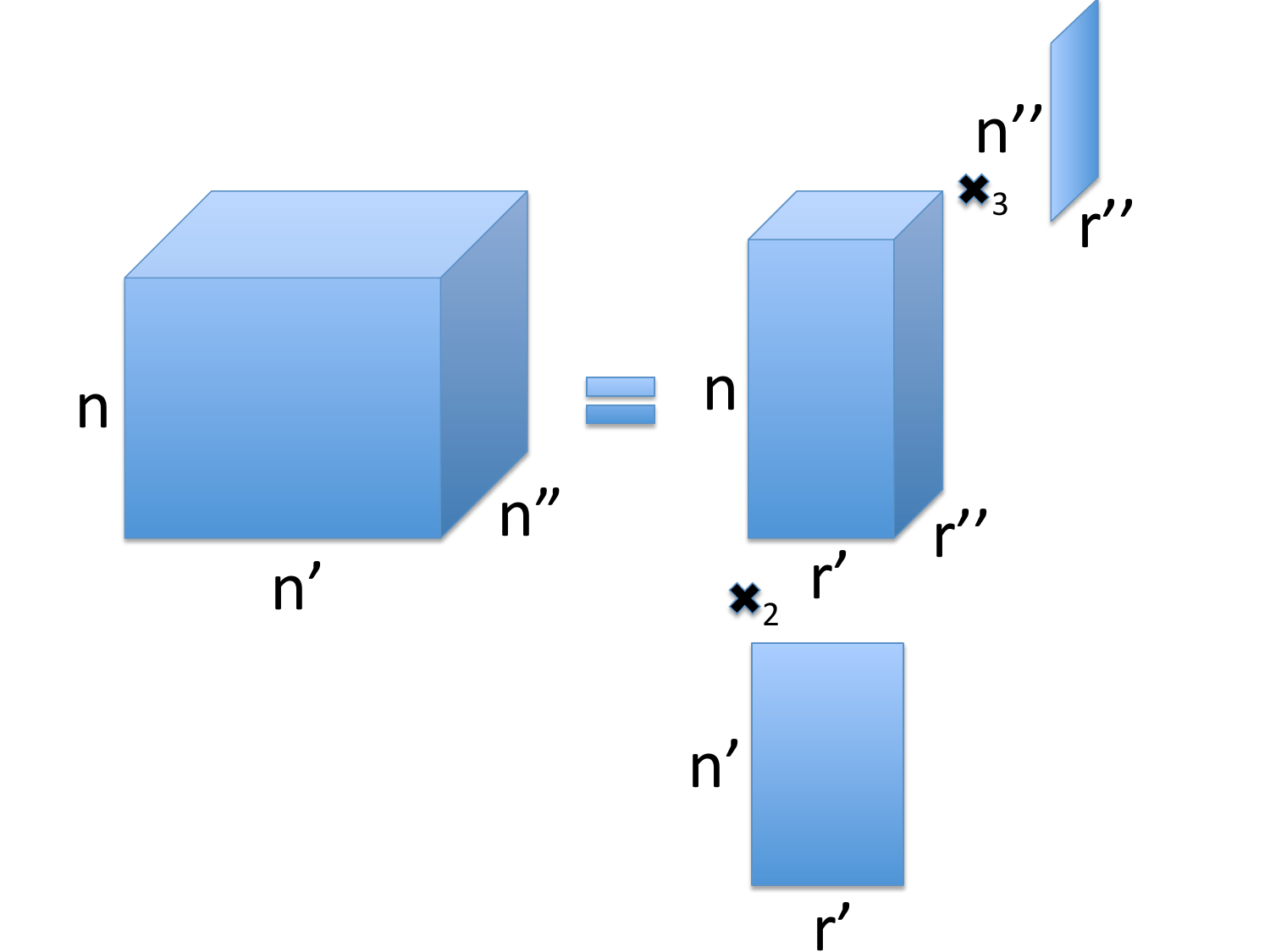}}
	\caption{  Tucker decomposition with $j=1$ and $d=3$.}
	\label{fig0}\vspace{-4mm}
\end{figure}
 
Here, we briefly mention some key points to highlight the fundamentals of our proposed analysis.
 
\begin{itemize}
\item {\bf Note $1$}: As it can be seen from  \eqref{tuckerj}, any observed entry $\mathcal{U}(\vec{x})$ results in an equation that involves $\Pi_{i=j+1}^{d} r_i$ entries of $\mathcal{C}$ and also $r_i$ entries of $\mathbf{T}_i$, $i=j+1,\dots,d$. Considering the entries of core tensor $\mathcal{C}$ and tuple $\mathbb{T}$ as variables (right-hand side of  \eqref{tuckerj}), each observed entry results in a polynomial in terms of these variables.

\item {\bf Note $2$}: For any observed entry $\mathcal{U}(\vec{x})$, the tuple $(x_1,\dots,x_j)$ specifies the coordinates of the $\Pi_{i=j+1}^{d} r_i$ entries of $\mathcal{C}$ that are involved in the corresponding polynomial.

\item {\bf Note $3$}: For any observed entry $\mathcal{U}(\vec{x})$, the value of $x_i$ specifies the column of the $r_i$ entries of $\mathbf{T}_i$ that are involved in the corresponding polynomial, $i=j+1,\dots,d$.

\item {\bf Note $4$}: Given all observed entries $\{\mathcal{U}(\vec{x}): \Omega(\vec{x}) = 1 \}$, we are interested in finding the number of possible solutions in terms of entries of $(\mathcal{C},\mathbb{T})$ (infinite, finite or unique) via investigating the algebraic independence among these polynomials.

\item {\bf Note $5$}: Note that it can be concluded from Bernstein's theorem  \cite{Bernstein} that in a system of $n$ polynomials in $n$ variables with each consisting of a given set of monomials, the $n$ polynomials are algebraically independent with probability one with respect to the corresponding probability measure, and therefore there exist only finitely many solutions. However, in the structure of the polynomials in our model, the set of involved monomials  are different for different set of polynomials, and therefore to ensure algebraically independency we need to have for any selected subset of the original $n$ polynomials, the number of involved variables should be more than the number of selected polynomials.
\end{itemize}
 
Given $\mathcal{C}$, we are interested to find a subset of the mentioned polynomials that guarantees tuple $\mathbb{T}$ can be determined finitely. The following definition will be used to determine the number of involved variables in a set of polynomials.
 
{\bf Definition 1}: For any $i \in \{j+1,\dots,d\}$ and nonempty $ \mathcal{S}_i \subseteq \{1,\dots,n_i\}$, define $\mathcal{U}^{(\mathcal{S}_{i})}$ as a set containing the locations of the entries of $|\mathcal{S}_i|$ rows (corresponding to the elements of $\mathcal{S}_i$) of $\mathbf{U}_{(i)}$. Moreover, define $\mathcal{U}^{(\mathcal{S}_{j+1},\dots,\mathcal{S}_{d})} = \mathcal{U}^{(\mathcal{S}_{j+1})}\cup \dots \cup \mathcal{U}^{(\mathcal{S}_{d})}$. Let $\tau$ be a subset of the locations of the entries of $\mathcal{U}$. Then, $\mathcal{U}^{(\mathcal{S}_{j+1},\dots,\mathcal{S}_{d})} $ is called the minimal hull of $\tau$ if any $\mathcal{U}^{(\mathcal{S}_{i})}$ includes exactly only those rows of $\mathbf{U}_{(i)}$ that include at least one of the locations of the entries in $\tau$.

\begin{example}
Consider a tensor $\mathcal{U} \in \mathbb{R}^{4 \times 4 \times 4}$ with $j=1$, $r_2 = 2$ and $r_3=3$. Then we have
\begin{eqnarray}\label{exequa}
\mathcal{U}(\vec{x}) =  \sum_{k_{2}=1}^{2}  \sum_{k_{3}=1}^{3} \mathcal{C}(x_1,k_{2}, k_3) \mathbf{T}_{2}(k_{2},x_{2})   \mathbf{T}_3(k_3,x_3).
\end{eqnarray}
where $\mathcal{C} \in \mathbb{R}^{4 \times 2 \times 3}$, $\mathbf{T}_2 \in \mathbb{R}^{2 \times 4}$ and $\mathbf{T}_3 \in \mathbb{R}^{3 \times 4}$. Define $\tau = \{(1,2,2),(2,2,3),(3,4,2)\}$. Hence, Among the four rows of the second matricization of $\mathcal{U}$, i.e. $\mathbf{U}_{(2)}$, row numbers $2$ and $4$ include at least one entry in $\tau$ (and row numbers $2$ and $3$ for $\mathbf{U}_{(3)}$). Then, for $\mathcal{S}_2=\{2,4\}$ $\mathcal{S}_3=\{2,3\}$, $\mathcal{U}^{(\mathcal{S}_2,\mathcal{S}_3)}$ is the minimal hull of $\tau$. Note that due to the definition, the minimal hull is unique.
\end{example}

\begin{remark}\label{minhull}
{\rm Consider any set of polynomials $\{p_1,\dots,p_k\}$ in form of \eqref{tuckerj}. Given the core tensor $\mathcal{C}$ these polynomials are in terms of entries of $\mathbb{T}$ and let $\tau$ be the set of corresponding entries to these polynomials in $\mathcal{U}$ and $\mathcal{U}^{(\mathcal{S}_{j+1},\dots,\mathcal{S}_{d})}$ be the minimal hull of $\tau$. Let $\mathcal{S}$ denote the set of all variables (entries of $\mathbb{T}$) that are involved in at least one of the polynomials $\{p_1,\dots,p_k\}$. Recall that according to Note $3$, if an entry of $\mathbf{T}_i$ is involved in a polynomial, all entries of the column that includes that entry are also involved in that polynomial. Therefore, $|\mathcal{S}|=\sum_{i=j+1}^{d} |\mathcal{S}_i| r_i$. } {  \hfill  \qedsymbol}
\end{remark}

Note that each observed entry results in a scalar equation of the form of \eqref{tuckerj}. Given the core tensor $\mathcal{C}$, we need at least $\sum_{i=j+1}^{d} \left( n_ir_i \right)$ polynomials to ensure the number of possible tuples $\mathbb{T}$ is not infinite since the number of variables (entries of $\mathbb{T}$) is $\sum_{i=j+1}^{d} \left( n_ir_i \right)$ in total. On the other hand, the $\sum_{i=j+1}^{d} \left( n_ir_i \right)$ mentioned polynomials should be algebraically independent to ensure the finiteness of tuples $\mathbb{T}$ since any algebraically independent polynomial reduces the dimension of the set of solutions by one. To ensure this independency, any subset of $t$ polynomials of the set of polynomials corresponding to the $\sum_{i=j+1}^{d} \left( n_ir_i \right)$ observed entries, should involve at least $t$ variables. The following assumption will be used frequently as we show it satisfies the mentioned property.

\noindent{\underline{\it{Assumption $A_j$}}}: Anywhere that this assumption is stated, there exist $\sum_{i=j+1}^{d} \left( n_ir_i \right)$ observed entries such that for any nonempty $ \mathcal{S}_i \subseteq \{1,\dots,n_i\}$ for $i \in \{j+1,\dots,d\}$, $\mathcal{U}^{(\mathcal{S}_{j+1},\dots,\mathcal{S}_{d})} $ includes at most $\sum_{i=j+1}^{d} |\mathcal{S}_i|r_i$ of the mentioned $\sum_{i=j+1}^{d} \left( n_ir_i \right)$ observed entries.

\begin{example}
Consider a tensor $\mathcal{U} \in \mathbb{R}^{3 \times 3 \times 3}$ with $j=1$, $r_2 = 1$ and $r_3=2$. Define the following two sets of observed entries each including $\sum_{i=j+1}^{d} \left( n_ir_i \right) = 9$ entries
\begin{eqnarray}
\mathcal{I}_1 = \{(1,1,1),(1,2,1),(1,1,3),(2,1,1),(2,2,1),(2,2,3),(2,3,1),(2,3,2),(3,1,1)\}, \nonumber \\
\mathcal{I}_2 = \{(1,1,1),(1,1,2),(1,2,1),(1,1,3),(2,2,2),(2,2,3),(2,3,1),(2,3,2),(3,3,3)\}. \nonumber
\end{eqnarray}
First, we can show that $\mathcal{I}_1$ does not satisfy Assumption $A_1$. For $\mathcal{S}_2 = \{1,2\}$ and $\mathcal{S}_3= \{1\}$ we have 
\begin{eqnarray}
\{(1,1,1),(1,2,1),(2,1,1),(2,2,1),(3,1,1)\} \subset \mathcal{U}^{(\mathcal{S}_2,\mathcal{S}_3)}. \nonumber
\end{eqnarray}
Hence, $\mathcal{U}^{(\mathcal{S}_2,\mathcal{S}_3)}$ includes 5 entries belonging to $\mathcal{I}_1$ and $\sum_{i=j+1}^{d} |\mathcal{S}_i|r_i = 4 < 5$. Therefore, $\mathcal{I}_1$ does not satisfy Assumption $A_1$ since there exists $(\mathcal{S}_2,\mathcal{S}_3)$ that violates the mentioned condition.

Second, it is easy to verify that $\mathcal{I}_2$ satisfies Assumption $A_1$ by checking all possible pairs $(\mathcal{S}_2,\mathcal{S}_3)$.
\end{example}

\begin{remark}\label{remthm5}
Assume that each column of $\mathbf{U}_{(1)}$ includes at least one observed entry, and also $\sum_{i=j+1}^{d} \left( n_ir_i \right) < n_{j+1}\dots n_d$, for some $j \in \{1,\dots,d-1\}$. Then, for any tuple $ (x_2,x_3,\dots,x_d)$ that $x_i \in \{1,\dots,n_i\}$, there exists at least one observed entry among the set 
\begin{eqnarray}
\{(1,x_2,x_3,\dots,x_d),(2,x_2,x_3,\dots,x_d),\dots,(n_1,x_2,x_3,\dots,x_d)\}.
\end{eqnarray}
Hence, there exist $\sum_{i=j+1}^{d} \left( n_ir_i \right)$ observed entries that satisfy Assumption $A_j$. This is because all possible tuples $(x_{j+1},\dots,x_d)$ are available to be selected. For example, assuming that each column of $\mathbf{U}_{(1)}$ includes at least one observed entry and $\sum_{i=2}^{d} \left( n_ir_i \right) < n_{2}\dots n_d$, we can choose $\sum_{i=2}^{d} \left( n_ir_i \right)<n_2\dots n_d=N_{-1}$ observed entries in different columns of $\mathbf{U}_{(1)}$ to satisfy Assumption $A_1$, by choosing either zero or one observed entry in each column.
\end{remark}

Given the core tensor, the following lemma characterizes the necessary and sufficient condition on observed entries that leads to finite completability.

\begin{lemma}\label{Ber}
Assume that in \eqref{nonlinear2} the core tensor $\mathcal{C} \in \mathbb{R}^{n_1  \times \cdots \times n_j \times r_{j+1} \times \cdots \times r_d }$ is given and $\mathbf{T}_i \in \mathbb{R}^{r_i \times n_i}$ are variables. Then, for almost every $\mathcal{U}$ with probability one, there are at most finitely many possible tuples $\mathbb{T}$ that satisfy \eqref{nonlinear2} if and only if Assumption $A_j$ holds.
\end{lemma}

\begin{proof}
The $\sum_{i=j+1}^{d} \left( n_ir_i \right)$ observed entries results in $\sum_{i=j+1}^{d} \left( n_ir_i \right)$ scalar polynomials in terms of entries of $\mathbb{T}$ as in \eqref{nonlinear2}-\eqref{tuckerj}. We claim that any subset of these $\sum_{i=j+1}^{d} \left( n_ir_i \right)$ polynomials with $t$ members involves at least $t$ variables in total. Then, by Note $5$, the sufficiency holds. 

In order to prove the necessity, by contradiction, assume that there exists a subset of polynomials $\{p_1,\dots,p_t\}$ that involves at most $t-1$ variables in total. Let $\tau$ be the subset of entries of $\mathcal{U}$ that result in polynomials $\{p_1,\dots,p_t\}$ and denote the minimal hull of $\tau$ by $\mathcal{U}^{(\mathcal{S}_{j+1},\dots,\mathcal{S}_{d})} $. Observe that according to Remark \ref{minhull}, $\sum_{i=j+1}^{d} |\mathcal{S}_i|r_i \leq t-1$. On the other hand, Assumption $A_j$ results that the number of polynomials in $\{p_1,\dots,p_t\}$ is at most $\sum_{i=j+1}^{d} |\mathcal{S}_i|r_i $, i.e., $t \leq \sum_{i=j+1}^{d} |\mathcal{S}_i|r_i $. Hence, we have a contradiction, which completes the proof of the lemma. 
\end{proof}

\begin{remark}\label{finiteuni}
{\rm Assumption $A_j$ results that given a core tensor there are finitely many tuples $\mathbb{T}$ such that \eqref{nonlinear2} holds. Consequently, in what follows without loss of generality, we analyze the finite completability of core tensor $\mathcal{C}$ for one particular tuple $\mathbb{T}$ among all finitely many tuples.  } {  \hfill  \qedsymbol}
%As a result, having assumption $A_j$, there are finitely many completions for the tensor $\mathcal U$ if and only if for any of the possible tuples, $(\mathbf{T}_{j+1},\dots,\mathbf{T}_{d})$, there are finitely many completions of the tensors. In other words, it is similar to say having finite number of  natural numbers, their summation is finite if and only if any of these natural numbers are finite.
\end{remark}

%\begin{definition}
Since $\sum_{i=j+1}^{d} \left( n_ir_i \right)$ entries of the sampled tensor $\mathcal{U}$ are used to determine $\mathbb{T}$, in what follows we will use the polynomials corresponding to the set of the rest $N_{\Omega}(\mathcal{U}) - \sum_{i=j+1}^{d} \left( n_ir_i \right)$ observed entries, denoted by $\mathcal{P}(\Omega)$, to obtain $\mathcal{C}$. Note that since $\mathbb{T}$ is already solved in terms of $\mathcal{C}$, each polynomial in $\mathcal{P}(\Omega)$ is in terms of elements of $\mathcal{C}$. %Let $\mathcal{P}(\Omega)$ be the set of all $N_{\Omega}(\mathcal{U}) - \sum_{i=j+1}^{d} \left( n_ir_i \right)$ other polynomial equations in terms of the entries of $\mathcal{V}$, where each entry of $\mathbb{T}$ has been replaced by the solution in terms of $\mathcal{V}$ that have been obtained using the mentioned $\sum_{i=j+1}^{d} \left( n_ir_i \right)$ polynomials.
%\end{definition}

\subsection{Geometry of Tucker Manifold}\label{indepbern}

We need to define the following equivalence class in order to characterize a condition on the sampling pattern to study the algebraic independency of the polynomials in $\mathcal{P}(\Omega)$. This equivalence class leads to a geometric structure for core tensors which specifies exactly one core tensor among all core tensors that span the same space.

\begin{definition}\label{eqclassdef}
Define an equivalence class for all core tensors $\mathcal{C} \in \mathbb{R}^{n_1 \times n_2 \times \cdots \times n_j \times r_{j+1} \times  \cdots \times r_d }$ of the sampled tensor $\mathcal{U}$ such that two core tensors $\mathcal{C}_1$ and $\mathcal{C}_2$ belong to the same class if and only if there exist full rank matrices $\mathbf{D}_i \in \mathbb{R}^{r_i \times r_i}$, $i = j+1,\dots,d$, such that
\begin{eqnarray}\label{classeq}
\mathcal{C}_2 = \mathcal{C}_1 \times_{i=j+1}^d \mathbf{D}_i.
\end{eqnarray}
\end{definition}

A subtensor $\mathcal{Y} \in \mathbb{R}^{n_1 \times n_2 \times \cdots \times n_j \times 1  \times \cdots \times 1 }$ of $\mathcal{U}$ can be represented by a core tensor $\mathcal{C}$ if there exist vectors $\theta_i \in \mathbb{R}^{r_i \times 1}$, $i=j+1,\dots,d$, such that
\begin{eqnarray}\label{coverequy}
\mathcal{Y} = \mathcal{C} \times_{i=j+1}^d \theta_i.
\end{eqnarray}
According to Definition \ref{eqclassdef}, it is easy to verify that two core tensors are in the same class if and only if one of them can represent each subtensor in $\mathbb{R}^{n_1 \times n_2 \times \cdots \times n_j \times 1  \times \cdots \times 1 }$ of the other one.

\begin{definition}\label{defunften}
Let $N_i= n_1 n_2 \dots n_i$, $\bar N_i= n_{i+1} n_{i+2} \dots n_d$ and define the matrix $\mathbf{\widetilde U}_{(i)} \in \mathbb{R}^{N_i\times \bar N_i}$ as the $i$-th unfolding of the tensor $\mathcal{U}$, such that $\mathcal{U}(\vec{x}) = \mathbf{\widetilde U}_{(i)}(\mathcal{\bar M}_{(i)} (x_1,\dots,x_i),\mathcal{\bar{\bar{M}} }_{(i)} (x_{i+1},\ldots,x_d))$, where $\mathcal{\bar M}_{(i)}$ and $\mathcal{\bar{\bar{M}}}_{(i)}$ are two  bijective mappings $\mathcal{\bar M}_{(i)}: (x_1,\dots,x_i) \rightarrow  \{1,2,\dots, N_i\}$ and $\mathcal{\bar{\bar{M}}}_{(i)}: (x_{i+1},\ldots,x_d) \rightarrow  \{1,2,\dots, \bar N_i\}$.
\end{definition}

We make the following assumption which will be referred to, when it is needed. 

{\underline{\em Assumption $B_j$}}: $n_1 n_2 \dots  n_j \geq \sum_{i=j+1}^{d} r_{i} $. 

Consider an arbitrary entry $\mathcal{C}(\vec{x})$ of core tensor $\mathcal{C}$. Note that the tuple $(x_{j+1},\dots,x_d)$ specifies the column number of this entry in $\mathbf{\widetilde{C}}_{(j)}$. Furthermore, $x_i$ specifies the row number of this entry in $\mathbf{{C}}_{(i)}$. Consequently, each column of $\mathbf{\widetilde{C}}_{(j)}$ indexed by $(x_{j+1},\dots,x_d)$ belongs to the $x_i^{\text{th}}$ row of $\mathbf{{C}}_{(i)}$ for $i \geq j+1$.

We are interested in providing a structure on the core tensor $\mathcal{C}$ such that exactly one core tensor in any class satisfies it. We present this structure on $\mathbf{\widetilde{C}}_{(j)}$ ($j$-th unfolding of $\mathcal{C}$). 

\begin{definition}\label{defstrucgood}
Consider any $(d-j)$ disjoint submatrices{\footnote{Specified by a subset of rows and a subset of columns (not necessarily consecutive).}}  $(\mathbf{P}_{j+1},\dots,\mathbf{P}_{d})$ of $\mathbf{\widetilde{C}}_{(j)}$ such that (i) $\mathbf{P}_i \in \mathbb{R}^{r_i \times r_i}$, $i=j+1,\dots,d$, (ii) the $\sum_{i=j+1}^{d} r_i$ rows of $\mathbf{\widetilde{C}}_{(j)}$ corresponding to rows of these submatrices are disjoint, (iii) the $r_i$ columns of $\mathbf{\widetilde{C}}_{(j)}$ corresponding to columns of $\mathbf{P}_i$ belong to $r_i$ disjoint rows of $\mathbf{{C}}_{(i)}$, $i=j+1,\dots,d$. Then, $\mathcal{C}$ is said to have a proper structure if $\mathbf{P}_i = \mathbf{I}_{r_i}$, $i=j+1,\dots,d$. %an identity matrix of size $r_i \times r_i$, $i=j+1,\dots,d$. 
\end{definition}

Assumption $B_j$ ensures the existence of a proper structure since the number of rows in $\mathbf{\widetilde{C}}_{(j)}$ should be at least $\sum_{i=j+1}^{d} r_i$. Note that given a proper structure there exists one core tensor in each class that satisfies it. Moreover, we can permute the rows of $\mathbf{P}_i$'s and obtain another proper structure. Consider the following  specific structure. Define 
\begin{eqnarray}\label{canonical1}
\mathbf{P}_i^{\text{can}} (x_1^{\prime},x_i) = \mathcal{C}(x_1,\underbrace{1,\ldots,1}_{i-2},x_i,\underbrace{1,\ldots,1}_{d-i}) \in \mathbb{R}^{r_i \times r_i}, \ \ \ \ \ \ \ \ i = j+1,\dots,d,
\end{eqnarray}
where $x_1^{\prime} = x_1-\sum_{s=j+1}^{i-1} r_s$, $1 + \sum_{s=j+1}^{i-1} r_s \leq x_1 \leq \sum_{s=j+1}^{i} r_s$ and $1 \leq x_i \leq r_i$. It is easily verified all three properties in Definition \ref{defstrucgood} are satisfied by $\mathbf{P}_i^{\text{can}}$, $i = j+1,\dots,d$.

%Note that $\mathbf{P}_i \in \mathbb{R}^{r_i \times r_i}$ for $i=j+1,\dots,d$, and therefore property (i) in Definition \ref{defstrucgood} is satisfied. Recall that for any entry $\mathcal{V}(\vec{x})$, the tuple $(x_1,\dots,x_j)$ specifies the corresponding row number in $\mathbf{\widetilde{C}}_{(j)}$ and as $x_1$ varies for all $\mathbf{P}_i^{\text{can}}$'s, property (ii) in Definition \ref{defstrucgood} is satisfied. Finally, recall that $x_i$ specifies the corresponding row number in $\mathbf{{C}}_{(i)}$ and as $x_i$ varies for different columns of $\mathbf{P}_i^{\text{can}}$, $i=j+1,\dots,d$, property (iii) in Definition \ref{defstrucgood} is satisfied.

%Note that all results that we show for bases having this property hold for any proper permutations of the above property, and therefore we can only consider one specific permutation without loss of generality.
 
\begin{definition}\label{classp1}
{\bf (Canonical core tensor)} We call $\mathcal{C}$ a canonical core tensor if for $i \geq j+1$ we have $\mathbf{P}_i^{\text{can}}  = \mathbf{I}_{r_i}$, where $\mathbf{I}_{r_i}$ is the $r_i \times r_i$ identity matrix.
\end{definition}

\begin{example}
{\rm Assume $\mathcal{U} \in \mathbb{R}^{7 \times 3 \times 4}$, $j=1$, $r_2=2$ and $r_3 = 3$. For simplicity in representing the canonical core tensor, we consider bijective mappings $\mathcal{\bar M}_{(i)}$ and $\mathcal{\bar{\bar{M}}}_{(i)}$ (in Definition \ref{defunften}) that result in the structure of $\mathbf{\widetilde{C}}_{(1)}$ shown in Figure \ref{figexperjj}. Observe that any permutation of rows of $\mathbf{\widetilde{C}}_{(1)}$ results in a proper structure, e.g., Figure \ref{figexper}. However, only those permutations of columns of $\mathbf{\widetilde{C}}_{(1)}$ that satisfy property (iii) in Definition \ref{defstrucgood} result in a proper structure.

%\captionsetup[subfigure]{labelformat=simple,position=rightdown}
\begin{figure}[h]
\captionsetup[subfigure]{aboveskip=1mm,belowskip=20mm}
\centering 
 \begin{subfigure}[Example of canonical core tensor.]{
 {\begin{tabular}{ |c|c|c|c|c|c|c|c|c|c|c|c| } 
 \hline 
$1$ & $0$ &  &  &  & \ \    \\ \hline
$0$ & $1$ &  & &  & \\ \hline
 &  & $1$ & $0$ & $0$ &  \\ \hline
 &  & $0$ & $1$ & $0$ &  \\ \hline
 &  & $0$ & $0$ & $1$ &  \\ \hline
 &  &  &  &  &    \\ \hline
 &  &  &  &  &    \\ 
 \hline  
\end{tabular} }
	%\caption{Example of canonical basis.}
    \label{figexperjj}
    }
\end{subfigure}\ \ \ \ \ 
\begin{subfigure}[A proper structure.]{
	\begin{tabular}{ |c|c|c|c|c|c|c|c|c|c|c|c| } 
 \hline
 &  & \ \  &  &  &   \\ \hline
$1$ & $0$ &  & &  &   \\ \hline
 &  &  & $1$ & $0$ & $0$  \\ \hline
$0$ & $1$ &  &  &  &  \\ \hline
 &  &  &  &  &   \\ \hline
 &  &  & $0$ & $1$ & $0$   \\ \hline
 &  &  & $0$ & $0$ & $1$    \\ 
 \hline 
\end{tabular} \vspace{2 mm}
    %\caption{A proper structure in $\widetilde{\mathbf{V}}_{(1)}$.}
    \label{figexper}
}
\end{subfigure}
\caption{Two proper structures in $\widetilde{\mathbf{C}}_{(1)}$.}
\end{figure}

%\begin{figure}[h]
%	\centering
%\begin{center}
%\begin{tabular}{ |c|c|c|c|c|c|c|c|c|c|c|c| } 
% \hline
%$1$ & $0$ &  &  &  & \ \  & \ \ & \ \ & \ \  & \ \ & \ \ & \ \  \\ \hline
%$0$ & $1$ &  & &  &  &  & & &  & & \\ \hline
% &  & $1$ & $0$ & $0$ &   &  &  &  &  & &  \\ \hline
% &  & $0$ & $1$ & $0$ &  &  &  &   &  & &  \\ \hline
% &  & $0$ & $0$ & $1$ & &  &  & &  & &  \\ \hline
% &  &  &  &  &   &  &  &   &  & &  \\ \hline
% &  &  &  &  &    &  & & &  & &  \\ 
% \hline
%\end{tabular}
%\end{center}
%	\caption{ Example of canonical basis.}
%	\label{figexperjj}\vspace{-4mm}
%\end{figure}

} {  \hfill  \qedsymbol}
 %Moreover, note that if we consider any proper structure, it can be obtained through permuting rows and columns of the above structure.
\end{example}

Note that \eqref{classeq} leads to the fact that the dimension of all core tensors $\mathcal{C}$ that span different spaces (without any polynomial restrictions in $\mathcal{P}(\Omega)$) is equal to $\left(\Pi_{i=1}^{j} n_i\right) \left(  \Pi_{i=j+1}^{d} r_i\right)  - \left( \sum_{i=j+1}^{d}   r_i^{2} \right)  $, as the total number of entries of $\mathbf{D}_i$'s is equal to $\left( \sum_{i=j+1}^{d}   r_i^{2} \right)$. Moreover, observe that a core tensor with a proper structure has $\left( \sum_{i=j+1}^{d}   r_i^{2} \right)$ known entries, and therefore the number of unknown entries is equal to $\left(\Pi_{i=1}^{j} n_i\right) \left(  \Pi_{i=j+1}^{d} r_i\right)  - \left( \sum_{i=j+1}^{d}   r_i^{2} \right)  $.

\begin{remark}\label{class}
{\rm In order to prove there are finitely many completions for tensor $\mathcal{U}$, it suffices to prove that there are finitely many canonical core tensors that fit in $\mathcal{U}$. } {  \hfill  \qedsymbol}
\end{remark}

Suppose $\mathcal{C}$ has a proper structure. Let $g_{j+1}(x)$ denote the maximum number of known entries among any $x$ rows of $\mathbf{\widetilde{C}}_{(j)}$. As will be seen in Section \ref{consttensdef}, $g_{j+1}(x)$ plays an important role in expressing the maximum number of algebraically independent polynomials in a subset of $\mathcal{P}(\Omega)$. Note that in exactly $r_i$ rows of $\mathbf{\widetilde{C}}_{(j)}$ there are exactly $r_i$ known entries, i.e., entries of $\mathbf{P}_i$, $i=j+1,\dots,d$. Also, there are $\sum_{i=j+1}^{d} r_i$ rows that include known entries in $\mathbf{\widetilde{C}}_{(j)}$, i.e., rows of $\mathbf{P}_i$, $i=j+1,\dots,d$.

Recall the assumption $r_{j+1} \geq \dots \geq r_d$. Therefore, as long as $x \leq r_{j+1}$, the maximum number of known entries is $g_{j+1}(x) = r_{j+1}x$ by selecting the $x$ rows of $\mathbf{\widetilde{C}}_{(j)}$ to cover any $x$ rows of $\mathbf{P}_{j+1}$. On the other hand, if $r_{j+1} \leq x \leq r_{j+1}+r_{j+2}$, the maximum number of known entries is $g_{j+1}(x) = r_{j+1}^2 + r_{j+2}(x-r_{j+1})$ by selecting the $x$ rows of $\mathbf{\widetilde{C}}_{(j)}$ to cover all rows of $\mathbf{P}_{j+1}$ and any $(x-r_{j+1})$ rows of $\mathbf{P}_{j+2}$. Then, in general we have 
\begin{eqnarray}\label{eqdefgmax}
g_{j+1}(x) = \sum_{i=j+1}^{d}  \min \left\{  r_i ,\left( x - \sum_{i^{\prime}=j+1}^{i-1} r_{i^{\prime}} \right)^+ \right\} r_i.
\end{eqnarray}

\subsection{Constraint Tensor}\label{consttensdef}

In the following, we propose a procedure to construct a $(j+1)^{\text{th}}$-order binary tensor $\mathbf{\breve{\Omega}}$ based on $\Omega$ such that $\mathcal{P}(\mathbf{\breve{\Omega}}) = \mathcal{P}(\Omega) $. Using $\mathbf{\breve{\Omega}}$, we are able to recognize the observed entries that have been used to obtain the tuple $\mathbb{T}$, and we can easily verify if two polynomials in $\mathcal{P}(\Omega)$ are in terms of the same set of variables. Then, in Section \ref{algebraind}, we characterize the relationship between the maximum number of algebraically independent polynomials in $\mathcal{P}(\mathbf{\breve{\Omega}})$ and $\mathbf{\breve{\Omega}}$.

For any subtensor $\mathcal{Y} \in \mathbb{R}^{n_1 \times n_2 \times \cdots \times n_j \times 1  \times \cdots \times 1 }$ of the tensor $\mathcal{U}$, there exist row vectors $\theta_i \in \mathbb{R}^{r_i \times 1}$, $i = j+1,\dots,d$, such that \eqref{coverequy} holds or equivalently
\begin{eqnarray}
\mathcal{Y}(x_1,\ldots,x_j,\vec{1}_{d-j}) =  \sum_{k_{j+1}=1}^{r_{j+1}} \cdots \sum_{k_{d}=1}^{r_{d}} \mathcal{C}(x_1,\ldots,x_j,k_{j+1},\ldots,k_d) \theta_{j+1}(k_{j+1},1)  \dots \theta_d(k_d,1),
\end{eqnarray}
where $\vec{1}_{d-j}$ is an all-$1$ $(d-j)$-dimensional row vector. %Note that, two bases $\mathcal{V}_1$ and $\mathcal{V}_2$ are in the same class according the defined equivalence class if and only if they cover the same set of tensors $\mathcal{Y}$ through \eqref{nonlinear}, i.e., they cover the same space.

For each subtensor $\mathcal{Y}$ of the sampled tensor $\mathcal{U}$, let $N_{\Omega}(\mathcal{Y}^{\mathbb{T}})$ denote the number of sampled entries in $\mathcal{Y}$ that have been used to obtain the tuple $\mathbb{T}$. Then, $\mathcal{Y}$ contributes $N_{\Omega}(\mathcal{Y}) - N_{\Omega}(\mathcal{Y}^{\mathbb{T}})$ polynomial equations in terms of the entries of the core tensor $\mathcal{C}$ among all $N_{\Omega}(\mathcal{U}) - \sum_{i=j+1}^{d} \left( n_ir_i \right)$ polynomials in $\mathcal{P}(\Omega)$.

The sampled tensor $\mathcal{U}$ includes $n_{j+1}  n_{j+2} \cdots n_d$ subtensors that belong to  $\mathbb{R}^{n_1 \times n_2 \times \cdots \times n_j \times 1  \times \cdots \times 1 }$ and we label these subtensors by $\mathcal{Y}_{(t_{j+1},\dots,t_{d})}$ where ${(t_{j+1},\dots,t_{d})}$ represents the coordinate of the subtensor. Define a binary valued tensor $\mathcal{\breve{Y}}_{(t_{j+1},\cdots,t_{d})} \in \mathbb{R}^{n_1 \times n_2 \times \cdots \times n_j \times \overbrace{ 1  \times \dots \times 1}^{d-j} \times  k}$, where $k= N_{\Omega}(\mathcal{Y}_{(t_{j+1},\ldots,t_{d})}) - N_{\Omega}(\mathcal{Y}_{(t_{j+1},\ldots,t_{d})}^{\mathbb{T}})$ and its entries are described as the following. We can look at $\mathcal{\breve{Y}}_{(t_{j+1},\cdots,t_{d})}$ as $k$ tensors each belongs to $\mathbb{R}^{n_1 \times n_2 \times \cdots \times n_j \times 1  \times \cdots \times 1 }$. For each of the mentioned $k$ tensors in $\mathcal{\breve{Y}}_{(t_{j+1},\cdots,t_{d})}$ we set the entries corresponding to the $N_{\Omega}(\mathcal{Y}_{(t_{j+1},\ldots,t_{d})}^{\mathbb{T}})$ observed entries that are used to obtain $\mathbb{T}$ in  \eqref{nonlinear2} equal to $1$. For each of the other $k$ observed entries, we pick one of the $k$ tensors of $\mathcal{\breve{Y}}_{(t_{j+1},\cdots,t_{d})}$ and set its corresponding entry (the same location as that specific observed entry) equal to $1$ and set the rest of the entries equal to $0$. 

%\begin{example}
%Assume that $d=3$ and $j=2$. Consider $\mathcal{Y}_{(1)} \in \mathbb{R}^{3 \times 2 \times 1}$ with $N_{\Omega}(\mathcal{Y}_{(1)}^{\mathbb{T}}) = 2$ where the observed entries are $(1,1,1)$, $(1,2,1)$, $(2,2,1)$ and $(3,1,1)$. Then, the corresponding tensor for this subtensor will be $\mathcal{\breve{Y}}_{(1)} \in \mathbb{R}^{3 \times 2 \times 1 \times 2 }$. Suppose that the $N_{\Omega}(\mathcal{Y}_{(1)}^{\mathbb{T}}) =2$ entries that we use to obtain $\mathbb{T}$ are $(1,1,1)$ and $(1,2,1)$. Hence, $\mathcal{\breve{Y}}_{(1)}(1,1,1,1) = \mathcal{\breve{Y}}_{(1)}(1,1,1,2) = \mathcal{\breve{Y}}_{(1)}(1,2,1,1) = \mathcal{\breve{Y}}_{(1)}(1,2,1,2) =1$, and also for the two other observed entires we have $\mathcal{\breve{Y}}_{(1)}(2,2,1,1) =1 $ and $\mathcal{\breve{Y}}_{(1)}(3,1,1,2)=1$ and the rest of the entries of $\mathcal{\breve{Y}}_{(1)}$ are equal to $0$.
%\end{example}

For the sake of simplicity in notation, we treat tensors $\mathcal{\breve{Y}}_{(t_{j+1},\cdots,t_{d})}$ as a member of $\mathbb{R}^{n_1 \times n_2 \times \cdots \times n_j \times  k}$ instead of  $\mathbb{R}^{n_1 \times n_2 \times \cdots \times n_j \times \overbrace{ 1  \times \cdots \times 1}^{d-j} \times  k}$. Now, by putting together all $n_{j+1}  n_{j+2} \cdots n_d$ tensors in dimension $(j+1)$, we construct a binary valued tensor $\mathbf{\breve{\Omega}} \in \mathbb{R}^{n_1 \times n_2 \times \cdots \times n_j \times K_j}$, where $K_j = N_{\Omega}(\mathcal{U}) - \sum_{i=j+1}^{d} \left( n_ir_i\right)$ and call it the {\bf constraint tensor}. In order to shed some light on the above procedure we give an illustrative example in the following.

\begin{example}
{\rm Consider an example in which $d=3$, $j=2$, $r_3=2$ and $\mathcal{U} \in \mathbb{R}^{3 \times 2 \times 2}$. Assume that $\Omega(x,y,z)=1$ if $(x,y,z) \in \mathcal{S}$ and $\Omega(x,y,z)=0$ otherwise, where 
\begin{eqnarray}
\mathcal{S} = \{(1,1,1),(1,2,1),(2,2,1), (3,1,1),(1,1,2),(2,1,2),(3,2,2)\}, \nonumber
\end{eqnarray}
represents the set of observed entries. Also, assume that $N_{\Omega}(\mathcal{Y}_{(1)}^{\mathbb{T}})= N_{\Omega}(\mathcal{Y}_{(2)}^{\mathbb{T}}) = 2$ and the entries that we use to obtain $\mathbb{T}$ are $(1,1,1)$, $(1,2,1)$, $(1,1,2)$ and $(3,2,2)$. Hence, $\mathcal{\breve{Y}}_{(1)} \in \mathbb{R}^{3 \times 2 \times 1 \times 2 }$, $\mathcal{\breve{Y}}_{(2)} \in \mathbb{R}^{3 \times 2 \times 1 \times 1 }$, and therefore the constraint tensor $\mathbf{\breve{\Omega}}$ belongs to $\mathbb{R}^{3 \times 2 \times 3}$.

Note that $\mathcal{\breve{Y}}_{(1)}(1,1,1,1) = \mathcal{\breve{Y}}_{(1)}(1,1,1,2) = \mathcal{\breve{Y}}_{(1)}(1,2,1,1) = \mathcal{\breve{Y}}_{(1)}(1,2,1,2) =1$, and also for the two other observed entires we have $\mathcal{\breve{Y}}_{(1)}(2,2,1,1) =1 $ and $\mathcal{\breve{Y}}_{(1)}(3,1,1,2)=1$ and the rest of the entries of $\mathcal{\breve{Y}}_{(1)}$ are equal to zero. Moreover, $\mathcal{\breve{Y}}_{(2)}(1,1,1,1) = \mathcal{\breve{Y}}_{(2)}(2,1,1,1) = \mathcal{\breve{Y}}_{(2)}(3,2,1,1)=1$ and the rest of the entries of $\mathcal{\breve{Y}}_{(2)}$ are equal to zero.

Then, $\mathbf{\breve{\Omega}}(x,y,z)=1$ if $(x,y,z) \in \mathcal{S}^{\prime}$ and $\mathbf{\breve{\Omega}}(x,y,z)=0$ otherwise, where 
\begin{eqnarray}
\mathcal{\breve{S}} = \{(1,1,1),(1,2,1),(2,2,1),(1,1,2),(1,2,2),(3,1,2),(1,1,3),(3,2,3),(2,1,3)\}. \nonumber
\end{eqnarray}  } {  \hfill  \qedsymbol}
\end{example}

Note that each subtensor of $\mathbf{\breve{\Omega}}$ that belongs to $\mathbb{R}^{n_1 \times \dots \times n_j \times 1}$ represents one of the polynomials in $\mathcal{P}(\Omega)$ besides showing the polynomials that have been used to obtain $\mathbb{T}$. More specifically, consider a subtensor of $\mathbf{\breve{\Omega}}$ that belongs to $\mathbb{R}^{n_1 \times \dots \times n_j \times 1}$ with $l+1$ nonzero entries. Observe that exactly $l$ of them correspond to the observed entries that have been used to obtain $\mathbb{T}$. Hence, this subtensor represents a polynomial after replacing entries of $\mathbb{T}$ by the expressions in terms of entries of $\mathcal{C}$, i.e., a polynomial in $\mathcal{P}(\Omega)$.

Recall that the tuple $(x_{j+1},\dots,x_d)$ specifies the column number of entry $\mathcal{C}(\vec{x})$ in $\mathbf{\widetilde{C}}_{(j)}$. Hence, if the $i$-th column of $\mathbf{\breve{\Omega}}_{(j+1)}$ is nonzero, then there exists a polynomial in $\mathcal{P}(\Omega)$ that involves all entries of core tensor corresponding to the entries of the $i$-th row of $\mathbf{\widetilde{C}}_{(j)}$.

\subsection{Algebraic Independence}\label{algebraind}

In this subsection, we derive the required number of algebraically independent polynomials in $\mathcal{P}(\Omega)$ for finite completability. Then, a sampling pattern on the constraint tensor is proposed to obtain the maximum number of algebraically independent polynomials in $\mathcal{P}(\mathbf{\breve{\Omega}})$.

The following lemma determines the required number of algebraically independent polynomials in $\mathcal{P}(\Omega)$ that is needed to ensure finite completability of the core tensor.

\begin{lemma}\label{thm2}
Assume that Assumptions $A_j$ and $B_j$ hold. For almost every $\mathcal{U}$, there exist only finitely many completions of $\mathcal{U}$ if and only if there exist $  \left(\Pi_{i=1}^{j} n_i\right) \left(  \Pi_{i=j+1}^{d} r_i\right)  - \left( \sum_{i=j+1}^{d}   r_i^{2} \right) $ algebraically independent polynomials in $\mathcal{P}(\Omega)$.
\end{lemma}

\begin{proof}
Assume that $\mathcal{C} \in \mathbb{R}^{n_1 \times n_2 \times \cdots \times n_j \times r_{j+1}  \times \cdots \times r_d }$ is a core tensor for the sampled tensor $\mathcal U$. Since assumption $A_j$ holds, Lemma \ref{Ber} results that there exist finitely many tuples $\mathbb{T}$ such that \eqref{nonlinear2} holds. However, according to Remark \ref{finiteuni}, it suffices to assume $\mathbb{T}$ is fixed and then prove the statement. Let $\mathcal{P}(\Omega)=\{p_1,\dots,p_m\}$ and define $\mathcal{S}_i$ as the set of all core tensors that satisfy polynomial restrictions $\{p_1,\dots,p_i\}$, $i=0,\dots,m$ ($\mathcal{S}_0$ is the set of all core tensors without any polynomial restriction).

Observe that each algebraically independent polynomial reduces the dimension (degree of freedom) of the set of solutions by one. In other words, $\text{dim}(\mathcal{S}_i)=\text{dim}(\mathcal{S}_{i-1})$ if the maximum number of algebraically independent polynomials in sets $\{p_1,\dots,p_i\}$ and $\{p_1,\dots,p_{i-1}\}$ are the same and $\text{dim}(\mathcal{S}_i)=\text{dim}(\mathcal{S}_{i-1})-1$ otherwise. Moreover, with probability one, $|\mathcal{S}_m|$ is finite if and only if there are $\text{dim}(\mathcal{C})=\text{dim}(\mathcal{S}_0)$ algebraically independent polynomial restrictions on the entries of the core tensor $\mathcal C$, i.e., $|\mathcal{S}_m|$ is finite if and only if $\text{dim}(\mathcal{S}_m)=0$ \cite{charact}. Hence, there are finitely many completions of the sampled tensor $\mathcal{U}$ if and only if there exist $ \text{dim}(\mathcal{C}) = \left(\Pi_{i=1}^{j} n_i\right) \left(  \Pi_{i=j+1}^{d} r_i\right)  - \left( \sum_{i=j+1}^{d}   r_i^{2} \right) $ algebraically independent polynomials in $\mathcal{P}(\Omega)$.
\end{proof}

As a result of Lemma \ref{thm2}, we can certify finite completability based on the maximum number of algebraically independent polynomials in $\mathcal{P}(\Omega)=\mathcal{P}(\mathbf{\breve{\Omega}})$.

\begin{definition}
Let $\mathbf{\breve{\Omega}}^{\prime} \in \mathbb{R}^{n_1 \times n_2 \times \cdots \times n_j \times t}$ be a subtensor of the constraint tensor $\mathbf{\breve{\Omega}}$. Let $m_i(\mathbf{\breve{\Omega}}^{\prime})$ denote the number of nonzero columns of $\mathbf{\breve{\Omega}}^{\prime}_{(i)}$. Also, let $\mathcal{P}(\mathbf{\breve{\Omega}}^{\prime})$ denote the set of polynomials that correspond to nonzero entries of $\mathbf{\breve{\Omega}}^{\prime}$.
\end{definition}

Recall Note $2$ regarding the number of involved entries of core tensor in a set of polynomials. However, as mentioned earlier, some of the entries of $\mathcal{C}$ are known, i.e., $(\mathbf{P}_{j+1},\dots,\mathbf{P}_d)$. Therefore, in order to find the number of variables (unknown entries of $\mathcal{C}$) in a set of polynomials, we should subtract the number of known entries in the corresponding pattern from the total number of involved entries.

For any subtensor $\mathbf{\breve{\Omega}}^{\prime} \in \mathbb{R}^{n_1 \times n_2 \times \cdots \times n_j \times t}$ of the constraint tensor, the next theorem states an upper bound on the number of algebraically independent polynomials in the set $\mathcal{P}(\mathbf{\breve{\Omega}}^{\prime})$. Recall that $\mathcal{P}(\mathbf{\breve{\Omega}}^{\prime})$ includes exactly $t$ polynomials. %More specifically, in Theorem \ref{thm3} below, we show that the maximum number of algebraically independent polynomials is upper bounded by a function of the number of known entries that are involved in the polynomials in the worst possible structure (explained in the proof of Theorem \ref{thm3}) among all possible proper structure.

%The following example illustrates the above definition.
%\begin{example}
%Assume $d=4$, $j=1$, $r_2=4$, $r_3 = 2$ and $r_4=1$. Then, we have
%\begin{eqnarray}
%g_2(1) &=& (r_2) = 4, \nonumber \\
%g_2(2) &=& 2(r_2)= 8, \nonumber \\
%g_2(3) &=& 3(r_2)= 12, \nonumber \\
%g_2(4) &=& 4(r_2)= 16, \nonumber \\
%g_2(5) &=& 4(r_2) + (r_3) = 18, \nonumber \\
%g_2(6) &=& 4(r_2) + 2(r_3) = 20, \nonumber \\
%g_2(7) &=& 4(r_2) + 2(r_3) + (r_4) = 21. \nonumber
%\end{eqnarray}
%Moreover, for any $x \in \mathbb{N}$ that $x > \sum_{i=j+1}^{d} r_i = 7$, we have $g_{2}(x) = \sum_{i=j+1}^{d} r_i^2 = 21$.
%\end{example}

\begin{theorem}\label{thm3}
Assume that Assumption $B_j$ holds. For any subtensor $\mathbf{\breve{\Omega}}^{\prime} \in \mathbb{R}^{n_1 \times n_2 \times \cdots \times n_j \times t}$ of the constraint tensor, the maximum number of algebraically independent polynomials in $\mathcal{P}(\mathbf{\breve{\Omega}}^{\prime})$ is no more than
\begin{eqnarray}\label{numbalginprw}
\left(\Pi_{i=j+1}^{d} r_i \right)  m_{j+1}(\mathbf{\breve{\Omega}}^{\prime}) -  g_{j+1}(m_{j+1}(\mathbf{\breve{\Omega}}^{\prime})),
\end{eqnarray}
where $g_{j+1}(\cdot)$ is given in \eqref{eqdefgmax}.
\end{theorem}

\begin{proof}
Observe that the number of algebraically independent polynomials in a subset of polynomials of $\mathcal{P}(\mathbf{\breve{\Omega}}^{\prime})$ is at most equal to the total number of variables that are involved in the corresponding polynomials. According to  \eqref{tuckerj}, for each observed entry $\mathcal{U}(\vec{x})$, we have a polynomial in terms of $\Pi_{i=j+1}^{d} r_i$ entries of the core tensor $\mathcal{C}$ corresponding to the first $j$ coordinates of the location of the observed entry. Therefore, the number of such variables in $\mathcal{P}(\mathbf{\breve{\Omega}}^{\prime})$ will be $\Pi_{i=j+1}^{d} r_i$ times the number of different tuples $(x_1,\dots,x_j)$ among the corresponding observed entries. The number of nonzero columns of  $\mathbf{\breve{\Omega}}^{\prime}_{(j+1)}$ is exactly equal to the number such tuples. Therefore, the number of involved entries of $\mathcal{C}$ (known and unknown) in polynomials in $\mathcal{P}(\mathbf{\breve{\Omega}}^{\prime})$ is equal to $  \left(\Pi_{i=j+1}^{d} r_i \right) m_{j+1}(\mathbf{\breve{\Omega}}^{\prime})$

On the other hand, among the $\sum_{i=j+1}^{d} r_i^2$ known entries corresponding to $(\mathbf{P}_{j+1},\dots,\mathbf{P}_d)$ in $\widetilde{\mathbf{C}}_{(j)}$, $g_{j+1}(m_{j+1}(\mathbf{\breve{\Omega}}^{\prime}))$ of them are involved in polynomials of $\mathcal{P}(\mathbf{\breve{\Omega}}^{\prime})$. This is because entries of the $i$-th row of $\widetilde{\mathbf{C}}_{(j)}$ are involved in a polynomials if and only if the $i$-th column of $\mathbf{\breve{\Omega}}^{\prime}_{(j+1)}$ includes at least one nonzero entry. Hence, the number of variables that are involved in the set of polynomials $\mathcal{P}(\mathbf{\breve{\Omega}}^{\prime})$ is given by \eqref{numbalginprw} for a particular proper structure and proof is complete as \eqref{numbalginprw} is an upper bound for the number of algebraically independent polynomials.
\end{proof}

We are also interested in finding a condition on $\mathbf{\breve{\Omega}}^{\prime}$ which results that $\mathcal{P}(\mathbf{\breve{\Omega}}^{\prime})$ is minimally algebraically dependent, i.e., the polynomials in $\mathcal{P}(\mathbf{\breve{\Omega}}^{\prime})$ are algebraically dependent but polynomials in every of its proper subset are algebraically independent. This can help obtain the maximum number of algebraically independent polynomials in $\mathcal{P}(\mathbf{\breve{\Omega}}^{\prime})$ as Theorem \ref{thm3} only provides an upper bound. The next lemma will be used in Theorem \ref{thm4} in order to find a condition on $\mathbf{\breve{\Omega}}^{\prime}$ which results that the set of polynomials in $\mathcal{P}(\mathbf{\breve{\Omega}}^{\prime})$ is minimally algebraically dependent. The following lemma is a re-statement of Lemma 7 in \cite{ashraphijuo4}.

\begin{lemma}\label{lemma3new}
Assume that Assumption $B_j$ holds. Suppose that $\mathbf{\breve{\Omega}}^{\prime} \in \mathbb{R}^{n_1 \times n_2 \times \cdots \times n_j \times t}$ is a subtensor of the constraint tensor such that $\mathcal{P}(\mathbf{\breve{\Omega}}^{\prime})$ is minimally algebraically dependent. Then, for almost every $\mathcal{U}$, the number of variables that are involved in the set of polynomials $\mathcal{P}(\mathbf{\breve{\Omega}}^{\prime})$ is $t-1$.
\end{lemma}

Finally, the next theorem provides a relationship between the exact number of algebraically independent polynomials in $\mathcal{P}(\mathbf{\breve{\Omega}})$ and a geometric property on $\mathbf{\breve{\Omega}}$.

\begin{theorem}\label{thm4}
Assume that Assumption $B_j$ holds. The polynomials in the set $\mathcal{P}(\mathbf{\breve{\Omega}})$ are algebraically dependent if and only if $ \left(\Pi_{i=j+1}^{d} r_i \right)  m_{j+1}(\mathbf{\breve{\Omega}}^{\prime}) - g_{j+1}(m_{j+1}(\mathbf{\breve{\Omega}}^{\prime})) < t$ \ for some subtensor $\mathbf{\breve{\Omega}}^{\prime} \in \mathbb{R}^{n_1 \times n_2 \times \cdots \times n_j \times t}$ of the constraint tensor $\mathbf{\breve{\Omega}}$.
\end{theorem}

\begin{proof}
If the polynomials in set $\mathcal{P}(\mathbf{\breve{\Omega}})$ are algebraically dependent, then there exists a subset of the polynomials that are minimally algebraically dependent. According to Lemma \ref{lemma3new}, if $\mathbf{\breve{\Omega}}^{\prime} \in \mathbb{R}^{n_1 \times n_2 \times \cdots \times n_j \times t}$ is the corresponding subtensor to this minimally algebraically dependent set of polynomials, the number of variables that are involved in $\mathcal{P}(\mathbf{\breve{\Omega}}^{\prime})=\{p_1,,p_2\dots,p_t\}$ is equal to $t-1$. On the other hand, $\left(\Pi_{i=j+1}^{d} r_i \right)  m_{j+1}(\mathbf{\breve{\Omega}}^{\prime}) - g_{j+1}(m_{j+1}(\mathbf{\breve{\Omega}}^{\prime}))$ is the minimum possible number of involved variables in $\mathcal{P}(\mathbf{\breve{\Omega}}^{\prime})$ since $g_{j+1}(m_{j+1}(\mathbf{\breve{\Omega}}^{\prime}))$ is the maximum number of known entries of core tensor that are involved in $\mathcal{P}(\mathbf{\breve{\Omega}}^{\prime})$. Therefore, $\left(\Pi_{i=j+1}^{d} r_i \right)  m_{j+1}(\mathbf{\breve{\Omega}}^{\prime}) - g_{j+1}(m_{j+1}(\mathbf{\breve{\Omega}}^{\prime})) \leq t-1 $.%, as $t-1$ is the number of involved variables in polynomials in $\mathcal{P}(\mathbf{\breve{\Omega}}^{\prime})$.

In order to prove the other side of the statement, assume that $ \left(\Pi_{i=j+1}^{d} r_i \right)  m_{j+1}(\mathbf{\breve{\Omega}}^{\prime}) - g_{j+1}(m_{j+1}(\mathbf{\breve{\Omega}}^{\prime})) < t$ \ for some subtensor $\mathbf{\breve{\Omega}}^{\prime} \in \mathbb{R}^{n_1 \times n_2 \times \cdots \times n_j \times t}$ of the constraint tensor $\mathbf{\breve{\Omega}}$. Recall that $t$ is the number of polynomials in $\mathcal{P}(\mathbf{\breve{\Omega}}^{\prime})$. On the other hand, according to Theorem \ref{thm3}, $ \left(\Pi_{i=j+1}^{d} r_i \right)  m_{j+1}(\mathbf{\breve{\Omega}}^{\prime}) - g_{j+1}(m_{j+1}(\mathbf{\breve{\Omega}}^{\prime}))$ is the maximum number of algebraically independent polynomials, and therefore the polynomials in $\mathcal{P}(\mathbf{\breve{\Omega}}^{\prime})$ are not algebraically independent and it completes the proof.
\end{proof}

\subsection{Finite Completability Using Analysis on Tucker Manifold}\label{sec4}

Theorem \ref{thm4} together with Lemma \ref{thm2} can lead to a necessary and sufficient condition on the constraint tensor $\mathbf{\breve{\Omega}}$ in order to ensure that there are finitely many completions for the sampled tensor $\mathcal{U}$, as stated by the next theorem.

\begin{theorem}\label{thm5}
Assume that Assumptions $A_j$ and $B_j$ hold. Then, for almost every $\mathcal{U}$, there are only finitely many tensors that fit in the sampled tensor $\mathcal{U}$, and have tensor rank components $r_i$ for $i = j+1,\dots,d$ if and only if the following two conditions hold: 

(i) there exists a subtensor $\mathbf{\breve{\Omega}}^{\prime} \in \mathbb{R}^{n_1 \times n_2 \times \cdots \times n_j \times n}$ of the constraint tensor such that \ $ n = \left(\Pi_{i=1}^{j} n_i\right) \left(  \Pi_{i=j+1}^{d} r_i\right)  - \left( \sum_{i=j+1}^{d}  r_i^{2} \right)$, and 

(ii) for any $t \in \{1,\dots,n\}$ and any subtensor $\mathbf{\breve{\Omega}}^{\prime \prime} \in \mathbb{R}^{n_1 \times n_2 \times \cdots \times n_j \times t}$ of the tensor $\mathbf{\breve{\Omega}}^{\prime}$ (in condition (i)), the following inequality holds
\begin{eqnarray}\label{ineqp}
\left(\Pi_{i=j+1}^{d} r_i \right)  m_{j+1}(\mathbf{\breve{\Omega}}^{\prime \prime}) - g_{j+1}(m_{j+1}(\mathbf{\breve{\Omega}}^{\prime \prime})) \geq t.
\end{eqnarray}
\end{theorem}

\begin{proof}
According to Lemma \ref{thm2}, for almost every $\mathcal{U}$, there are finitely many completions of $\mathcal{U}$ if and only if there exist $\left(\Pi_{i=1}^{j} n_i\right) \left(  \Pi_{i=j+1}^{d} r_i\right)  - \left(\sum_{i=j+1}^{d}   r_i^{2} \right)$ algebraically independent polynomials in $\mathcal{P}(\mathbf{\breve{\Omega}})$. On the other hand, according to Theorem \ref{thm4} we conclude that a set of polynomials are algebraically independent if and only if condition (ii) in the statement of the theorem holds. Hence, for almost every $\mathcal{U}$, there are finitely many completions of $\mathcal{U}$ if and only if conditions (i) and (ii) hold.
%Assume that there exists such tensor $\mathbf{\breve{\Omega}}^{\prime}$ and we want to show the finiteness for completion. According to Lemma \ref{thm2}, it suffices to show that there are at least $\left(\Pi_{i=1}^{j} n_i\right) \left(  \Pi_{i=j+1}^{d} r_i\right)  - \left(\sum_{i=j+1}^{d}   r_i^{2} \right)$ algebraically independent polynomials in $\mathcal{P}(\mathbf{\breve{\Omega}}^{\prime})$. On the other hand, according to Theorem \ref{thm4} we conclude that all  $n = \left(\Pi_{i=1}^{j} n_i\right) \left(  \Pi_{i=j+1}^{d} r_i\right)  - \left( \sum_{i=j+1}^{d}  r_i^{2} \right)$ polynomials in $\mathcal{P}(\mathbf{\breve{\Omega}}^{\prime})$ are algebraically independent.
%For the converse, suppose that there are finitely many completions and we need to demonstrate there exists $\mathbf{\breve{\Omega}}^{\prime}$ that satisfies the above property in the statement of theorem. By contradiction, assume that there does not exist such $\mathbf{\breve{\Omega}}^{\prime}$. As a result, the maximum number of algebraically independent polynomials in $\mathcal{P}(\mathbf{\breve{\Omega}}^{\prime})$ is strictly less than $ n = \left(\Pi_{i=1}^{j} n_i\right) \left(  \Pi_{i=j+1}^{d} r_i\right)  - \left( \sum_{i=j+1}^{d}  r_i^{2} \right)$. Then, according to Lemma \ref{thm2} there are infinitely many completions for tensor $\mathcal{U}$ with tensor rank components $r_i$ for $i \geq j+1$ and this contradiction completes the proof.
\end{proof}

\begin{remark}
As a sanity check, we next show that when $d=2$ (matrix case), Theorem \ref{thm5}  reduces to Theorem 1 in \cite{charact}. To see this, note that for $d=2$, we only have one rank component and denote it by $r$. Then, Condition (i) states that there exists a submatrix $\breve{\mathbf{\Omega}}^{\prime} \in \mathbb{R}^{n_1 \times (n_1r-r^2)}$ (basically $n = n_1r - r^2$ in Condition (i)) of the constraint matrix $\breve{\mathbf{\Omega}}$. And Condition (ii) on the property of submatrix $\breve{\mathbf{\Omega}}^{\prime}$ becomes: for any $t \in \{1,\dots,n_1r - r^2\}$ and any submatrix $\mathbf{\breve{\Omega}}^{\prime \prime} \in \mathbb{R}^{n_1 \times t}$ of the matrix $\mathbf{\breve{\Omega}}^{\prime}$, the following inequality holds
\begin{eqnarray}\label{ineqpcomprew}
r  m_{1}(\mathbf{\breve{\Omega}}^{\prime \prime}) - g_{1}(m_{1}(\mathbf{\breve{\Omega}}^{\prime \prime})) \geq t.
\end{eqnarray}

Note that due to the way that we constructed the constraint matrix (tensor) each column of $\mathbf{\breve{\Omega}}^{\prime \prime}$ has exactly $r+1$ non-zero entries. Therefore, $\mathbf{\breve{\Omega}}^{\prime \prime}$ has at least $r+1$ non-zero rows, i.e., $m_{1}(\mathbf{\breve{\Omega}}^{\prime \prime}) \geq r+1$. Then, according to the definition of $g_1$ in \eqref{eqdefgmax}, $g_{1}(m_{1}(\mathbf{\breve{\Omega}}^{\prime \prime})) = r \min\{r,m_{1}(\mathbf{\breve{\Omega}}^{\prime \prime})-0\} = r^2$. Therefore,  \eqref{ineqpcomprew} becomes
\begin{eqnarray}\label{ineqpcomprew2}
r  m_{1}(\mathbf{\breve{\Omega}}^{\prime \prime}) - r^2 \geq t,
\end{eqnarray}
or equivalently,
\begin{eqnarray}\label{ineqpcomprew3}
 m_{1}(\mathbf{\breve{\Omega}}^{\prime \prime})  \geq t/r + r.
\end{eqnarray}
Hence, Theorem  \ref{thm5}  for $d=2$ is exactly the same as  Theorem 1 in \cite{charact}.
\end{remark}

There are a few observations based on Theorem \ref{thm5}:

\begin{itemize}
\item In theorem \ref{thm5}, we characterize all of the sampling patterns that ensure there are only finitely many completions such that if only one single sample from that pattern is missed, then there are infinitely many completions.

\item If we use the Grassmannian analysis on each dimension individually, we will only obtain a {\em sufficient} condition on the sampling patterns for finite completability, given multiple rank constraints. However, our proposed analysis on Tucker manifold results in a {\em necessary and sufficient} condition on the sampling patterns for finite completability when multiple rank components are given. Hence, our proposed analysis on Tucker manifold in general requires much less number of of samples to ensure  finite completability in comparison with the analysis on Grassmannian manifold when more than one rank component is given.
      
\item It is also important to observe the advantage of the proposed method as we decrease the value of $j$ since we are incorporating more rank components. Intuitively, for the case of $j=1$, the polynomials obtained in \eqref{nonlinear2} involve much more variables in comparison with the low-order analysis on Grassmannian manifold, i.e., $j=d-1$. Therefore, the case of $j=1$ requires much less number of samples in order to have sufficient number of algebraically independent polynomials for finite completability.

\item Note that our analysis is not valid when $j=0$ (given all rank components) since the defined proper structure does not have meaning any more as zero unfolding does not exist, and therefore we are not able to characterize the necessary and sufficient condition on sampling pattern for finite completability. However, if all rank components are given we can simply ignore one of them and characterize the necessary and sufficient condition for $j=1$ which results in a sufficient condition for $j=0$.
\end{itemize}

If the number of completions given $(r_{j+1},\dots,r_d)$ is finite, it can be concluded that the number of completions given $(r_{j},\dots,r_d)$ is finite as well. Therefore, this property should be verifiable through the geometric property \eqref{ineqp} proposed in Theorem \ref{thm5}. In the following lemma, using the pigeonhole principle, we show this result without analyzing the algebraic variety.

\begin{lemma}\label{subpro}
Assume that Assumptions $A_j$, $B_j$, $A_{j-1}$ and $B_{j-1}$ hold. Suppose that the sampling pattern is such that the properties (i) and (ii) in the statement of Theorem \ref{thm5} hold for $j$. Then, the properties (i) and (ii) in the statement of Theorem \ref{thm5} hold for $j-1$.
\end{lemma}

\begin{proof}
The proof is given in Appendix \ref{apdx1}.
\end{proof}

\section{Probabilistic Conditions for Finite Completability}\label{corrcetion}

In this section, two different lower bounds on the sampling probability are proposed and analyzed  to ensure  finite completability. The first bound is obtained by applying \cite[Theorem $3$]{charact} and the second bound is obtained through the proposed geometric approach in Theorem \ref{thm5}. We will observe later that our proposed analysis on Tucker manifold leads to a better lower bound through numerical analysis.

Here we briefly outline the key steps of the second approach. Lemmas \ref{minsamp} and \ref{minsamp2} each provides a lower bound on the sampling probability that results in a geometric property  for $\Omega$, i.e., inequalities \eqref{proper1} and \eqref{proper2}, respectively. Then, Theorem \ref{thm6} takes advantage of the above lemmas to propose a bound on the sampling probability to guarantee property \eqref{ineqp} for $\Omega$. Finally, Lemma \ref{Omega} shows that \eqref{ineqp} also holds for the constraint tensor $\mathbf{\breve{\Omega}}$. In order to show Lemma \ref{Omega} we develop a generalization of Hall's Theorem on bipartite graphs in Theorem \ref{genHall}.

We use the approach similar to \cite[Lemma $9$]{charact} in order to apply Theorem \ref{thm5} and obtain a lower bound on the sampling probability to ensure that there are only finitely many completions. According to our earlier discussion, Theorem \ref{thm5} for the case of $j=1$ results in the mildest condition on the sampling patterns for ensuring finite completability. In other words, setting $j=1$ in Theorem \ref{thm5} results in a tighter lower bound.

\subsection{Lower Bound on Sampling Probability based on Analysis on Grassmannian Manifold}

We first state a lemma, whose  corollary (Corollary \ref{azuma}) is used extensively in this section in order to find a lower bound on the sampling probability to guarantee a lower bound on the number of sampled entries with high probability.

\begin{lemma}\label{xgfh}
Consider a vector with $n_i$ entries and assume each entry is observed with probability $p$ and independently from the other entries. Then, with probability  at least $1-\exp(-\frac{n_i}{2c^2})$ at least $(p - \frac{1}{c})n_i$ entries are observed.
\end{lemma}

\begin{proof}
Azuma's inequality states that for a martingale $\{ X_k: k = 0, 1, 2, ... \}$ that $|X_k - X_{k-1}| \leq 1$ holds, we have $P(X_n - \mathbb{E}[X_n] >t) \leq \exp\left(\frac{-t^2}{2n}\right)$ and $P(\mathbb{E}[X_n] - X_n > t) \leq \exp\left(\frac{-t^2}{2n}\right)$ \cite{janson2002concentration}. Therefore, using Azuma's inequality and the fact that sampling has Bernoulli distribution with parameter $p$, it can be seen that with probability  at most $\exp(-\frac{t^2}{2n_i})$ the number of observed entries is less than $n_i p - t$. Now, by setting $t$ equal to $\frac{n_i}{c}$ the proof is complete.
\end{proof}

\begin{corollary}\label{azuma}
Consider a vector with $n_i$ entries where each entry is observed with  probability  $p$  independently from the other entries. If $p > p^{\prime} = \frac{2r_i}{n_i} + \frac{1}{\sqrt[4]{n_i}}$, with probability  at least $\left(1-\exp(-\frac{\sqrt{n_i}}{2})\right)$, more than $2r_i$ entries are observed. Similarly, if $p > p^{\prime\prime} = \frac{12\log({\frac{en_i}{\epsilon}})}{n_i} + \frac{1}{\sqrt[4]{n_i}}$, with probability  at least $\left(1-\exp(-\frac{\sqrt{n_i}}{2})\right)$, more than $12\log({\frac{en_i}{\epsilon}})$ entries are observed.
\end{corollary}
\begin{proof}
Both parts follow by setting  $c=\sqrt[4]{n_i}$ in Lemma \ref{xgfh}.
\end{proof}

The following theorem uses the matrix result for single rank constraint and provides a sufficient condition on the sampling pattern for finite completabilitily of tensor. We will assume that  
\begin{eqnarray}
r_i \leq  \frac{n_i}{6}, \quad \   
N_{-i} \geq  r_i(n_i-r_i)  \quad \ \text{for} \ \ i=1,2,\dots,d \label{ass2}.
\end{eqnarray}

\begin{theorem}\label{thm11}
Consider a randomly sampled tensor $\mathcal{U}$ with Tucker rank $(r_1,\cdots,r_d)$. Suppose that the inequalities in \eqref{ass2} hold. Moreover, assume that the sampling probability satisfies
\begin{eqnarray}\label{comp1}
p > \min_{ 1  \leq i \leq d} \ \left( \max \left( \frac{2r_i}{n_i} , \frac{12\log({\frac{en_i}{\epsilon}})}{n_i} \right) + \frac{1}{\sqrt[4]{n_i}} \right). \label{ass5}
\end{eqnarray}
Then, there are only finitely many completions of the sampled tensor $\mathcal{U}$ with the given rank vector with probability  at least $(1- 2\epsilon)$.
\end{theorem}

\begin{proof}
Denote ${\ell}= \arg \min_{ 1  \leq i \leq d} \ \left( \max \left( \frac{2r_i}{n_i} , \frac{12\log({\frac{en_i}{\epsilon}})}{n_i} \right) + \frac{1}{\sqrt[4]{n_i}} \right)$. By \eqref{ass5} and using Corollary \ref{azuma}, with  probability  almost one the number of observed entries at each column of $\mathbf{U}_{({\ell})}$ is at least $\max\{2r_{\ell},12\log({\frac{en_{\ell}}{\epsilon}})\}$. Then, with \eqref{ass2} and using \cite[Theorem 3]{charact}, it follows that there are finitely many matrix completions for the observed entries of $\mathbf{U}_{({\ell})}$ with rank $r_{\ell}$. As a result, there are finitely many tensors with the ${\ell}$-th component of rank being $r_{\ell}$ that agree with the observed entries. The proof is complete since the set of tensors that are a completion with the rank vector $(r_1,\dots,r_d)$ is a subset of the set of tensors whose finiteness is shown above.
\end{proof}

\subsection{Lower Bound on Sampling Probability based on Analysis on Tucker Manifold}

We are interested in taking advantage of Theorem \ref{thm5} to obtain another lower bound on the sampling probability for finite completability. We assume $j=1$, and therefore the constraint tensor $\mathbf{\breve{\Omega}}$ is a second-order tensor, i.e., an ${n_1 \times K_1}$ matrix.

\begin{lemma}\label{minsamp}
Consider an arbitrary set $\Omega^{\prime}_{(1)}$ of $n_1 -1$ columns of $\Omega_{(1)}$ (first matricization of $\Omega$). Assume that each column of $\Omega_{(1)}$ includes at least $l$ nonzero entries, where 
\begin{eqnarray}\label{minlforset1}
l > 6 \ \log \left( {n_1} \right) + 2 \ \log \left( \frac{k}{\epsilon} \right) + 4. 
\end{eqnarray}
Then, with probability  at least $1-\frac{\epsilon}{k}$, every subset $\Omega^{\prime \prime}_{(1)}$ of columns of $\Omega^{\prime}_{(1)}$ satisfies 
\begin{eqnarray}\label{proper1}
m_{2}(\Omega^{\prime \prime}_{(1)}) -1 \geq t,
\end{eqnarray}
where $t$ is the number of columns of $\Omega^{\prime \prime}_{(1)}$ and $m_{2}(\Omega^{\prime \prime}_{(1)})$ is the number of nonzero rows of $\Omega^{\prime \prime}_{(1)}$ (observe that second matricization of a matrix is its transpose).
\end{lemma}

\begin{proof}
The proof is similar to the proof of \cite[Lemma $9$]{charact} with some delicate modifications to improve the result for this case. Note that \eqref{minlforset1} can be rewritten as
\begin{eqnarray}
l &>& 2  \ \log \left( \frac{n_1  k}{ \epsilon}\right) + 4 \ \log (n_1)  + 4 \label{ineqrefl2} \\
&>& \log \left( \frac{n_1 e^2 k }{ \epsilon} \right) + 2. \label{ineqrefl1} 
\end{eqnarray}
 
Define $\mathcal{E}$ as the event that for some submatrix $\Omega^{\prime \prime}_{(1)} \in \mathbb{R}^{n_1 \times t}$ of the matrix $\Omega^{\prime}_{(1)}$   \eqref{proper1} does not hold. We are interested in finding an upper bound on the probability  of $\mathcal{E}$. Then, from the proof of \cite[Lemma $9$]{charact} and by setting $r=1$ in inequalities $(12)$ and $(13)$ in \cite{charact}, we have
\begin{eqnarray}\label{ineq}
P(\mathcal{E}) < \sum_{n=l}^{\frac{n_1}{2}} { n_1\choose{n}}^{2} \left( \frac{n}{n_1} \right) ^{ln_1} + \sum_{n=1}^{\frac{n_1}{2}} { n_1\choose{n_1-n}}^{2} \left( \frac{n_1-n}{n_1} \right) ^{l(n_1 - n)}.
\end{eqnarray}
Note that 
\begin{eqnarray}\label{ineqproofup}
{ s\choose{q}} = \frac{s(s-1)\dots(s-q+1)}{q!} \leq \frac{s^q}{q!} \leq \left(\frac{se}{q}\right)^{q},
\end{eqnarray}
where the last inequality holds since $e^q = \sum_{i=0}^{\infty} \frac{q^i}{i!} \leq \frac{q^q}{q!}$. Applying \eqref{ineqproofup} to each term of the first summation of \eqref{ineq} we obtain
\begin{eqnarray}
{ n_1\choose{n}}^{2} \left( \frac{n}{n_1} \right) ^{ln_1} < e^{2n} \left(\frac{n}{n_1}\right)^{(l-2)n} \leq \left( e^2  2^{-l+2} \right)^{n} <  \frac{\epsilon}{n_1  k},
\end{eqnarray}
where the last inequality follows from \eqref{ineqrefl1}. This directly results that the first summation in \eqref{ineq} is less than $\frac{\epsilon}{ 2k}$. 

Moreover, again by applying \eqref{ineqproofup} to each term of the second summation in \eqref{ineq} we obtain
\begin{align}
{ n_1\choose{n_1-n}}^{2} \left( \frac{n_1-n}{n_1} \right) ^{l(n_1 - n)} =& { n_1\choose{n}}^{2} \left( \frac{n_1-n}{n_1} \right) ^{l(n_1 - n)} < (n_1e)^{2n} \left(\frac{n_1 -n}{n_1}\right)^{\frac{ln_1}{2}} \nonumber \\
\leq& (n_1e)^{2n} e^{-\frac{ln}{2}} = \left( e^{2\log n_1 + 2 - \frac{l}{2}} \right) ^ n < \frac{\epsilon}{n_1  k},
\end{align}
where the last inequality follows from \eqref{ineqrefl2}. This directly results that the second summation in \eqref{ineq} is less than $\frac{\epsilon}{2k}$ and that completes the proof.
\end{proof}

\begin{lemma}\label{minsamp2}
Consider an arbitrary set $\Omega^{\prime}_{(1)}$ of $n_1 $ columns of $\Omega_{(1)}$ (first matricization of $\Omega$). Assume that each column of $\Omega_{(1)}$ includes at least $l$ nonzero entries, where $l > 6 \ \log \left( n_1 \right) + 2 \ \log \left( \frac{2k}{\epsilon} \right) + 4$. Then, with probability  at least $1-\frac{\epsilon}{2 k}$, every subset $\Omega^{\prime \prime}_{(1)}$ of columns of $\Omega^{\prime}_{(1)}$ satisfies 
\begin{eqnarray}\label{proper2}
m_{2}(\Omega^{\prime \prime}_{(1)})  \geq t,
\end{eqnarray}
where $t$ is the number of columns of $\Omega^{\prime \prime}_{(1)}$ and $m_{2}(\Omega^{\prime \prime}_{(1)})$ is the number of nonzero rows of $\Omega^{\prime \prime}_{(1)}$.
\end{lemma}

\begin{proof}
By setting $r=0$ in inequalities $(12)$ and $(13)$ in \cite{charact}, we have
\begin{eqnarray}\label{ineqsimilar}
P(\mathcal{E}) < \sum_{n=l}^{\frac{n_1}{2}} { n_1\choose{n}}^{2} \left( \frac{n}{n_1} \right) ^{l(n_1+1)} + \sum_{n=1}^{\frac{n_1}{2}} { n_1\choose{n_1-n}}^{2} \left( \frac{n_1-n}{n_1} \right) ^{l(n_1 - n+1)}.
\end{eqnarray}
Since $n < n_1$ it then follows that $P(\mathcal{E}) < {\rm RHS}$ of \eqref{ineqsimilar} $< {\rm RHS}$ of \eqref{ineq} $< \frac{\epsilon}{2k}$, where the last inequality follows from Lemma \ref{minsamp}.
\end{proof}

The next theorem gives a lower bound on the number of needed samples to ensure finite completability for tensor $\mathcal{U}$. More specifically, we show under some mild assumptions if \eqref{minlfinite} holds, all  conditions and assumptions in the statement of Theorem \ref{thm5} hold with high probability, and therefore the tensor is finitely completable.

\begin{theorem}\label{thm6}
Assume that $\sum_{i=2}^{d} \left( n_ir_i \right) < n_{2}\dots n_d$, $\sum_{i=2}^{d} r_i^2 \leq \Pi_{i=2}^{d} r_i$,  $\Pi_{i=2}^{d}n_i \geq n_1  \Pi_{i=2}^{d} r_i - \sum_{i=2}^{d} r_i^2$, and Assumption $B_1$ hold. Furthermore, assume that each column of $\mathbf{U}_{(1)}$ includes at least $l$ observed entries, where
\begin{eqnarray}\label{minlfinite}
l > 6 \ \log \left( {n_1} \right) + 2 \max \left\{ \log \left( \frac{2 \sum_{i=2}^{d} r_i^2}{\epsilon}\right) ,  \log \left( \frac{2\Pi_{i=2}^{d} r_i - 2\sum_{i=2}^{d} r_i^2}{\epsilon} \right) \right\} + 4.
\end{eqnarray}
Then, with probability  at least $1-\epsilon$, there exist only finitely many completions, given the rank components $r_2,\dots,r_d$.
\end{theorem}

\begin{proof}
Since each column of  $\Omega_{(1)}$ includes at least one nonzero entry, there exist $\sum_{i=2}^{d} n_ir_i$ entries in different columns such that they satisfy Assumption $A_1$ (as mentioned in Remark \ref{remthm5}). Also, $\Omega_{(1)}$ has $\Pi_{i=2}^{d}n_i$ columns and by assumption $\Pi_{i=2}^{d}n_i \geq n_1  \Pi_{i=2}^{d} r_i - \sum_{i=2}^{d} r_i^2 $, therefore there exist $\Pi_{i=2}^{d} r_i $ disjoint sets $\mathcal{S}_1 , \ldots, \mathcal{S}_{\Pi_{i=2}^{d} r_i}$ such that each $\mathcal{S}_{\ell}$ consists of $n_1 -1$ columns of $\Omega_{(1)}$ for $1 \leq {\ell} \leq  \sum_{i=2}^{d} r_i^2$ and each $\mathcal{S}_{\ell}$ consists of $n_1$ columns of $\Omega_{(1)}$ for $\sum_{i=2}^{d} r_i^2+1 \leq {\ell} \leq \Pi_{i=2}^{d} r_i$.

%In the following we show $\mathcal{S}_{\ell}$ satisfies \eqref{proper1} for $1 \leq {\ell} \leq  \sum_{i=2}^{d} r_i^2$ and $\mathcal{S}_{\ell}$ satisfies \eqref{proper2} for $\sum_{i=2}^{d} r_i^2+1 \leq {\ell} \leq \Pi_{i=2}^{d} r_i$, with high probability. Then, we consider the union of all of them for $1 \leq {\ell} \leq  \Pi_{i=2}^{d} r_i$ and show that with high probability it satisfy condition (i) and (ii) in the statement of Theorem \ref{thm5}. 

Note that by \eqref{minlfinite}, we have $l > 6 \ \log \left( {n_1} \right) + 2  \log \left( \frac{2 \sum_{i=2}^{d} r_i^2}{\epsilon}\right) + 4$. Using Lemma \ref{minsamp}, with $k=2 \sum_{i=2}^{d} r_i^2$, for each $1 \leq {\ell}  \leq \sum_{i=2}^{d} r_i^2 $, with probability  at least $1-\frac{\epsilon}{ 2 \sum_{i=2}^{d} r_i^2}$, \eqref{proper1} holds for every subset of columns of $\mathcal{S}_{\ell}$. Hence, with probability  at least $1-\frac{\epsilon}{2}$, \eqref{proper1} holds for every subset of columns of $\mathcal{S}_{\ell}$ for $1 \leq {\ell}  \leq \sum_{i=2}^{d} r_i^2 $, simultaneously. According to Remark \ref{relatclarify} and Lemma \ref{Omega} below and by setting $r=1$, as a subset of columns $\mathcal{S}_{\ell}$ of $\Omega_{(1)}$ satisfies  \eqref{proper1}, there exist a subset of columns $\mathcal{\breve{S}}_{\ell}$ of the constraint matrix $\mathbf{\breve{\Omega}}$ (corresponding columns to the columns of $\mathcal{S}_{\ell}$) that satisfies \eqref{proper1} as well, for $1 \leq {\ell}  \leq \sum_{i=2}^{d} r_i^2 $. Denote  $\mathcal{\breve{S}}^{(0)}= \cup_{\ell=1}^{\sum_{i=2}^{d} r_i^2}  \ \mathcal{\breve{S}}_{\ell}$.

Consider any subset of columns of $\mathcal{\breve{S}}^{(0)}$ and denote it by $\mathcal{\breve{S}}^{\prime}$. Let $\mathcal{\breve{S}}_{\ell}^{\prime} $ denote the columns of $\mathcal{\breve{S}}^{\prime}$ that belong in $\mathcal{\breve{S}}_{\ell}$ and, without loss of generality, assume that $|\mathcal{\breve{S}}_1^{\prime}| \geq \cdots \geq |\mathcal{\breve{S}}_{\sum_{i=2}^{d} r_i^2 }^{\prime}|$, where $|\mathcal{\breve{S}}_{\ell}^{\prime}|$ denotes the number of columns of $\mathcal{\breve{S}}_{\ell}^{\prime}$. As a result, we obtain 
\begin{eqnarray}
|\mathcal{\breve S}^{\prime}| = \sum_{\ell=1}^{\sum_{i=2}^{d} r_i^2} |\mathcal{\breve S}_{\ell}^{\prime}| \leq \left( \sum_{i=2}^{d} r_i^2 \right) |\mathcal{\breve S}_1^{\prime}| \leq \left( \sum_{i=2}^{d} r_i^2 \right) (m_2(\mathcal{\breve S}_1^{\prime})-1) \leq \left( \sum_{i=2}^{d} r_i^2 \right) (m_2(\mathcal{\breve S}^{\prime})-1),
\end{eqnarray}
and consequently
\begin{eqnarray}\label{setonein}
|\mathcal{\breve S}^{\prime}|  \leq \left( \sum_{i=2}^{d} r_i^2 \right) m_2(\mathcal{\breve S}^{\prime}) - \sum_{i=2}^{d} r_i^2, \ \quad \quad  \forall \ \mathcal{\breve S}^{\prime} \subseteq \mathcal{\breve{S}}^{(0)}.
\end{eqnarray}

Moreover, by \eqref{minlfinite}, we have  $l > 6 \ \log \left( {n_1} \right) + 2  \log \left( \frac{2\Pi_{i=2}^{d} r_i - 2\sum_{i=2}^{d} r_i^2}{\epsilon} \right) + 4$. Using Lemma \ref{minsamp2}, with $k=2\Pi_{i=2}^{d} r_i-2\sum_{i=2}^{d} r_i^2$, for each $ \sum_{i=2}^{d} r_i^2 + 1 \leq \ell  \leq \Pi_{i=2}^{d} r_i $, with probability  at least $1-\frac{\epsilon}{ 2\Pi_{i=2}^{d} r_i-2\sum_{i=2}^{d} r_i^2}$, \eqref{proper2} holds for every subset of columns of $\mathcal{ S}_{\ell}$. As a result, with probability  at least $1-\frac{\epsilon}{2}$, \eqref{proper2} holds for every subset of columns of $\mathcal{ S}_{\ell}$ for $\sum_{i=2}^{d} r_i^2 + 1 \leq {\ell}  \leq \Pi_{i=2}^{d} r_i$, simultaneously. According to Remark \ref{relatclarify} and Lemma \ref{Omega} below and by setting $r=0$, as a subset of columns $\mathcal{S}_{\ell}$ of $\Omega_{(1)}$ satisfies  \eqref{proper2}, the corresponding subset of columns $\mathcal{\breve S}_{\ell}$ of the constraint matrix $\mathbf{\breve{\Omega}}$  satisfies \eqref{proper2} as well, $\sum_{i=2}^{d} r_i^2 + 1 \leq {\ell}  \leq \Pi_{i=2}^{d} r_i$. Denote  $\mathcal{\breve{S}}^{(1)}= \cup_{\ell=\sum_{i=2}^{d} r_i^2 + 1}^{\Pi_{i=2}^{d} r_i}  \ \mathcal{\breve{S}}_{\ell}$. Similarly, we have
\begin{eqnarray}\label{settwoin}
|\mathcal{\breve S}^{\prime}|  \leq \left( \Pi_{i=2}^{d} r_i - \sum_{i=2}^{d} r_i^2 \right) m_2(\mathcal{\breve S}^{\prime}), \ \quad \quad  \forall \ \mathcal{\breve S}^{\prime} \subseteq \mathcal{\breve{S}}^{(1)}.
\end{eqnarray}

For any subset of columns $\mathbf{\breve{\Omega}}^{\prime\prime}$ of $\mathbf{\breve{\Omega}}^{\prime} = \mathcal{\breve S}^{(0)} \cup \mathcal{\breve S}^{(1)}$ we define $\mathcal{\breve S}^{(0)^{\prime}}$ and $\mathcal{\breve S}^{(1)^{\prime}}$ as the set of columns of $\mathbf{\breve{\Omega}}^{\prime \prime}$ that belong to $\mathcal{\breve S}^{(0)}$ and $\mathcal{\breve S}^{(1)}$, respectively. Observe that as $\mathbf{\breve{\Omega}}^{\prime}$ has exactly  $ \sum_{i=1}^{\Pi_{i=2}^{d} r_i} |\mathcal{S}_i| = n_1  \Pi_{i=2}^{d} r_i - \sum_{i=2}^{d} r_i^2$ columns, condition (i) of Theorem \ref{thm5} for $j=1$ is satisfied. Then, by \eqref{setonein}, \eqref{settwoin} and the assumption that $\sum_{i=2}^{d} r_i^2 \leq \Pi_{i=2}^{d} r_i$, we have (recall $|\mathbf{\breve{\Omega}}^{\prime \prime}|$ denotes the number of columns of $\mathbf{\breve{\Omega}}^{\prime \prime}$ here)
\begin{align}
 |\mathbf{\breve{\Omega}}^{\prime \prime}| =& |\mathcal{\breve S}^{(0)^{\prime}}| + |\mathcal{\breve S}^{(1)^{\prime}}| \leq \left( \sum_{i=2}^{d} r_i^2 \right) m_2(\mathcal{\breve S}^{(0)^{\prime}})-\sum_{i=2}^{d} r_i^2 + \left( \Pi_{i=2}^{d} r_i - \sum_{i=2}^{d} r_i^2 \right) m_2(\mathcal{\breve S}^{(1)^{\prime}}) \\ \nonumber
 \leq& \left( \Pi_{i=2}^{d} r_i  \right) m_2(\mathbf{\breve{\Omega}}^{\prime \prime}) -\sum_{i=2}^{d} r_i^2 \leq \left( \Pi_{i=2}^{d} r_i  \right) m_2(\mathbf{\breve{\Omega}}^{\prime \prime}) -g_{2}(m_2(\mathbf{\breve{\Omega}}^{\prime \prime})). 
\end{align}
Therefore, any subset of columns of $\mathbf{\breve{\Omega}}^{\prime}$ satisfies \eqref{ineqp} (condition (ii) of Theorem \ref{thm5} for $j=1$), with probability  at least $1- \epsilon$. Then, according to Theorem \ref{thm5}, with probability  at least $1- \epsilon$, there are only finitely many completions that fit in the sampled tensor $\mathcal{U}$, given the rank components $r_2,\dots,r_d$.
\end{proof}

\begin{remark}
{\rm In the case that only a subset of rank components are given, e.g, $r_{j+1},\dots,r_d$, we can treat the first $j$ dimensions of the tensor as one single dimension, and therefore the above result is still applicable. In particular, assume that each column of $\widetilde{\mathbf{U}}_{(j)}$ ($j$-th unfolding) includes at least $l$ observed entries, where (recall $N_j=n_1\dots n_j$)
\begin{eqnarray}\label{minlfinitez}
l > 6 \ \log \left( {N_j} \right) + 2 \max \left\{ \log \left( \frac{2 \sum_{i=j+1}^{d} r_i^2}{\epsilon}\right) ,  \log \left( \frac{2\Pi_{i=j+1}^{d} r_i - 2\sum_{i=j+1}^{d} r_i^2}{\epsilon} \right) \right\} + 4.
\end{eqnarray}
Given that $\sum_{i=j+1}^{d} r_i^2 \leq \Pi_{i=j+1}^{d} r_i$,  $\Pi_{i=j+1}^{d}n_i \geq N_j  \Pi_{i=j+1}^{d} r_i - \sum_{i=j+1}^{d} r_i^2$, and Assumption $B_j$ hold, with probability  at least $1-\epsilon$, there exist only finitely many completions, given the rank components $r_{j+1},\dots,r_d$.} {  \hfill  \qedsymbol}
\end{remark}

\begin{remark}
{\rm Assume that there exist only finitely many completions, given the rank components $r_{j+1},\dots,r_d$, where $j \geq 1$. Then, there exist only finitely many completions of rank $(r_1,\dots,r_d)$.} {  \hfill  \qedsymbol}
\end{remark}

As in the proof of Theorem \ref{thm6} we referred to Lemma \ref{Omega}, we need to show that if a subset of columns of $\Omega_{(1)}$ satisfies \eqref{proper1} or \eqref{proper2}, the same property holds for a corresponding subset of columns of the constraint matrix $\mathbf{\breve{\Omega}}$.

\begin{remark}\label{relatclarify}
{\rm Recall that for each column of $\Omega_{(1)}$ that has $k+r$ nonzero entries and $r$ of them have been used to obtain $\mathbb{T}$, there exist $k$ columns in the constraint matrix $\mathbf{\breve{\Omega}}$, each with exactly $r+1$ nonzero entries.

Recall that as we showed after stating Assumption $A_j$, we can choose $\sum_{i=2}^{d} n_ir_i$ nonzero entries of $\Omega_{(1)}$ to obtain $\mathbb{T}$, i.e., to satisfy Assumption $A_1$ such that in each column of $\Omega_{(1)}$ either we choose one or zero nonzero entry. In that case, each column of $\mathbf{\breve{\Omega}}$ includes $0+1$ or $1+1$ nonzero entries. As a result, in the proof of Theorem \ref{thm6}, we only need Lemma \ref{Omega} for $r=0$ and $r=1$. However the general statement is needed to complete a proof in \cite{charact}, as we explain in Remark \ref{missingpart}.

Consequently, the matrix $\mathbf{\breve{\Omega}}^{\prime}$ defined in Lemma \ref{Omega} for $r=0$ and $r=1$ corresponds to the columns of the constraint matrix $\mathbf{\breve{\Omega}}$. } {  \hfill  \qedsymbol}
\end{remark}

\begin{lemma}\label{Omega}
Let $r$ be a given nonnegative integer. Assume that there exists an $n_1 \times (n_1 -r)$ matrix $\Omega^{\prime}$ composed  of $n_1-r$ columns of $\Omega_{(1)}$ such that each column of $\Omega^{\prime}$ includes at least $r+1$ nonzero entries and satisfies the following property:{\footnote{This property is exactly property (ii) in \cite{charact}.}}
\begin{itemize}
\item Denote an $n_1 \times t$  matrix (for any $1 \leq t \leq n_1-r$) composed of any $t$ columns of $\Omega^{\prime}$ by $\Omega^{\prime \prime}$. Then 
\begin{eqnarray}\label{proper233}
m_{2}(\Omega^{\prime \prime}) -r  \geq t,
\end{eqnarray}
where  $m_{2}(\Omega^{\prime \prime})$ is the number of nonzero rows of $\Omega^{\prime \prime}$.
\end{itemize}
Then, there exists an $n_1 \times (n_1 -r)$ matrix $\mathbf{\breve{\Omega}}^{\prime}$ such that: each column has exactly $r+1$ entries equal to one, and if $\mathbf{\breve{\Omega}}^{\prime}(x,y)=1$ then we have $\Omega^{\prime}(x,y)=1$. Moreover, $\mathbf{\breve{\Omega}}^{\prime}$ satisfies the above-mentioned property.
\end{lemma}

%\begin{center}
%\begin{tikzpicture}[->,>=stealth',shorten >=1pt,auto,node distance=1.9cm,
%                    thick,main node/.style={circle,draw,font=\sffamily\small\bfseries}]
%
%  \node[main node] (1) {1};
%  \node[main node] (2) [below of=1] {2};
%  \node[main node] (3) [below of=2] {3};
%  \node[main node] (4) [below of=3] {4};
%  \node[main node] (5) [right of=1] {5};
%  \node[main node] (6) [below of=5] {6};
%  \node[main node] (7) [below of=6] {7};
%  \node[main node] (8) [below of=7] {8};
%
%  \path[every node/.style={font=\sffamily\small}]
%    (1) edge node [bend left] {} (4)
%        edge [bend right] node[left] {} (2)
%    (2) edge node [right] {} (1)
%        edge node {} (4)
%        edge [bend right] node[left] {} (3)
%    (3) edge node [right] {} (2)
%        edge [bend right] node[right] {} (4)
%    (4) edge node [left] {} (3)
%        edge [bend right] node[right] {} (1);
%\end{tikzpicture}
%\end{center}

\begin{proof}
Let $\Omega^{\prime}$ denote a $n_1 \times (n_1-r)$ matrix that consists of $n_1-r$ columns of $\Omega_{(1)}$ and satisfies the mentioned property in the statement of the lemma. Note that the given property in the statement of the lemma is equivalent to the following statement
\begin{itemize}
\item The matrix obtained by choosing any subset of the columns of $\Omega^{\prime}$ includes at least $r+t$ nonzero rows, where $t$ is the number of the selected columns of $\Omega^{\prime}$.
\end{itemize}

Consider a bipartite graph $\mathcal{G}$ with the two sets of nodes $\mathcal{T}_1$ with $n_1-r$ nodes corresponding to the columns and $\mathcal{T}_2$ with $n_1$ nodes corresponding to the rows. We add an edge between the $i$-th node of $\mathcal{T}_1$ and the $i^{\prime}$-th node of $\mathcal{T}_2$ if and only if the $(i,i^{\prime})$-th entry of $\Omega^{\prime}$ is equal to one. For a set of nodes $\mathcal{S}$, let $N(\mathcal{S})$ denote the set of their neighbors. Note that according to the assumption, any subset of the columns of $\Omega^{\prime}$ includes at least $r+t$ nonzero rows ($r+t$ edges in the defined graph), where $t$ is the number of the selected columns of $\Omega_{\tau}$. Therefore, every subset of nodes of $\mathcal{T}_1$ has the following property
\begin{eqnarray}\label{Hall}
N(\mathcal{S})  \geq  |\mathcal{S}| + r.
\end{eqnarray}
The statement of the lemma is equivalent to proving that there exists a spanning subgraph $\mathcal{G}^{\prime}$ of $\mathcal{G}$ such that every node of $\mathcal{T}_1$ is of degree $r+1$, and also for each subset of nodes of $\mathcal{T}_1$,  \eqref{Hall} holds in $\mathcal{G}^{\prime}$.

First consider the scenario that $r=0$. Then, according to Hall's theorem \cite{kierstead1983effective}, there is a perfect matching (since inequality \eqref{Hall} holds). Observe that a perfect matching in the described graph is exactly equivalent to what we are looking for when $r=0$. For other values of $r > 0$, the proof is a direct result of Theorem \ref{genHall} which is a {\bf generalization} of the {\bf Hall's theorem}. An example of such spanning subgraph for a graph with $n_1=6$ and $r=1$ is shown in Figure \ref{fig5}. In Figure \ref{fig:2.1} a bipartite graph is given such that \eqref{Hall} holds for $r=1$. Figure \ref{fig:2.2} gives an spanning subgraph of the given graph in Figure \ref{fig:2.1} such that degree of each node in $\mathcal{T}_1$ is equal to $2$ and still \eqref{Hall} holds for $r=1$.
\end{proof}

\begin{figure}[htbp]
\centering
\subfigure[The original graph.]{
	\includegraphics[width=5cm]{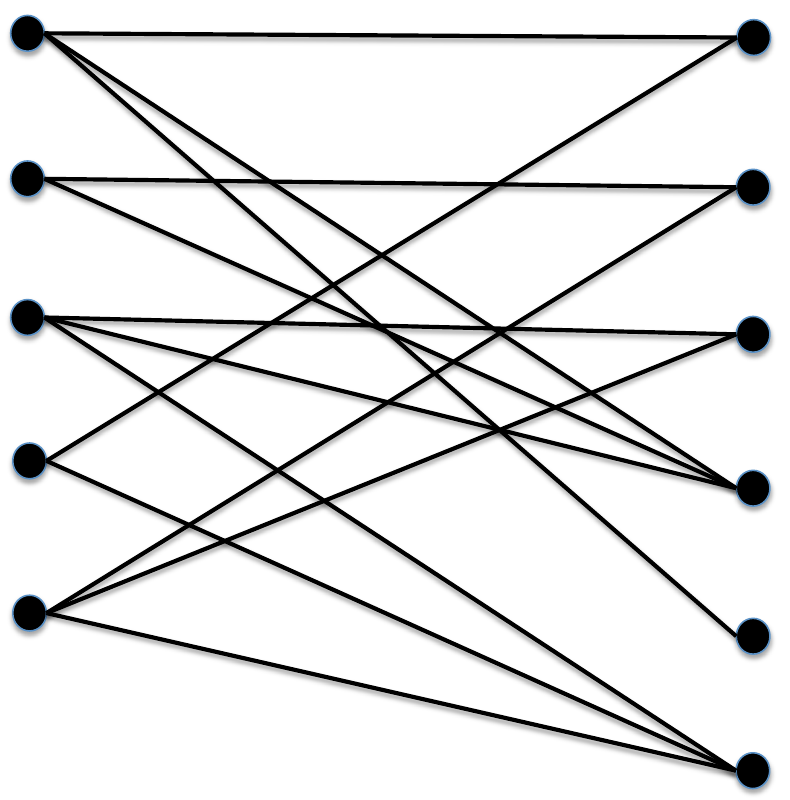}
    \label{fig:2.1}
}
\subfigure[The spanning subgraph.]{
	\includegraphics[width=5cm]{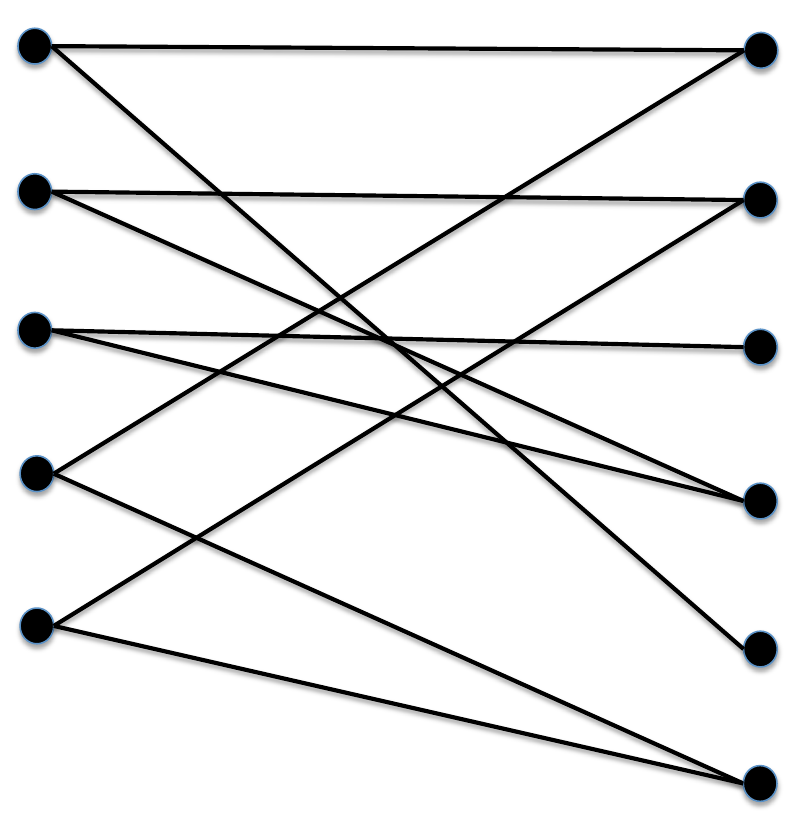}
    \label{fig:2.2}
}
\caption[Optional caption for list of figures]{An example of the corresponding bipartite graph with $r=1$ and $n_1=6$.}
\label{fig5}
\end{figure}

\begin{theorem}\label{genHall}
{(Generalized Hall's Theorem)} Consider a bipartite graph $\mathcal{G}$ with two sets of nodes, $\mathcal{T}_1$ with $x$ nodes and $\mathcal{T}_2$ with $x+r$ nodes, where $x, r \in \mathbb{N}$. Suppose that for each subset $\mathcal{S}$ of $\mathcal{T}_1$ the following inequality holds
\begin{eqnarray}\label{neigh1}
|N(\mathcal{S})|  \geq  |\mathcal{S}| + r.
\end{eqnarray}
Then, there exists a spanning subgraph $\mathcal{G}^{\prime}$ of $\mathcal{G}$ such that every node of $\mathcal{T}_1$ is of degree $r+1$, and also for each subset of nodes of $\mathcal{T}_1$, the inequality \eqref{neigh1} holds in $\mathcal{G}^{\prime}$ as well.
\end{theorem}
\begin{proof}
The proof is given in Appendix \ref{appB}.
\end{proof}

\begin{corollary}
Assume that $\sum_{i=j+1}^{d} r_i^2 \leq \Pi_{i=j+1}^{d} r_i$,  $\Pi_{i=j+1}^{d}n_i \geq N_j  \Pi_{i=j+1}^{d} r_i - \sum_{i=j+1}^{d} r_i^2$, and Assumption $B_j$ hold. Furthermore, assume that we observe each entry of  $\mathcal{U}$ with probability  $p$, where
\begin{eqnarray}\label{comp2}
p >  \frac{1}{N_j} \left( 6 \ \log \left( {N_j} \right) + 2 \log \left(  \max \left\{ \frac{2 \sum_{i=j+1}^{d} r_i^2}{\epsilon} ,   \frac{2\Pi_{i=j+1}^{d} r_i - 2\sum_{i=j+1}^{d} r_i^2}{\epsilon} \right\} \right) + 4  \right)+ \frac{1}{\sqrt[4]{N_j}}.
\end{eqnarray}
Then, with probability  at least $1-2\epsilon$, there are only finitely many completions for $\mathcal{U}$, given the rank components $r_{j+1},\dots,r_d$.
\end{corollary}

\begin{proof}
The proof is straightforward given Theorem \ref{thm6} and Corollary \ref{azuma}.
\end{proof}

\begin{remark}
{\rm In Theorem \ref{thm11} which is obtained by applying the analysis in \cite{charact}, the number of needed samples per column of the $i$-th matricization to guarantee finite completability with high probability is on the order of $\mathcal{O}(\max\{\log(n_i),r_i\})$. Therefore, the number of needed samples in total is on the order of $\mathcal{O}(N_{-i} \max\{\log(n_i),r_i\})$. On the other hand, in Theorem \ref{thm6} for the case of $d >2$ (order of the tensor is at least $3$), the number of needed samples per column for the $i$-th unfolding becomes $\mathcal{O}(\max\{\log(N_i),\log(r_{i+1}\dots r_d)\})$, and the number of needed samples in total is on the order of $\mathcal{O}(\frac{N}{N_i}  \max\{\log(N_i),\log(r_{i+1}\dots r_d)\})$, if $r_{i+1}\dots r_d < \frac{N}{{N_i}^{2}}$.

As an example, consider the case that $n_i=n$, $i=1,\dots,d$. Also, assume that $r_i=r$, $i=j+1,\dots,d$, where $r \leq \sqrt{n}$. The matricization analysis requires $\mathcal{O}(n^{d-1} \max\{\log(n),r\})$ samples for any $j$. However, according to Theorem \ref{thm6}, if $j=\frac{d}{3}$ it requires $\mathcal{O}(n^{\frac{2d}{3}} \log(n^{\frac{d}{3}}))$ samples for finite completability. Note that under the assumption $r_{j+1}\dots r_d < \frac{N}{{N_j}^{2}}$, i.e., $r^{d-j} < n^{d-2j}$, $j=\frac{d}{3}$ results in the best possible bound as $r = \sqrt{n}$ since increasing $j$ violates the assumption.} {  \hfill  \qedsymbol}
\end{remark}
%This is a significant improvement when $\log(N_i) \ll r_{i+1}\dots r_d < N_i$
\begin{remark}\label{missingpart}
{\rm We note that the proof in \cite[Theorem $3$]{charact} is incomplete and Lemma \ref{Omega} is needed to complete that proof.
 
First, observe that sampling pattern matrix $\mathbf{\Omega}$ represents the observed entries and it has at least $l$ nonzero elements per column. On the other hand, $\mathbf{\breve{\Omega}}$ represents the constraint matrix defined in \cite{charact} and it has exactly $(r+1)$ nonzero elements per column, where $r$ is the given rank.

Secondly, observe that in \cite[Lemma $1$]{charact} each $\mathbf{\hat{\Omega}}{\tau}$ is a subset of the constraint matrix $\mathbf{\breve{\Omega}}$ (not a subset of $\mathbf{\Omega}$).

Finally, in the proof and statement of \cite[Lemma $9$]{charact} the authors consider a subset of $\mathbf{\Omega}$. However, in the proof of \cite[Theorem $3$]{charact} they refer to \cite[Lemma $1$]{charact}, which considers the columns as the subsets of $\mathbf{\breve{\Omega}}$. Now, considering $\mathbf{\breve{\Omega}}$ instead of $\mathbf{\Omega}$ in \cite[Lemma $9$]{charact}, the equation below the equation ($11$) in \cite{charact} does not hold. This is because the assumption right below equation ($11$) in \cite{charact} that $\mathbf{\breve{\Omega}}$ has at least $l$ nonzero elements per column is incorrect (this is a property of $\mathbf{\Omega}$ but not necessarily $\mathbf{\breve{\Omega}}$).

As a result, a part of the proof in \cite[Theorem $3$]{charact} is missing which is Lemma \ref{Omega} in this paper.  } {  \hfill  \qedsymbol}
\end{remark}

\section{Deterministic and Probabilistic Guarantees for Unique Completability}\label{sec:uniq}

In this section, we first provide an example with exactly two completions to emphasize that finite completability does not necessarily result in unique completability. Then, we provide the conditions on the sampling pattern to guarantee unique completability.

\begin{example}\label{lemexfinuni}
Assume that the sampled matrix $\mathbf{U} \in \mathbb{R}^{5 \times 4}$ is given as the incomplete matrix on the left below:

\begin{center}
\begin{tabular}{ |c|c|c|c| } 
 \hline
$1$ & \ \ \ \  & $1$ & $-\frac{1}{2}$ \\ \hline
$-4$ & $2$ & $-1$ &  \\ \hline
$0$ & $1$ &  & $2$ \\ \hline
$1$ &  & $4$ &  \\ \hline
 & $4$ & $-2$ & $\frac{3}{2}$ \\
 \hline
\end{tabular} \ $\Rightarrow$ \ 
\begin{tabular}{ |c|c|c|c| } 
 \hline
$1$ & $-\frac{21}{32}$ & $1$ & $-\frac{1}{2}$ \\ \hline
$-4$ & $2$ & $-1$ & $\frac{3}{4}$ \\ \hline
$0$ & $1$ & -$\frac{24}{5}$ & $2$ \\ \hline
$1$ & $-\frac{41}{32}$ & $4$ & $-\frac{7}{4}$ \\ \hline
$-8$ & $4$ & $-2$ & $\frac{3}{2}$ \\
 \hline
\end{tabular} and 
\begin{tabular}{ |c|c|c|c| } 
 \hline
$1$ & $-2$ & $1$ & $-\frac{1}{2}$ \\ \hline
$-4$ & $2$ & $-1$ & $-10$ \\ \hline
$0$ & $1$ & -$\frac{1}{2}$ & $2$ \\ \hline
$1$ & $-8$ & $4$ & $-\frac{25}{2}$ \\ \hline
$-\frac{39}{21}$ & $4$ & $-2$ & $\frac{3}{2}$ \\
 \hline
\end{tabular}
\end{center}

Moreover, assume that $\text{rank}(\mathbf{U}) = 2$. In Appendix \ref{appC}, it is shown that there exist exactly two completions of $\mathbf{U}$ as given by the two complete matrices on the right.
\end{example}

The following assumptions is a stronger version of Assumption $A_j$ to ensure that there exists only one tuple $\mathbb{T}$ given the core tensor.

\noindent{\underline{\it{Assumption $A_j^+$}}}: Anywhere that this assumption is stated, there exist $\sum_{i=j+1}^{d} \left( n_i(r_i+1) \right)$ observed entries such that for any $ \mathcal{S}_i \subseteq \{1,\dots,n_i\}$ for $i \in \{j+1,\dots,d\}$, $\mathcal{U}^{(\mathcal{S}_{j+1},\dots,\mathcal{S}_{d})} $ includes at most $\sum_{i=j+1}^{d} |\mathcal{S}_i|(r_i+1)$ of the mentioned $\sum_{i=j+1}^{d} \left( n_i(r_i+1) \right)$ observed entries.

\begin{remark}
{\rm In Lemma \ref{Ber} we showed that Assumption $A_j$ results that tuple $\mathbb{T}$ can be determined finitely. Note that Assumption $A_j^+$ implies Assumption $A_j$ and therefore using only $\sum_{i=j+1}^{d} \left( n_ir_i \right)$ of the observed entries tuple $\mathbb{T}$ can be determined finitely. Moreover, observe that the rest of the sampled entries result in a set of polynomials that involve any of variables at least once and according to Note $5$, tuple $\mathbb{T}$ can be determined uniquely, with probability one. } {  \hfill  \qedsymbol}
\end{remark}

The following lemma is a re-statement of Lemma 25 in \cite{ashraphijuo4} (also an adaptation of Lemma 7 in \cite{charact}) and gives the conditions under which a subset of entries of the core tensor $\mathcal{C}$ can be determined uniquely.

\begin{lemma}\label{lemma3}
Assume that Assumption $B_j$ holds. Suppose that $\mathbf{\breve{\Omega}}^{\prime} \in \mathbb{R}^{n_1 \times n_2 \times \cdots \times n_j \times t}$ is a subtensor of the constraint tensor such that $\mathcal{P}(\mathbf{\breve{\Omega}}^{\prime})$ is minimally algebraically dependent. Then, for almost every $\mathcal{U}$, all variables that are involved in the set of polynomials $\mathcal{P}(\mathbf{\breve{\Omega}}^{\prime})$ can be uniquely determined.
\end{lemma}

\begin{proof}
According to Lemma \ref{lemma3new}, the number of involved variables in polynomials in $\mathcal{P}(\mathbf{\breve{\Omega}}^{\prime})=\{p_1,p_2,\dots,p_t\}$ is $t-1$ and are denoted by $\{x_1,\dots,x_{t-1}\}$. Moreover, as mentioned in the proof of Lemma \ref{lemma3new}, $\mathcal{P}(\mathbf{\breve{\Omega}}^{\prime}) \backslash p_i$ is a set of algebraically independent polynomials and the number of involved variables is $t-1$, $i=1,\dots,t$. Consider an arbitrary variable $x_1$ that is involved in polynomials in $\mathcal{P}(\mathbf{\breve{\Omega}}^{\prime})$ and without loss of generality, assume that $x_1$ is involved in $p_1$.

On the other hand, as mentioned all variables $\{x_1,\dots,x_{t-1}\}$ are involved in algebraically independent polynomials in $\mathcal{P}(\mathbf{\breve{\Omega}}^{\prime}) \backslash p_1$. Hence, according to Note $5$, tuple $(x_1,\dots,x_{t-1})$ can be determined finitely. Given that one of these finitely many tuples $(x_1,\dots,x_{t-1})$ satisfy polynomial equation $p_1$, according to Note $5$, with probability one there does not exist another tuple $(x_1,\dots,x_{t-1})$ among them that satisfies polynomial equation $p_1$. This is because with probability zero a tuple satisfy a polynomial equation in which the coefficients are chosen generically, and also the fact that the number of such tuples is finite.
\end{proof}

The next theorem provides two conditions on the sampling pattern to ensure unique completability. The first one is as the same as the condition proposed in Theorem \ref{thm5} for finite completability. The second condition takes advantage of finite completability and leads to the unique core tensor.

\begin{theorem}\label{thm:uniq}
Suppose that assumptions $A_j^{+}$ and $B_j$ hold. Also, assume that there exist two disjoint subtensors $\mathbf{\breve{\Omega}}^{\prime} \in \mathbb{R}^{n_1 \times n_2 \times \cdots \times n_j \times n}$ and $\mathbf{\breve{\Omega}}^{{\prime}}_0 \in \mathbb{R}^{n_1 \times n_2 \times \cdots \times n_j \times n_0}$ of the constraint tensor such that $ n = \left(\Pi_{i=1}^{j} n_i\right) \left(  \Pi_{i=j+1}^{d} r_i\right)  - \left( \sum_{i=j+1}^{d}  r_i^{2} \right)$ and $n_0 = \left(\Pi_{i=1}^{j} n_i\right)  - {\lfloor \frac{\sum_{i=j+1}^{d} r_i^2}{\Pi_{i=j+1}^{d} r_i} \rfloor}$ with the following conditions:

\noindent (i) For any $t \in \{1,\dots,n\}$ and any subtensor $\mathbf{\breve{\Omega}}^{\prime \prime} \in \mathbb{R}^{n_1 \times n_2 \times \cdots \times n_j  \times t}$ of the tensor $\mathbf{\breve{\Omega}}^{\prime}$, the following inequality holds
\begin{eqnarray}
\left(\Pi_{i=j+1}^{d} r_i \right)  m_{j+1}(\mathbf{\breve{\Omega}}^{\prime \prime}) - g_{j+1}(m_{j+1}(\mathbf{\breve{\Omega}}^{\prime \prime})) \geq t.
\end{eqnarray}
(ii) For any $t^{\prime} \in \{1,\dots,n_0\}$ and any subtensor $\mathbf{\breve{\Omega}}^{{\prime \prime}}_0 \in \mathbb{R}^{n_1 \times n_2 \times \cdots \times n_j \times  t^{\prime}}$ of the tensor $\mathbf{\breve{\Omega}}^{{\prime}}_0$, the following inequality holds 
\begin{eqnarray}\label{unisub}
 \left( \Pi_{i=j+1}^{d} r_i \right) m_{j+1}(\mathbf{\breve{\Omega}}^{\prime \prime}_0)  -  g_{j+1}(m_{j+1}(\mathbf{\breve{\Omega}}^{\prime \prime}_0)) \geq \left( \Pi_{i=j+1}^{d} r_i \right)t^{\prime} - \left( \sum_{i=j+1}^{d}  r_i^{2} \right) \left( t^{\prime} - n_0 +1\right)^+.
\end{eqnarray}
Then, for almost every $\mathcal{U}$ with probability one, there exists exactly one tensor that fits in the sampled tensor, and also $\text{rank}(\mathbf{U}_{(i)})=r_i$, $i = j+1,\dots,d$. Therefore there is a unique completion of the sampled tensor with rank of $(r_1,r_2,\ldots,r_d)$.
\end{theorem}

\begin{proof}
As mentioned Assumption $A_j^+$ results that $\mathbb{T}$ can be determined uniquely. In order to complete the proof it suffices to show the core tensor $\mathcal{C}$ can be determined uniquely as well. As defined earlier, $\mathcal{P}(\mathbf{\breve{\Omega}}^{\prime})$ and $\mathcal{P}(\mathbf{\breve{\Omega}}^{\prime}_0)$ denote the set of polynomials obtained from the sampled entries corresponding to $\mathbf{\breve{\Omega}}^{\prime}$ and $\mathbf{\breve{\Omega}}^{\prime}_0$, respectively. According to Theorem \ref{thm5}, $\mathcal{P}(\mathbf{\breve{\Omega}}^{\prime})$ results in finite completability of the sampled tensor $\mathcal{U}$, and also $n$ algebraically independent polynomials. Hence, adding any of the polynomials in $\mathcal{P}(\mathbf{\breve{\Omega}}^{\prime}_0)$ to $\mathcal{P}(\mathbf{\breve{\Omega}}^{\prime})$ results in an algebraically dependent set of polynomials. Using $n$ algebraically independent polynomials in $\mathcal{P}(\mathbf{\breve{\Omega}}^{\prime})$ and polynomials in $\mathcal{P}(\mathbf{\breve{\Omega}}^{\prime}_0)$, we show all entries of core tensor can be determined uniquely.

Let $p_0$ denote an arbitrary polynomial in $\mathcal{P}(\mathbf{\breve{\Omega}}^{\prime}_0)$. Consider $\mathcal{P}^{\prime}(p_0) \subset \mathcal{P}(\mathbf{\breve{\Omega}}^{\prime})$ such that $\mathcal{P}^{\prime}(p_0) \cup p_0$ is a minimally algebraically dependent set of polynomials. Let tuple $(x_1,\dots,x_j)$ denote the first $j$ coordinates of the corresponding observed entry to polynomial $p_0$. Then, according to Note $2$, all $\Pi_{i=j+1}^{d} r_i$ entries of core tensor $\mathcal{C}$ with the first $j$ coordinates as $(x_1,\dots,x_j)$ are involved in polynomial $p_0$, and therefore they can be determined uniquely, according to Lemma \ref{lemma3}. However, a subset of these $\Pi_{i=j+1}^{d} r_i$ entries of core tensor may be the elements of the proper structure which are known entries.

Similarly, repeating this procedure for any subtensor $\mathbf{\breve{\Omega}}^{{\prime \prime}}_0 \in \mathbb{R}^{n_1 \times n_2 \times \cdots \times n_j \times  t^{\prime}}$ of the tensor $\mathbf{\breve{\Omega}}^{{\prime}}_0$ results in $\left( \Pi_{i=j+1}^{d} r_i \right) t^{\prime}$ polynomials in terms of the $\left( \Pi_{i=j+1}^{d} r_i \right) m_{j+1}(\mathbf{\breve{\Omega}}^{\prime \prime}_0)  -  g_{j+1}(m_{j+1}(\mathbf{\breve{\Omega}}^{\prime \prime}_0))$ unknown entries of the core tensor. Observe that according to Note $5$, these  polynomials are algebraically independent if 
\begin{eqnarray}
\left( \Pi_{i=j+1}^{d} r_i \right) m_{j+1}(\mathbf{\breve{\Omega}}^{\prime \prime}_0)  -  g_{j+1}(m_{j+1}(\mathbf{\breve{\Omega}}^{\prime \prime}_0)) \geq \left( \Pi_{i=j+1}^{d} r_i \right) t^{\prime},
\end{eqnarray}
for any subtensor $\mathbf{\breve{\Omega}}^{{\prime \prime}}_0 \in \mathbb{R}^{n_1 \times n_2 \times \cdots \times n_j \times  t^{\prime}}$ of the tensor $\mathbf{\breve{\Omega}}^{{\prime}}_0$. In order to include $\text{dim}(\mathcal{C}) = \left(\Pi_{i=1}^{j} n_i\right) \left(  \Pi_{i=j+1}^{d} r_i\right)  - \left( \sum_{i=j+1}^{d}   r_i^{2} \right)$ algebraically independent polynomials, $\mathbf{\breve{\Omega}}^{{\prime}}_0$ should include at least $n_0$ columns otherwise for any subtensor $\mathbf{\breve{\Omega}}^{{\prime \prime}}_0 \in \mathbb{R}^{n_1 \times n_2 \times \cdots \times n_j \times  t^{\prime}}$ of the tensor $\mathbf{\breve{\Omega}}^{{\prime}}_0$ we have $ \text{dim}(\mathcal{C}) > \left( \Pi_{i=j+1}^{d} r_i \right)t^{\prime}$ since $t^{\prime} < n_0$. Therefore, inequality \eqref{unisub} results that all $n$ variables (unknown entries) of the core tensor can be determined uniquely.
\end{proof}

We are also interested to obtain a lower bound on the sampling probability, which ensures the unique completability. The following lemma leads to Lemma \ref{uniqsampprob} that characterizes the number of required sampled entries to ensure the unique completability with high probability.

\begin{lemma}\label{minsampuni}
Consider an arbitrary set $\Omega^{\prime}_{(1)}$ of $n_1 $ columns of $\Omega_{(1)}$ (first matricization of $\Omega$). Assume that each column of $\Omega_{(1)}$ includes at least $l$ nonzero entries, where $l > 6 \ \log \left( {n_1} \right) + 2 \ \log \left( \frac{n_1 k}{\epsilon} \right) + 4$. Then, with probability  at least $1-\frac{\epsilon}{k}$ every {\bf proper} subset $\Omega^{\prime \prime}_{(1)}$ of columns of $\Omega^{\prime}_{(1)}$ satisfies 
\begin{eqnarray}\label{properuni}
m_{2}(\Omega^{\prime \prime}_{(1)}) -1 \geq t,
\end{eqnarray}
where $t$ is the number of columns of $\Omega^{\prime \prime}_{(1)}$ and $m_{2}(\Omega^{\prime \prime}_{(1)})$ is the number of nonzero rows of $\Omega^{\prime \prime}_{(1)}$.
\end{lemma}

\begin{proof}
According to Lemma \ref{minsamp}, every subset $\Omega^{\prime \prime}_{(1)}$ of columns of $\Omega^{\prime}_{(1)}$ with $n_1 -1$ columns satisfies property  \eqref{properuni} with probability  at least $1-\frac{\epsilon}{n_1 k}$. There are exactly $n_1$ subsets with $n_1 -1 $ columns, and therefore property \eqref{properuni} holds for all of them with probability  at least $1-\frac{\epsilon}{k}$.
\end{proof}

\begin{lemma}\label{uniqsampprob}
Assume that $\sum_{i=2}^{d} \left( n_ir_i \right) < n_{2}\dots n_d$, $\sum_{i=2}^{d} r_i^2 \leq \Pi_{i=2}^{d} r_i$,  $\Pi_{i=2}^{d}n_i \geq n_1  (\Pi_{i=2}^{d} r_i +1) - \sum_{i=2}^{d} r_i^2$, and Assumption $B_1$ hold. Furthermore, assume that each column of $\mathbf{U}_{(1)}$  includes at least $l$ observed entries, where
\begin{eqnarray}\label{unisampnumb}
l > 6 \ \log \left( {n_1} \right) + 2 \max \left\{ \log \left( \frac{ \sum_{i=2}^{d} r_i^2}{\epsilon}\right) ,  \log \left( \frac{\Pi_{i=2}^{d} r_i - \sum_{i=2}^{d} r_i^2}{\epsilon} \right), \log \left( \frac{n_1}{\epsilon} \right) \right\} + 8.
\end{eqnarray}
Then, with probability  at least $1-\epsilon$, there exists only one completion of rank $(r_1,\cdots,r_d)$.
\end{lemma}

\begin{proof}
Observe that since $\Pi_{i=2}^{d}n_i \geq n_1 (\Pi_{i=2}^{d} r_i +1) - \sum_{i=2}^{d} r_i^2$ there exist $ n_1 \left(  \Pi_{i=2}^{d} r_i\right)  - \left( \sum_{i=2}^{d}  r_i^{2} \right)$ and $ n_1$ disjoint columns in $\Omega_{(1)}$ denoted by $\mathbf{\breve{\Omega}}^{{\prime}}$ and $\mathbf{\breve{\Omega}}^{{\prime}}_0$, respectively. In order to show unique completability, it suffices to show $\mathbf{\breve{\Omega}}^{{\prime}}$ and $\mathbf{\breve{\Omega}}^{{\prime}}_0$ satisfy properties (i) and (ii) in the statement of Theorem \ref{thm:uniq}, respectively, with probability  at least $1-\frac{\epsilon}{2}$. 

Note that having \eqref{unisampnumb}, it is easy to see that \eqref{minlfinite} holds with $\epsilon$ replaced by $\frac{\epsilon}{2}$. Therefore, according to Theorem \ref{thm6}, $\mathbf{\breve{\Omega}}^{{\prime}}$ satisfies property (i) in the statement of Theorem \ref{thm:uniq} with probability  at least $1-\frac{\epsilon}{2}$.

On the other hand, $\sum_{i=2}^{d} r_i^2 \leq \Pi_{i=2}^{d} r_i$ and according to Lemma \ref{minsampuni} (with $k=2$), with probability  at least $1-\frac{\epsilon}{2}$, any subset of columns of $\mathbf{\breve{\Omega}}^{{\prime}}_0$ satisfies \eqref{properuni}. For any {\bf proper} subset of column $\mathbf{\breve{\Omega}}^{{\prime \prime}}_0$ of $\mathbf{\breve{\Omega}}^{{\prime}}_0$, we have $m_{2}(\mathbf{\breve{\Omega}}^{{\prime \prime}}_0) -1 \geq t^{\prime}$ or equivalently ($t^{\prime}$ is the number of columns of $\mathbf{\breve{\Omega}}^{{\prime \prime}}_0$)
\begin{eqnarray}\label{inkajsnm}
 \left( \Pi_{i=j+1}^{d} r_i \right) m_{j+1}(\mathbf{\breve{\Omega}}^{\prime \prime}_0)  -  \Pi_{i=2}^{d} r_i \geq \left( \Pi_{i=j+1}^{d} r_i \right)t^{\prime},
\end{eqnarray}
which results in \eqref{unisub} as $g_{j+1}(m_{j+1}(\mathbf{\breve{\Omega}}^{\prime \prime}_0)) \leq \sum_{i=2}^{d} r_i^2 \leq \Pi_{i=2}^{d} r_i$ and $\left( t^{\prime} - n_0 +1\right)^+ = 0$. Hence, in order to show property (ii) in the statement of Theorem \ref{thm:uniq} holds with probability  at least $1-\frac{\epsilon}{2}$ (which completes the proof) we only need to show that  \eqref{unisub} holds when $\mathbf{\breve{\Omega}}^{{\prime \prime}}_0=\mathbf{\breve{\Omega}}^{{\prime}}_0$.

Note that $m_{j+1}(\mathbf{\breve{\Omega}}^{\prime \prime}_0) \leq m_{j+1}(\mathbf{\breve{\Omega}}^{\prime}_0)$. Therefore, as $m_{2}(\mathbf{\breve{\Omega}}^{{\prime \prime}}_0) -1 \geq t^{\prime}$ holds for any proper subset of column $\mathbf{\breve{\Omega}}^{{\prime \prime}}_0$ of $\mathbf{\breve{\Omega}}^{{\prime}}_0$, we have $m_{j+1}(\mathbf{\breve{\Omega}}^{\prime}_0) \geq n_1$ (by setting $t^{\prime}=n_1-1$). Then, \eqref{unisub} holds as $g_{j+1}(m_{j+1}(\mathbf{\breve{\Omega}}^{\prime}_0)) \leq \sum_{i=2}^{d} r_i^2$ and $\left( t^{\prime} - n_0 +1\right)^+ = 1$.
\end{proof}

Finally, in the following, we use Corollary \ref{azuma} and the above lemma to propose a lower bound on the sampling probability to guarantee the unique completability with high probability.

\begin{corollary}
Assume that $\sum_{i=j+1}^{d} r_i^2 \leq \Pi_{i=j+1}^{d} r_i$, $\Pi_{i=j+1}^{d}n_i \geq N_j  (\Pi_{i=j+1}^{d} r_i +1) - \sum_{i=j+1}^{d} r_i^2$, and Assumption $B_j$ hold. Furthermore, assume that we observe each entry of the  tensor $\mathcal{U}$ with probability  $p$, where
\begin{eqnarray}\label{comp4}
p >  \frac{1}{N_j} \left( 6 \ \log \left( {N_j} \right) + 2 \log \left( \max \left\{ \frac{ \sum_{i=j+1}^{d} r_i^2}{\epsilon},    \frac{\Pi_{i=j+1}^{d} r_i - \sum_{i=j+1}^{d} r_i^2}{\epsilon} ,  \frac{N_j}{\epsilon}  \right\} \right) + 8 \right) + \frac{1}{\sqrt[4]{N_j}}.
\end{eqnarray}
Then, with probability  at least $1-2\epsilon$, there exists only one completion for $\mathcal{U}$, given the rank components $r_{j+1},\dots,r_d$.
\end{corollary}

\begin{proof}
The proof is straightforward given Lemma \ref{uniqsampprob} and Corollary \ref{azuma}.
\end{proof}

\section{Numerical Comparisons}\label{simulations}

In order to show the advantage of our proposed method over matrix analysis, we compare the lower bound on the sampling probability obtained by matricization analysis with the bound obtained by tensor analysis. %The numerical comparisons are implemented on a $4^{\text{th}}$-order tensor $\mathcal{U} \in \mathbb{R}^{900 \times 900 \times 900 \times 900}$.  %All of the components of rank vector are equal and vary in a range simultaneously. 

In the first example, numerical comparisons are performed on a $4^{\text{th}}$-order tensor $\mathcal{U} \in \mathbb{R}^{900 \times 900 \times 900 \times 900}$. Figure \ref{fig1} plots the bounds given in \eqref{comp1} (Grassmannian manifold) and \eqref{comp2} (Tucker manifold) for finite completability, where $r_1= \cdots = r_4 = r$, and $\epsilon = 0.0001$. As the second example, we consider a $6^{\text{th}}$-order tensor $\mathcal{U} \in \mathbb{R}^{900 \times 900 \times 900 \times 900 \times 900 \times 900}$. Figure \ref{fig2} plots the bounds given in \eqref{comp1} (Grassmannian manifold) and \eqref{comp2} (Tucker manifold) for finite completability, where $r_3=r_4 = r_5 = r_4 = r$, and $\epsilon = 0.0001$. In this example, a significant reduction in the sampling probability is seen for the tensor model. 

\begin{figure}
	\centering
		{\includegraphics[width=11cm]{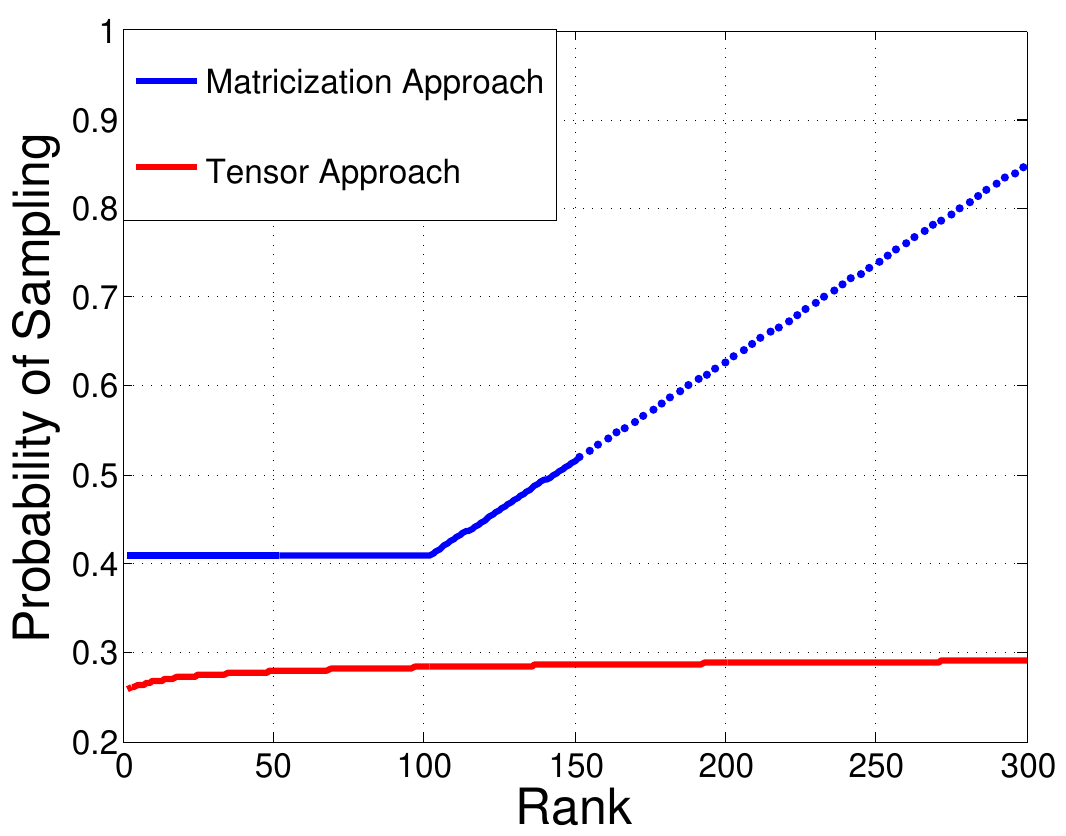}}
	\caption{ Lower bounds on sampling probability for a $4^{\text{th}}$-order tensor.}
	\label{fig1}\vspace{-4mm}
\end{figure}

\begin{figure}
	\centering	
		{\includegraphics[width=11cm]{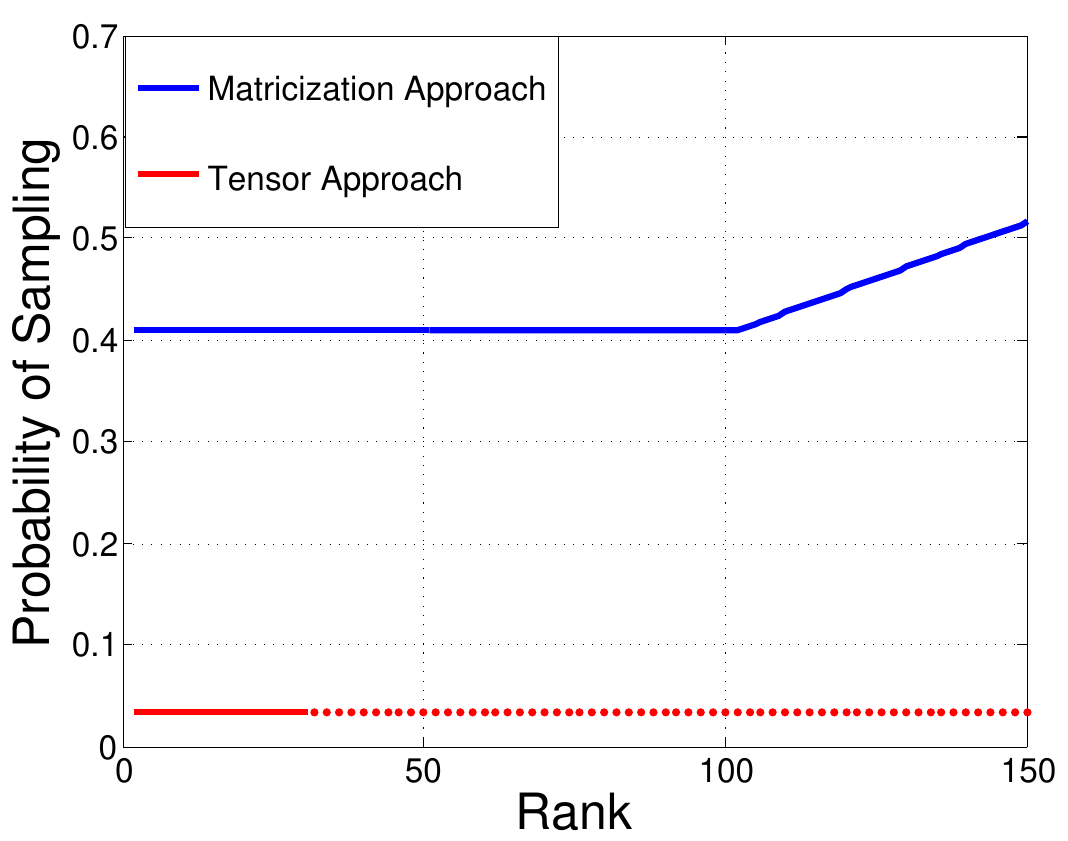}}
	\caption{  Lower bounds on sampling probability for a $6^{\text{th}}$-order tensor.}
	\label{fig2}\vspace{-4mm}
\end{figure}

We have the following observations:
\begin{itemize}
\item In Figure \ref{fig1}, the bound obtained through the analysis on Tucker manifold is less than the bound obtained through the low-order analysis on Grassmannian manifold. This improvement significantly increases as the value of the rank increases.

\item In Figure \ref{fig1}, the restriction $r \leq \frac{n}{6}$ in the analysis on Grassmannian manifold makes the bound valid only for $r \leq 150$ (and that is why the corresponding curve is dotted for $r > 150$ in Figure \ref{fig1}). On the other hand, the restriction $r  (d-1) \leq n$ in our proposed approach ensures the validity of the bound for $r \leq 300$. Similarly, in Figure \ref{fig2}, the restriction $\Pi_{i=j+1}^{d}n_i \geq N_j  \Pi_{i=j+1}^{d} r_i - \sum_{i=j+1}^{d} r_i^2$ in the analysis on Tucker manifold makes the bound valid only for $r \leq 30$.

\item Since the bounds for finite completability \eqref{comp4} and unique completability \eqref{comp2} result in almost the same curves in the above examples, we only plot bounds for finite completability in the figures. In general, the main difference between \eqref{comp2} and \eqref{comp4} is an additional term $\log\left(\frac{N_j}{\epsilon}\right)$ in the second term inside of maximum operator.
\end{itemize}

\section{Conclusions}\label{conc}

This paper aims to find fundamental conditions on the sampling pattern for finite completability of a low-rank partially sampled tensor. To solve this problem, a novel geometric approach on Tucker manifold is proposed. Specifically, a set of polynomials based on the location of the sampled entries are first defined, and then using Bernstein's theorem and analysis on Tucker manifold,  a relationship between a geometric pattern on the sampled entries and the number of algebraically independent polynomials among all of the defined polynomials is characterized. Moreover, an extension to Hall's theorem in graph theory is provided which is key to obtaining probabilistic conditions for finite and unique completabilities. Using these developed tools, we have addressed three problems in this paper: (i) Characterizing the necessary and sufficient conditions on the sampling pattern to have finitely many tensor completions for the given rank, (ii) Characterizing conditions on the sampling pattern to ensure that there is exactly one completion for the given rank, (iii) Lower bounding the sampling probability such that the conditions in Problems (i) and (ii) are satisfied with high probability.

Finally, through numerical examples it is seen that our proposed analysis on Tucker manifold for finite tensor completability leads to a much lower sampling rate than the matricization approach that is based on analysis 
on Grassmannian manifold.

%Furthermore, a converse of tensor completion problem is studied where the rank of the sampled tensor is unknown and there exists another known tensor that agrees with the sampled tensor. Under some assumptions, rank of the unknown sampled tensor can be determined uniquely with high probability.
\appendices
\section{Proof of Lemma \ref{subpro}}\label{apdx1}

As mentioned before, finiteness of the number of completions given $(r_{j+1},\dots,r_d)$ results finiteness of the number of completions given $(r_{j},\dots,r_d)$. We are interested to show this statement holds in Theorem \ref{thm5}, i.e, satisfying the properties (i) and (ii) for $j$ in the statement of Theorem \ref{thm5} results in satisfying the properties (i) and (ii) for $j-1$. Before we present this result, the following notation is introduced for the simplicity.
 
\begin{definition}
Let set $\mathcal{S}$ be a set of $d$-tuples such that each element of $\mathcal{S}$ consists of $d$ natural numbers $(x_1,\ldots,x_d)$. Define $f_{j}(\mathcal{S})$ as the number of different tuples of the first $j$ components of the elements of $\mathcal{S}$.
\end{definition}

For example, let $\mathcal{S} = \{ (2,2,1) , (2,3,3) , (2,2,2), (1,2,1),(2,3,1) \}$. Then, according to the definition $f_{2}(\mathcal{S}) = |\{ (2,2) , (2,3) , (1,2) \}| = 3$.

%\begin{proof} { \bf (Lemma \ref{subpro})}\\
Now, we are ready to provide the proof of lemma. It is easy to see that Lemma \ref{subpro} is exactly equivalent to the following statement:

Assume that we sample a tensor $\mathcal{U} \in \mathbb{R}^{n_1 \times n_2 \times \cdots \times n_d}$. Suppose that there exist $ n = \left(\Pi_{i=1}^{j} n_i\right) \left(  \Pi_{i=j+1}^{d} r_i\right)  - \left( \sum_{i=j+1}^{d}   r_i^{2} \right)$ of the sampled entries such that any subset of them (called $\mathcal{S}$) with $t$ observed entires (for any $1 \leq t \leq n$) satisfies the following inequality
\begin{eqnarray}\label{j-th}
\left(\Pi_{i=j+1}^{d} r_i \right)  f_{j}(\mathcal{S}) - g_{j+1}(f_{j}(\mathcal{S})) \geq t.
\end{eqnarray}

Then, there exist $ n^{\prime} = \left(\Pi_{i=1}^{j-1} n_i\right) \left(  \Pi_{i=j}^{d} r_i\right)  - \left( \sum_{i=j}^{d}  r_i^{2} \right)$of the sampled entries such that any subset of them (called $\mathcal{S}^{\prime}$) with $t$ observed entires (for any $1 \leq t \leq n^{\prime}$) satisfies the following inequality
\begin{eqnarray}\label{j-1th}
\left(\Pi_{i=j}^{d} r_i \right) f_{j-1}(\mathcal{S}^{\prime}) - g_{j}(f_{j-1}(\mathcal{S}^{\prime})) \geq t.
\end{eqnarray}

Therefore, it suffices to show the above statement. For $r_j=1$, the statement can be easily verified, and therefore we consider the case of  $r_j>1$ in the rest of proof. Also recall the assumption $n_j > \sum_{i=1}^{d} r_{i} $. We partition all of the sampled entries into $n_j$ groups such that group $i$ includes all of the sampled entries that $j$-th component of their coordinate is equal to $i$. Let $\mathcal{S}_i$ denote the $i$-th group. Observe that every subset of the sampled entries of each group satisfies inequality \eqref{j-th}. Moreover, according to pigeonhole principle we know that between the defined $n_j$ groups there exist $r_j$ groups (without loss of generality assume $\mathcal{S}_1,\ldots,\mathcal{S}_{r_j}$) such that $\sum_{i=1}^{r_j} | \mathcal{S}_i | \geq \frac{r_j}{n_j}  n $.

Since $\frac{r_j}{n_j} n  = \left(\Pi_{i=1}^{j-1} n_i\right) \left(  \Pi_{i=j}^{d} r_i\right)  - \frac{r_j}{n_j} \left( \sum_{i=j+1}^{d} \ r_i^{2}  \right) \geq n^{\prime} $, we only need to show that any subset of the elements of the groups $\mathcal{S}_1,\ldots,\mathcal{S}_{r_j}$ satisfies inequality \eqref{j-1th}. Consider an arbitrary subset of the elements of groups $\mathcal{S}_1,\ldots,\mathcal{S}_{r_j}$ and denote it by $\mathcal{S}^{\prime}$. Assume that $\mathcal{S}^{\prime}_i$ denotes the elements of $\mathcal{S}^{\prime}$ that belong in $\mathcal{S}_i$ for $1 \leq i \leq r_j$. Recall that every subset of the sampled entries of each group satisfies inequality \eqref{j-th}. Moreover, due to the fact that in each group the $j$-th component of the coordinates are the same, for each subset of the entries of a group $\mathcal{S}_i$ we have $f_{j}(\mathcal{S}^{\prime}_i)=f_{j-1}(\mathcal{S}^{\prime}_i)$. As a result, we have
\begin{eqnarray}
\left(\Pi_{i=j+1}^{d} r_i \right)  f_{j-1}(\mathcal{S}^{\prime}_i) - g_{j+1}(f_{j-1}(\mathcal{S}^{\prime}_i)) \geq | \mathcal{S}^{\prime}_i |.
\end{eqnarray}

Without loss of generality, assume  $f_{j}(\mathcal{S}^{\prime}_1) \geq f_{j}(\mathcal{S}^{\prime}_2) \geq \ldots \geq f_{j}(\mathcal{S}^{\prime}_{r_j})$. Also, observe that for each $1 \leq i \leq r_j$ we have $f_{j}(\mathcal{S}^{\prime}_i) \leq f_{j}(\mathcal{S}^{\prime})$. Therefore, by adding up the above inequalities for $1 \leq i \leq r_j$ we have
\begin{eqnarray}\label{fhg}
\left(\Pi_{i=j}^{d} r_i \right)  f_{j-1}(\mathcal{S}^{\prime}_1) -    g_{j+1}(f_{j-1}(\mathcal{S}^{\prime}_1)) r_j\geq | \mathcal{S}^{\prime} |.
\end{eqnarray}

For the case that $1 < r_j < \sum_{i=j+1}^{d} r_i^2$ holds, we can see that  $r_j  \left( \sum_{i=j+1}^{d} r_i^2 \right) \geq \sum_{i=j}^{d} r_i^2 $ also holds, and therefore using \eqref{fhg} we can obtain
\begin{eqnarray}
\left(\Pi_{i=j}^{d} r_i \right) f_{j-1}(\mathcal{S}^{\prime}) -  g_{j}(f_{j-1}(\mathcal{S}^{\prime})) \geq | \mathcal{S}^{\prime} |,
\end{eqnarray}
which completes the proof. If $r_j \geq \sum_{i=j+1}^{d} r_i^2$ similarly using the above inequality and by ignoring $r_j^2$ of entries proof can be completed.
%\end{proof}

\section{Proof of Theorem \ref{genHall}}\label{appB}

The proof of this theorem is based on strong induction on $x$. For $x=1$ the theorem is easy to verify. Assume that the statement of the theorem holds for $1 \leq x \leq x_0$. Then, it suffices to show that the statement also holds for the case that $x=x_0 +1$. There are two different scenarios that we need to show the statement holds for separately.

{\bf Scenario 1.} {\it{There exists a proper and nonempty subset of nodes $\mathcal{S}_1$ of $\mathcal{T}_1$ such that $|N(\mathcal{S}_1)| = |\mathcal{S}_1| + r$} }:

Consider the induced subgraph of $\mathcal{G}$ with the set of nodes $\mathcal{S}_1 \cup N(\mathcal{S}_1)$ and denote it by $\mathcal{G}_1$. Induction hypothesis results in a spanning subgraph $\mathcal{G}_1^{\prime}$ of $\mathcal{G}_1$ that satisfies the desired properties in the statement of the theorem for the corresponding subgraph $\mathcal{G}_1$.

Now, consider the induced subgraph of $\mathcal{G}$ with the set of nodes $\{ \mathcal{T}_1 \cup \mathcal{T}_2 \} \setminus \{ \mathcal{S}_1 \}$ and denote it by $\mathcal{G}_2$. Since $|N(\mathcal{S}_1)| = |\mathcal{S}_1| + r$, we conclude that for each subset of nodes of $ \{\mathcal{T}_1\}   \setminus \{\mathcal{S}_1\}$ we have $|N(\mathcal{S}) \cap \{\mathcal{T}_2   \setminus N(\mathcal{S}_1)\}| \geq |\mathcal{S}| $. The reason is that if we choose a set of nodes including the members of $\mathcal{S} \subset \{\mathcal{T}_1\}   \setminus \{\mathcal{S}_1\}$ plus all nodes in $\mathcal{S}_1$ and use the assumption in the statement of the theorem, it results that $|N(\mathcal{S}) \cap \{\mathcal{T}_2   \setminus N(\mathcal{S}_1)\}| \geq |\mathcal{S}| $. Moreover, induction hypothesis results that there exists a spanning subgraph that satisfies the desired properties in the statement of the theorem for the corresponding subgraph $\mathcal{G}_2$. Now, Lemma \ref{match} results that there exists a spanning subgraph $\mathcal{G}_2^{\prime}$ of the graph $\mathcal{G}_2$ so that every node of $\mathcal{T}_1 \setminus \mathcal{S}_1$ is of degree $r+1$, and also for each subset of nodes of $\mathcal{T}_1 \setminus \mathcal{S}_1$, the inequality \eqref{neigh1} holds and in addition, $\mathcal{G}_2^{\prime}$ includes a perfect matching between the nodes in $\mathcal{T}_1 \setminus \mathcal{S}_1$ and $\mathcal{T}_2 \setminus N(\mathcal{S}_1)$.

Now, consider a spanning subgraph $\mathcal{G}^{\prime}$ of the graph $\mathcal{G}$ that only includes all edges of $\mathcal{G}_1^{\prime}$ and $\mathcal{G}_2^{ \prime}$. We can easily observe that every node of $\mathcal{T}_1 $ is of degree $r+1$. Moreover, for each subset of nodes of $ \mathcal{S}_1$ the inequality \eqref{neigh1} holds, and also for each subset of nodes of $\mathcal{T}_1 \setminus \mathcal{S}_1$ the inequality \eqref{neigh1} holds. Now, consider a subset of nodes $\mathcal{S}$ of $\mathcal{T}_1$. It suffices to show the inequality \eqref{neigh1} holds for $\mathcal{S}$.

Define $\mathcal{S}^{\prime} = \mathcal{S} \cap \mathcal{S}_1$ and $\mathcal{S}^{\prime \prime} = \mathcal{S} \cap \{\mathcal{T}_1 \setminus \mathcal{S}_1\}$. Since for each subset of nodes of $ \mathcal{S}_1$ the inequality \eqref{neigh1} holds, we have $|N(\mathcal{S}) \cap N(\mathcal{S}_1)| \geq |\mathcal{S}^{\prime}| + r$. On the other hand, $\mathcal{G}^{\prime}$ includes a perfect matching between the nodes in $\mathcal{T}_1 \setminus \mathcal{S}_1$ and $\mathcal{T}_2 \setminus N(\mathcal{S}_1)$, and therefore $|N(\mathcal{S}) \cap \{\mathcal{T}_2 \setminus N(\mathcal{S}_1)\}| \geq |\mathcal{S}^{\prime \prime}|$. As a result, $|N(\mathcal{S})| \geq |\mathcal{S}^{\prime}| + r + |\mathcal{S}^{\prime \prime}| = |\mathcal{S}| + r$ and the proof is complete for this scenario.

{\bf Scenario 2:} {\it{For any proper and nonempty subset of nodes of $\mathcal{T}_1$ we have $|N(\mathcal{S})| \geq |\mathcal{S}| +r+1 $}}: 

Consider an arbitrary node $v_1$ and observe that $|N(v_1)| \geq r+1$. Define $\mathcal{S}_1 = \mathcal{T}_1 \setminus \{v_1\}$ and let $\mathcal{G}_1 $ denote the induced subgraph of $\mathcal{G}$ with set of nodes $\mathcal{S}_1 \cup \mathcal{T}_2$ (which is all nodes of graph $\mathcal{G}$ except for the node $v_1$). Induction hypothesis results that there exists a spanning subgraph $\mathcal{G}_1^{\prime}$ of the graph $\mathcal{G}_1$ such that every node of $\mathcal{S}_1$ is of degree $r+1$, and also for each subset of nodes of $\mathcal{S}_1$, the inequality \eqref{neigh1} holds. Now, consider a spanning subgraph $\mathcal{G}^{\prime}$ of the graph $\mathcal{G}$ that includes only all edges of $\mathcal{G}_1^{\prime}$, and also $r+1$ random edges among all edges that are connected to $v_1$. It is clear that $\mathcal{G}^{\prime}$ satisfies all of the desired properties mentioned in the statement of the theorem.
%\end{proof}

In order to complete the proof of Theorem \ref{genHall}, we need the following lemma as it was mentioned before.

\begin{lemma}\label{match}
Consider a bipartite graph $\mathcal{G}$ with the two sets of nodes, $\mathcal{T}_1$ with $x$ nodes and $\mathcal{T}_2$ with $x+r+y$ nodes, where $x, y, r \in \mathbb{N}$. Suppose that for each subset of the nodes of $\mathcal{T}_1$ we have
\begin{eqnarray}\label{neigh}
|N(\mathcal{S})|  \geq  |\mathcal{S}| + r.
\end{eqnarray}
Moreover, assume that there exists a subset of nodes $\mathcal{S}_0$ of $\mathcal{T}_2$ such that $|\mathcal{S}_0| = x$, and also for every subset of nodes $\mathcal{S}$ of $\mathcal{T}_1$ we have $|N(\mathcal{S}) \cap \mathcal{S}_0| \geq |\mathcal{S}|$. Assume that there exists a spanning subgraph $\mathcal{G}_0$ of $\mathcal{G}$ such that every node of $\mathcal{T}_1$ is of degree $r+1$, and also for each subset of nodes of $\mathcal{T}_1$, the inequality \eqref{neigh} holds. Then, there exists a spanning subgraph $\mathcal{G}^{\prime}$ of $\mathcal{G}$ such that every node of $\mathcal{T}_1$ is of degree $r+1$, and also for each subset of nodes of $\mathcal{T}_1$, the inequality \eqref{neigh} holds and in addition, $\mathcal{G}^{\prime}$ includes a perfect matching between the nodes in $\mathcal{T}_1$ and $\mathcal{S}_0$.
\end{lemma}

\begin{proof}
We prove the lemma using strong induction on $x$. For $x=1$ the lemma is easy to verify. Assume that the statement of lemma holds for $1 \leq x \leq x_0$. Then, we only need to show that lemma holds for the case that $x=x_0 +1$. There are two scenarios and we prove the statement for these two scenarios separately as follows.

{\bf Scenario 1:} {\it There exists a proper and nonempty subset of nodes $\mathcal{S}_1$ of $\mathcal{T}_1$ such that $|N(\mathcal{S}_1) \cap \mathcal{S}_0| = |\mathcal{S}_1|$}: 

For simplicity in notation, define $\mathcal{S}_1^{\prime} \triangleq \mathcal{T}_1 \setminus  \mathcal{S}_1$. In this scenario, we use the induction hypothesis for $\mathcal{S}_1$ and $\mathcal{S}_1^{\prime}$ separately (since $|\mathcal{S}_1| \leq x_0$ and $|\mathcal{S}_1^{\prime}| \leq x_0$). Define the sets $\mathcal{S}_0^{\prime} = N(\mathcal{S}_1) \cap \mathcal{S}_0$ and $\mathcal{S}_0^{\prime \prime} = \mathcal{S}_0 \setminus \mathcal{S}_0^{\prime}$. Consider the induced subgraph $\mathcal{G}_1$ of the graph $\mathcal{G}$ where the set of vertices of $\mathcal{G}_1$ is $\mathcal{T}_2 \cup \mathcal{S}_1$. Assumption of the lemma results that there exists a spanning subgraph of $\mathcal{G}_1$ such that every node of $\mathcal{S}_1$ is of degree $r+1$, and also for each subset of nodes of $\mathcal{S}_1$, the inequality \eqref{neigh} holds (by considering all the edges of the induced subgraph with vertices $\mathcal{T}_2 \cup \mathcal{S}_1$ that also exist in $\mathcal{G}_0$). Induction hypothesis results that there exists a spanning subgraph $\mathcal{G}_0^{\prime}$ of $\mathcal{G}_1$ such that every node of $\mathcal{S}_1$ is of degree $r+1$, and also for each subset of nodes of $\mathcal{S}_1$, the inequality \eqref{neigh} holds and in addition, $\mathcal{G}_0^{\prime}$ includes a perfect matching between nodes of $\mathcal{S}_1$ and nodes of $\mathcal{S}_0^{\prime}$.

Consider the induced subgraph $\mathcal{G}_1^{\prime}$ of the graph $\mathcal{G}$ where the set of vertices of $\mathcal{G}_2$ is $\mathcal{T}_2 \cup \mathcal{S}_1^{\prime}$. Again, assumption of the lemma results that there exists a spanning subgraph of $\mathcal{G}_1^{\prime}$ such that every node of $\mathcal{S}_1^{\prime}$ is of degree $r+1$, and also for each subset of nodes of $\mathcal{S}_1^{\prime}$, the inequality \eqref{neigh} holds (by considering all the edges of the induced subgraph with vertices $\mathcal{T}_2 \cup \mathcal{S}_1^{\prime}$ that also exist in $\mathcal{G}_0$). Also, observe that since $|N(\mathcal{S}_1) \cap \mathcal{S}_0| = |\mathcal{S}_1|$, for every subset of nodes $\mathcal{S}$ of $\mathcal{S}_1^{\prime}$ we have $|N(\mathcal{S}) \cap \mathcal{S}_0^{\prime \prime}| \geq |\mathcal{S}|$. As a result, induction hypothesis results that there exists a spanning subgraph $\mathcal{G}_0^{\prime \prime}$ of $\mathcal{G}_1^{\prime}$ such that every node of $\mathcal{S}_1^{\prime}$ is of degree $r+1$, and also for each subset of nodes of $\mathcal{S}_1^{\prime}$, the inequality \eqref{neigh} holds. In addition, $\mathcal{G}_0^{\prime \prime}$ includes a perfect matching between nodes of $\mathcal{S}_1^{\prime}$ and nodes of $\mathcal{S}_0^{\prime \prime}$.

Now, consider a spanning subgraph $\mathcal{G}^{\prime}$ of the graph $\mathcal{G}$ that only includes  all  edges of $\mathcal{G}_0^{\prime}$ and $\mathcal{G}_0^{\prime \prime}$. It can be verified that $\mathcal{G}^{\prime}$ satisfies all of the desired properties in the statement of the lemma and therefore the proof is complete for this case.

{\bf Scenario 2:} {\it For any proper and nonempty subset of nodes of $\mathcal{T}_1$ we have $|N(\mathcal{S}) \cap \mathcal{S}_0| \geq |\mathcal{S}| +1 $}:

In this case, consider an arbitrary vertex of $\mathcal{T}_1$ and denote it by $v_0$, and also define $\mathcal{S}_1 = \mathcal{T}_1 \setminus \{v_0 \}$. Hence, we have $|N(\mathcal{S}_1) \cap \mathcal{S}_0| \geq |\mathcal{S}_1| +1 $ and since for any $\mathcal{S}$ we have $|N(\mathcal{S}) \cap \mathcal{S}_0| \leq |\mathcal{S}_0|$ we conclude that $|N(\mathcal{S}_1) \cap \mathcal{S}_0| = |\mathcal{S}_1| +1 $. Also, according to the assumptions of the lemma we know $N(v_0) \geq r+1$ and $|N(v_0) \cap \mathcal{S}_0 | \geq 1$. We choose an arbitrary node among the nodes in $|N(v_0) \cap \mathcal{S}_0 |$ and denote it by $u_0$. Construct a graph $\mathcal{G}_0^{\prime \prime}$ with nodes $\mathcal{T}_2 \cup \{v_0\}$ and $r+1$ edges among edges that are connected to $v_0$ including the edge between $v_0$ and $u_0$, i.e., $(v_0,u_0)$.

Now, consider the induced subgraph $\mathcal{G}_1$ of the graph $\mathcal{G}$ where the set of vertices of $\mathcal{G}_1$ is $\mathcal{T}_2 \cup \mathcal{S}_1 $, i.e., all  nodes in $\mathcal{G}$ except for the node $v_0$. Assumption of the lemma results that there exists a spanning subgraph of $\mathcal{G}_1$ such that every node of $\mathcal{S}_1$ is of degree $r+1$, and also for each subset of nodes of $\mathcal{S}_1$, the inequality \eqref{neigh} holds (by considering all edges of the induced subgraph with vertices $\mathcal{T}_2 \cup \mathcal{S}_1$ that also exist in $\mathcal{G}_0$). Induction hypothesis results that there exists a spanning subgraph $\mathcal{G}_0^{\prime}$ of $\mathcal{G}_1$ such that every node of $\mathcal{S}_1$ is of degree $r+1$, and also for each subset of nodes of $\mathcal{S}_1$, the inequality \eqref{neigh} holds and in addition, $\mathcal{G}_0^{\prime}$ includes a perfect matching between nodes of $\mathcal{S}_1$ and nodes of $\mathcal{S}_0 \setminus \{u_0\}$.

Finally, consider a spanning subgraph $\mathcal{G}^{\prime}$ of the graph $\mathcal{G}$ that only includes all edges of $\mathcal{G}_0^{\prime}$ and $\mathcal{G}_0^{\prime \prime}$. The constructed graph $\mathcal{G}^{\prime}$ satisfies all of the mentioned properties in the lemma.
\end{proof}

\section{Proof of Claim in Example \ref{lemexfinuni}}\label{appC}

In the following we show that given the rank-$2$ sampled matrix in Example \ref{lemexfinuni}, there are exactly two completions. Note that the following decomposition always holds for some values of $r_1,r_2,\dots,r_6$ and $x_1,x_2,\dots,x_8$:

\begin{center}
\begin{tabular}{ |c|c|c|c| } 
 \hline
$1$ & \ \ \ \ & $1$ & $-\frac{1}{2}$ \\ \hline
$-4$ & $2$ & $-1$ &  \\ \hline
$0$ & $1$ &  & $2$ \\ \hline
$1$ &  & $4$ &  \\ \hline
 & $4$ & $-2$ & $\frac{3}{2}$ \\
 \hline
\end{tabular}
 \  \ \ $=$ \ \ \ 
\begin{tabular}{ |c|c| } 
 \hline
$1$ &  $0$ \\ \hline
$0$ & $1$ \\ \hline
$r_1$ & $r_2$ \\ \hline
$r_3$ & $r_4$  \\ \hline
$r_5$ & $r_6$ \\
 \hline
\end{tabular}
 \  \ \ $\times$ \ \ \ 
 \begin{tabular}{ |c|c|c|c| } 
 \hline
$x_1$ & $x_2$ & $x_3$ & $x_4$ \\ \hline
$x_5$ & $x_6$ & $x_7$ & $x_8$ \\ 
 \hline
\end{tabular}
\end{center}

Note that the $2 \times 2$ identity matrix in the above decomposition represents canonical basis defined in Definition \ref{canonical1}. The first two rows of $\mathbf{U}$ in the above decomposition result the following:
\begin{subequations}
\begin{eqnarray}
1 &=& x_1, \\
1 &=& x_3, \\
-\frac{1}{2} &=& x_4, \\
-4 &=& x_5, \\
2 &=& x_6, \\
-1 &=& x_7. 
\end{eqnarray}
\end{subequations}
Then, the decomposition can be simplified as
\begin{center}
\begin{tabular}{ |c|c|c|c| } 
 \hline
$0$ & $1$ &  & $2$ \\ \hline
$1$ & \ \ \ \ & $4$ & \ \ \ \  \\ \hline
 \ \ \ \ & $4$ & $-2$ & $\frac{3}{2}$ \\
 \hline
\end{tabular}
 \  \ \ $=$ \ \ \ 
\begin{tabular}{ |c|c| } 
 \hline
$r_1$ & $r_2$ \\ \hline
$r_3$ & $r_4$  \\ \hline
$r_5$ & $r_6$ \\
 \hline
\end{tabular}
 \  \ \ $\times$ \ \ \ 
 \begin{tabular}{ |c|c|c|c| } 
 \hline
$1$ & $x_2$  & $1$ & $-\frac{1}{2}$ \\ \hline
$-4$ & $2$ & $-1$ & $x_8$ \\ 
 \hline
\end{tabular}
\end{center}
Therefore, we have following system of equations
\begin{subequations}
\begin{eqnarray}
0 &=& r_1 - 4r_2, \label{eqex1} \\
1 &=& x_2r_1 + 2r_2, \label{eqex2} \\
2 &=& -\frac{1}{2}r_1 + x_8r_2, \label{eqex3} \\
1 &=& r_3 - 4r_4, \label{eqex4} \\
4 &=& r_3 - r_4, \label{eqex5} \\
4 &=& x_2r_5 + 2r_6, \label{eqex6} \\
-2 &=& r_5 - r_6, \label{eqex7} \\
\frac{3}{2} &=& -\frac{1}{2}r_5 + x_8r_6. \label{eqex8} 
\end{eqnarray}
\end{subequations}
Observe that $r_3=5$ and $r_4=1$ can be determined uniquely from \eqref{eqex4} and \eqref{eqex5}. %Then, we have:
%\begin{subequations}
%\begin{eqnarray}
%0 &=& r_1 - 4r_2, \label{eqex1x} \\
%1 &=& x_2r_1 + 2r_2, \label{eqex2x} \\
%2 &=& -\frac{1}{2}r_1 + x_8r_2, \label{eqex3x} \\
%4 &=& x_2r_5 + 2r_6, \label{eqex6x} \\
%-2 &=& r_5 - r_6, \label{eqex7x} \\
%\frac{3}{2} &=& -\frac{1}{2}r_5 + x_8r_6. \label{eqex8x} 
%\end{eqnarray}
%\end{subequations}
 Note that $r_1 = 4r_2 $ and $r_5 = (r_6-2) $ can be concluded from \eqref{eqex1} and \eqref{eqex7}, respectively. Then, by substituting $r_1 \leftarrow 4r_2 $ in \eqref{eqex2} and \eqref{eqex3} and substituting $r_5 \leftarrow (r_6-2) $ in \eqref{eqex6} and \eqref{eqex8} we obtain:
\begin{subequations}
\begin{eqnarray}
1 &=& 4x_2r_2 + 2r_2, \label{eqex2xx} \\
2 &=& -2r_2 + x_8r_2, \label{eqex3xx} \\
4 &=& x_2(r_6-2) + 2r_6, \label{eqex6xx} \\
\frac{3}{2} &=& -\frac{1}{2}(r_6-2) + x_8r_6. \label{eqex8xx} 
\end{eqnarray}
\end{subequations}
Similarly, $r_2 = \frac{2}{x_8-2} $ and $r_6 = \frac{1}{2x_8-1} $ can be concluded from \eqref{eqex3xx} and \eqref{eqex8xx}, respectively. Hence, substituting $r_2 \leftarrow \frac{2}{x_8-2} $ , $r_6 \leftarrow \frac{1}{2x_8-1} $ in \eqref{eqex2xx} and \eqref{eqex6xx} results in the following system of equations
\begin{subequations}
\begin{eqnarray}
x_8  &=& 8x_2 + 6, \label{eqex2xxx} \\
8x_8 -4 &=& x_2 - 2x_2(2x_8-1) + 2. \label{eqex6xxx} 
\end{eqnarray}
\end{subequations}
Finally, \eqref{eqex2xxx} results $x_8 = 8x_2 + 6 $, and therefore substituting $x_8 \leftarrow 8x_2 + 6 $ in \eqref{eqex6xxx} results
\begin{eqnarray}
8(8x_2 + 6) -4 &=& x_2 - 2x_2(2(8x_2 + 6)-1) + 2, \label{eqex6xxxx} 
\end{eqnarray}
which can be simplified as
\begin{eqnarray}
32x_2^2 + 85x_2 + 42 &=& 0. \label{eqex6xxxxx} 
\end{eqnarray}
Therefore, $x_2 \in \{-2,-\frac{21}{32}\}$. Given $x_2$, all other variables can be determined uniquely recursively as above, and the following completions of $\mathbf{U}$ can be obtained as the only possible rank-$2$ completions of $\mathbf{U}$.
\begin{center}
\begin{tabular}{ |c|c|c|c| } 
 \hline
$1$ & $-\frac{21}{32}$ & $1$ & $-\frac{1}{2}$ \\ \hline
$-4$ & $2$ & $-1$ & $\frac{3}{4}$ \\ \hline
$0$ & $1$ & -$\frac{24}{5}$ & $2$ \\ \hline
$1$ & $-\frac{41}{32}$ & $4$ & $-\frac{7}{4}$ \\ \hline
$-8$ & $4$ & $-2$ & $\frac{3}{2}$ \\
 \hline
\end{tabular}
 \  \ \ $=$ \ \ \ 
\begin{tabular}{ |c|c| } 
 \hline
$1$ &  $0$ \\ \hline
$0$ & $1$ \\ \hline
$-\frac{32}{5}$ & $-\frac{8}{5}$ \\ \hline
$5$ & $1$  \\ \hline
$0$ & $2$ \\
 \hline
\end{tabular}
 \  \ \ $\times$ \ \ \ 
 \begin{tabular}{ |c|c|c|c| } 
 \hline
$1$ & $-\frac{21}{32}$ & $1$ & $-\frac{1}{2}$ \\ \hline
$-4$ & $2$ & $-1$ & $\frac{3}{4}$ \\ 
 \hline
\end{tabular}
\end{center}
and
\begin{center}
\begin{tabular}{ |c|c|c|c| } 
 \hline
$1$ & $-2$ & $1$ & $-\frac{1}{2}$ \\ \hline
$-4$ & $2$ & $-1$ & $-10$ \\ \hline
$0$ & $1$ & -$\frac{1}{2}$ & $2$ \\ \hline
$1$ & $-8$ & $4$ & $-\frac{25}{2}$ \\ \hline
$-\frac{39}{21}$ & $4$ & $-2$ & $\frac{3}{2}$ \\
 \hline
\end{tabular}
 \  \ \ $=$ \ \ \ 
\begin{tabular}{ |c|c| } 
 \hline
$1$ &  $0$ \\ \hline
$0$ & $1$ \\ \hline
$-\frac{2}{3}$ & $-\frac{1}{6}$ \\ \hline
$5$ & $1$  \\ \hline
$-\frac{43}{21}$ & $-\frac{1}{21}$ \\
 \hline
\end{tabular}
 \  \ \ $\times$ \ \ \ 
 \begin{tabular}{ |c|c|c|c| } 
 \hline
$1$ & $-2$ & $1$ & $-\frac{1}{2}$ \\ \hline
$-4$ & $2$ & $-1$ & $-10$ \\ 
 \hline
\end{tabular}
\end{center}

\bibliographystyle{IEEETran}
\bibliography{bib}

\begin{thebibliography}{10}
\providecommand{\url}[1]{#1}
\csname url@samestyle\endcsname
\providecommand{\newblock}{\relax}
\providecommand{\bibinfo}[2]{#2}
\providecommand{\BIBentrySTDinterwordspacing}{\spaceskip=0pt\relax}
\providecommand{\BIBentryALTinterwordstretchfactor}{4}
\providecommand{\BIBentryALTinterwordspacing}{\spaceskip=\fontdimen2\font plus
\BIBentryALTinterwordstretchfactor\fontdimen3\font minus
  \fontdimen4\font\relax}
\providecommand{\BIBforeignlanguage}[2]{{%
\expandafter\ifx\csname l@#1\endcsname\relax
\typeout{** WARNING: IEEEtran.bst: No hyphenation pattern has been}%
\typeout{** loaded for the language `#1'. Using the pattern for}%
\typeout{** the default language instead.}%
\else
\language=\csname l@#1\endcsname
\fi
#2}}
\providecommand{\BIBdecl}{\relax}
\BIBdecl

\bibitem{phase}
E.~J. Cand{\`e}s, Y.~C. Eldar, T.~Strohmer, and V.~Voroninski, ``Phase
  retrieval via matrix completion,'' \emph{SIAM review}, vol.~57, no.~2, pp.
  225--251, 2015.

\bibitem{Image}
H.~Ji, C.~Liu, Z.~Shen, and Y.~Xu, ``Robust video denoising using low rank
  matrix completion.'' in \emph{CVPR}.\hskip 1em plus 0.5em minus 0.4em\relax
  Citeseer, 2010, pp. 1791--1798.

\bibitem{data}
L.~Eld{\'e}n, \emph{Matrix methods in data mining and pattern
  recognition}.\hskip 1em plus 0.5em minus 0.4em\relax SIAM, 2007, vol.~4.

\bibitem{network}
N.~J. Harvey, D.~R. Karger, and K.~Murota, ``Deterministic network coding by
  matrix completion,'' in \emph{Proceedings of the sixteenth annual ACM-SIAM
  symposium on Discrete algorithms}.\hskip 1em plus 0.5em minus 0.4em\relax
  Society for Industrial and Applied Mathematics, 2005, pp. 489--498.

\bibitem{lim}
L.-H. Lim and P.~Comon, ``Multiarray signal processing: Tensor decomposition
  meets compressed sensing,'' \emph{Comptes Rendus Mecanique}, vol. 338, no.~6,
  pp. 311--320, 2010.

\bibitem{sid}
N.~D. Sidiropoulos and A.~Kyrillidis, ``Multi-way compressed sensing for sparse
  low-rank tensors,'' \emph{IEEE Signal Processing Letters}, vol.~19, no.~11,
  pp. 757--760, 2012.

\bibitem{gandy}
S.~Gandy, B.~Recht, and I.~Yamada, ``Tensor completion and low-n-rank tensor
  recovery via convex optimization,'' \emph{Inverse Problems}, vol.~27, no.~2,
  pp. 1--19, 2011.

\bibitem{visual}
J.~Liu, P.~Musialski, P.~Wonka, and J.~Ye, ``Tensor completion for estimating
  missing values in visual data,'' \emph{IEEE Transactions on Pattern Analysis
  and Machine Intelligence}, vol.~35, no.~1, pp. 208--220, 2013.

\bibitem{liu2016low}
X.-Y. Liu, S.~Aeron, V.~Aggarwal, and X.~Wang, ``Low-tubal-rank tensor
  completion using alternating minimization,'' in \emph{SPIE Defense+
  Security}.\hskip 1em plus 0.5em minus 0.4em\relax International Society for
  Optics and Photonics, 2016, pp. 984\,809--984\,809.

\bibitem{kreimer}
N.~Kreimer, A.~Stanton, and M.~D. Sacchi, ``Tensor completion based on nuclear
  norm minimization for {5D} seismic data reconstruction,'' \emph{Geophysics},
  vol.~78, no.~6, pp. V273--V284, 2013.

\bibitem{ely20135d}
G.~Ely, S.~Aeron, N.~Hao, M.~E. Kilmer \emph{et~al.}, ``5d and 4d pre-stack
  seismic data completion using tensor nuclear norm (tnn),'' in \emph{2013 SEG
  Annual Meeting}.\hskip 1em plus 0.5em minus 0.4em\relax Society of
  Exploration Geophysicists, 2013.

\bibitem{wang2016tensor}
W.~Wang, V.~Aggarwal, and S.~Aeron, ``Tensor completion by alternating
  minimization under the tensor train {(TT)} model,'' \emph{arXiv
  preprint:1609.05587}, Sep. 2016.

\bibitem{liu2016tensor}
X.-Y. Liu, S.~Aeron, V.~Aggarwal, X.~Wang, and M.-Y. Wu, ``Tensor completion
  via adaptive sampling of tensor fibers: Application to efficient indoor {RF}
  fingerprinting,'' in \emph{IEEE International Conference on Acoustics, Speech
  and Signal Processing (ICASSP)}, 2016, pp. 2529--2533.

\bibitem{7347424}
X.~Y. Liu, S.~Aeron, V.~Aggarwal, X.~Wang, and M.~Y. Wu, ``Adaptive sampling of
  {RF} fingerprints for fine-grained indoor localization,'' \emph{IEEE
  Transactions on Mobile Computing}, vol.~15, no.~10, pp. 2411--2423, Oct.
  2016.

\bibitem{vaneetcnsm}
V.~Aggarwal, A.~A. Mahimkar, H.~Ma, Z.~Zhang, S.~Aeron, and W.~Willinger,
  ``Inferring smartphone service quality using tensor methods,'' in \emph{12th
  International Conference on Network and Service Management}, Oct.--Nov. 2016.

\bibitem{jain2013low}
P.~Jain, P.~Netrapalli, and S.~Sanghavi, ``Low-rank matrix completion using
  alternating minimization,'' in \emph{The forty-fifth annual symposium on
  Theory of computing}.\hskip 1em plus 0.5em minus 0.4em\relax ACM, 2013, pp.
  665--674.

\bibitem{candes}
E.~J. Cand{\`e}s and B.~Recht, ``Exact matrix completion via convex
  optimization,'' \emph{Foundations of Computational Mathematics}, vol.~9,
  no.~6, pp. 717--772, 2009.

\bibitem{candes2}
E.~J. Cand{\`e}s and T.~Tao, ``The power of convex relaxation: Near-optimal
  matrix completion,'' \emph{IEEE Transactions on Information Theory}, vol.~56,
  no.~5, pp. 2053--2080, 2010.

\bibitem{cai}
J.~F. Cai, E.~J. Cand{\`e}s, and Z.~Shen, ``A singular value thresholding
  algorithm for matrix completion,'' \emph{SIAM Journal on Optimization},
  vol.~20, no.~4, pp. 1956--1982, 2010.

\bibitem{ashraph19low}
M.~Ashraphijuo, X.~Wang, and J.~Zhang, ``Low-rank data completion with very low
  sampling rate using {N}ewton's method,'' \emph{IEEE Transactions on Signal
  Processing}, vol.~67, no.~7, pp. 1849--1859, 2019.

\bibitem{hav}
T.~S.~C. HAVE, ``A number of approaches to subspace clustering have been
  proposed in the past two decades.'' 2011.

\bibitem{parsons}
L.~Parsons, E.~Haque, and H.~Liu, ``Subspace clustering for high dimensional
  data: a review,'' \emph{ACM SIGKDD Explorations Newsletter}, vol.~6, no.~1,
  pp. 90--105, 2004.

\bibitem{pimentelgroup}
D.~Pimentel-Alarc{\'o}n, L.~Balzano, R.~Marcia, R.~Nowak, and R.~Willett,
  ``Group-sparse subspace clustering with missing data.''

\bibitem{pimentel2014sample}
D.~Pimentel-Alarc{\'o}n, R.~Nowak, and L.~Balzano, ``On the sample complexity
  of subspace clustering with missing data,'' in \emph{IEEE Workshop on
  Statistical Signal Processing (SSP)}, 2014, pp. 280--283.

\bibitem{tomioka}
R.~Tomioka, K.~Hayashi, and H.~Kashima, ``Estimation of low-rank tensors via
  convex optimization,'' \emph{arXiv preprint:1010.0789}, Oct 2010.

\bibitem{nuctensor}
M.~Signoretto, Q.~T. Dinh, L.~De~Lathauwer, and J.~A. Suykens, ``Learning with
  tensors: a framework based on convex optimization and spectral
  regularization,'' \emph{Machine Learning}, vol.~94, no.~3, pp. 303--351,
  2014.

\bibitem{romera}
B.~Romera-Paredes and M.~Pontil, ``A new convex relaxation for tensor
  completion,'' in \emph{Advances in Neural Information Processing Systems},
  2013, pp. 2967--2975.

\bibitem{ashraphijuoc}
M.~Ashraphijuo, R.~Madani, and J.~Lavaei, ``Characterization of
  rank-constrained feasibility problems via a finite number of convex
  programs,'' in \emph{IEEE 55th Conference on Decision and Control (CDC)},
  2016, pp. 6544--6550.

\bibitem{liulow2}
X.-Y. Liu, S.~Aeron, V.~Aggarwal, and X.~Wang, ``Low-tubal-rank tensor
  completion using alternating minimization,'' \emph{arXiv
  preprint:1610.01690}, Oct. 2016.

\bibitem{low}
D.~Kressner, M.~Steinlechner, and B.~Vandereycken, ``Low-rank tensor completion
  by riemannian optimization,'' \emph{BIT Numerical Mathematics}, vol.~54,
  no.~2, pp. 447--468, 2014.

\bibitem{low2}
A.~Krishnamurthy and A.~Singh, ``Low-rank matrix and tensor completion via
  adaptive sampling,'' in \emph{Advances in Neural Information Processing
  Systems}, 2013, pp. 836--844.

\bibitem{goldfarb}
D.~Goldfarb and Z.~Qin, ``Robust low-rank tensor recovery: Models and
  algorithms,'' \emph{SIAM Journal on Matrix Analysis and Applications},
  vol.~35, no.~1, pp. 225--253, 2014.

\bibitem{relern}
B.~Recht, ``A simpler approach to matrix completion,'' \emph{Journal of Machine
  Learning Research}, vol.~12, no. Dec, pp. 3413--3430, 2011.

\bibitem{charact}
D.~Pimentel-Alarc{\'o}n, N.~Boston, and R.~Nowak, ``A characterization of
  deterministic sampling patterns for low-rank matrix completion,'' \emph{IEEE
  Journal of Selected Topics in Signal Processing}, vol.~10, no.~4, pp.
  623--636, 2016.

\bibitem{ashraphijuo3}
M.~Ashraphijuo, V.~Aggarwal, and X.~Wang, ``On deterministic sampling patterns
  for robust low-rank matrix completion,'' \emph{IEEE Signal Processing
  Letters}, vol.~25, no.~3, pp. 343--347, 2018.

\bibitem{kiraly2}
F.~Kir{\'a}ly, L.~Theran, and R.~Tomioka, ``The algebraic combinatorial
  approach for low-rank matrix completion,'' \emph{Machine Learning}, vol.~16,
  pp. 1391--1436, 2015.

\bibitem{ashrphrization}
M.~Ashraphijuo and X.~Wang, ``A characterization of sampling patterns for union
  of low-rank subspaces retrieval problem.'' in \emph{International Symposium
  on Artificial Intelligence and Mathematics}, 2018, pp. 1--8.

\bibitem{bahigh}
L.~Balzano, B.~Eriksson, and R.~Nowak, ``High rank matrix completion and
  subspace clustering with missing data,'' in \emph{The conference on
  Artificial Intelligence and Statistics (AIStats)}, 2012.

\bibitem{ashraphijuo5}
M.~Ashraphijuo, X.~Wang, and V.~Aggarwal, ``Rank determination for low-rank
  data completion,'' \emph{Journal of Machine Learning Research}, vol. 18 (98),
  pp. 1--29, 2017.

\bibitem{pianmmple}
D.~Pimentel, R.~Nowak, and L.~Balzano, ``On the sample complexity of subspace
  clustering with missing data,'' in \emph{IEEE Workshop on Statistical Signal
  Processing (SSP)}.\hskip 1em plus 0.5em minus 0.4em\relax IEEE, 2014, pp.
  280--283.

\bibitem{ashraphijuoustering}
M.~Ashraphijuo and X.~Wang, ``Clustering a union of low-rank subspaces of
  different dimensions with missing data,'' \emph{Pattern Recognition Letters},
  vol. 120, pp. 31--35, 2019.

\bibitem{pimentelc4}
D.~Pimentel-Alarc{\'o}n, E.~R.~D. Nowak, and W.~EDU, ``The
  information-theoretic requirements of subspace clustering with missing
  data,'' in \emph{International Conference on Machine Learning}, 2016.

\bibitem{yarse}
C.~Yang, D.~Robinson, and R.~Vidal, ``Sparse subspace clustering with missing
  entries,'' in \emph{International Conference on Machine Learning (ICML)},
  2015, pp. 2463--2472.

\bibitem{ashrximation}
M.~Ashraphijuo, X.~Wang, and V.~Aggarwal, ``An approximation of the cp-rank of
  a partially sampled tensor,'' in \emph{Annual Allerton Conference on
  Communication, Control, and Computing (Allerton)}.\hskip 1em plus 0.5em minus
  0.4em\relax IEEE, 2017, pp. 604--611.

\bibitem{pimentelc3}
D.~Pimentel-Alarc{\'o}n, L.~Balzano, R.~Marcia, R.~Nowak, and R.~Willett,
  ``Group-sparse subspace clustering with missing data,'' in \emph{IEEE
  Workshop on Statistical Signal Processing (SSP)}, 2016, pp. 1--5.

\bibitem{Tuck}
T.~G. Kolda, ``Orthogonal tensor decompositions,'' \emph{SIAM Journal on Matrix
  Analysis and Applications}, vol.~23, no.~1, pp. 243--255, 2001.

\bibitem{SVD}
L.~Grasedyck, ``Hierarchical singular value decomposition of tensors,''
  \emph{SIAM Journal on Matrix Analysis and Applications}, vol.~31, no.~4, pp.
  2029--2054, 2010.

\bibitem{Tuckermanifold}
D.~Kressner, M.~Steinlechner, and B.~Vandereycken, ``Low-rank tensor completion
  by riemannian optimization,'' \emph{BIT Numerical Mathematics}, vol.~54,
  no.~2, pp. 447--468, 2014.

\bibitem{ten}
J.~Berge and N.~Sidiropoulos, ``On uniqueness in candecomp/parafac,''
  \emph{Psychometrika}, vol.~67, no.~3, pp. 399--409, 2002.

\bibitem{kruskal}
A.~Stegeman and N.~D. Sidiropoulos, ``On kruskal’s uniqueness condition for
  the candecomp/parafac decomposition,'' \emph{Linear Algebra and its
  applications}, vol. 420, no.~2, pp. 540--552, 2007.

\bibitem{oseledets}
I.~V. Oseledets, ``Tensor-train decomposition,'' \emph{SIAM Journal on
  Scientific Computing}, vol.~33, no.~5, pp. 2295--2317, 2011.

\bibitem{holtz}
S.~Holtz, T.~Rohwedder, and R.~Schneider, ``On manifolds of tensors of fixed
  tt-rank,'' \emph{Numerische Mathematik}, vol. 120, no.~4, pp. 701--731, 2012.

\bibitem{backll}
J.~Ballani, L.~Grasedyck, and M.~Kluge, ``Black box approximation of tensors in
  hierarchical {T}ucker format,'' \emph{Linear Algebra and its Applications},
  vol. 438, no.~2, pp. 639--657, 2013.

\bibitem{hack9new}
W.~Hackbusch and S.~K{\"u}hn, ``A new scheme for the tensor representation,''
  \emph{Journal of Fourier Analysis and Applications}, vol.~15, no.~5, pp.
  706--722, 2009.

\bibitem{kilmer2013third}
M.~E. Kilmer, K.~Braman, N.~Hao, and R.~C. Hoover, ``Third-order tensors as
  operators on matrices: A theoretical and computational framework with
  applications in imaging,'' \emph{SIAM Journal on Matrix Analysis and
  Applications}, vol.~34, no.~1, pp. 148--172, 2013.

\bibitem{Bernstein}
B.~Sturmfels, \emph{Solving systems of polynomial equations}.\hskip 1em plus
  0.5em minus 0.4em\relax American Mathematical Soc., 2002, no.~97.

\bibitem{schaeffer1941inequalities}
A.~Schaeffer, ``Inequalities of a. markoff and s. bernstein for polynomials and
  related functions,'' \emph{Bulletin of the American Mathematical Society},
  vol.~47, no.~8, pp. 565--579, 1941.

\bibitem{ashraphijuo4}
M.~Ashraphijuo and X.~Wang, ``Fundamental conditions for low-{CP}-rank tensor
  completion,'' \emph{Journal of Machine Learning Research}, vol. 18 (63), pp.
  1--29, 2017.

\bibitem{janson2002concentration}
S.~Janson, ``On concentration of probability,'' \emph{Contemporary
  combinatorics}, vol.~10, no.~3, pp. 1--9, 2002.

\bibitem{kierstead1983effective}
H.~A. Kierstead, ``An effective version of hall’s theorem,''
  \emph{Proceedings of the American Mathematical Society}, vol.~88, no.~1, pp.
  124--128, 1983.

\end{thebibliography}

\end{document}